\documentclass[twoside]{IEEEtran}
\usepackage{amsmath,amsfonts,amssymb,amsthm}
\usepackage{url}

\newtheorem{assumption}{Assumption}[section]

\newtheorem{theorem}{Theorem}[section]

\newtheorem{lemma}{Lemma}[section]

\newtheorem{corollary}{Corollary}[section]

\newcounter{appendixcounter}

\newtheorem{lemmaappendix}{Lemma}[appendixcounter]

\newtheorem*{remark}{Remark}%[section]

\begin{document}

\title{Asymptotic Properties of Recursive Maximum Likelihood Estimation in Non-Linear State-Space Models\thanks
{A short version of the paper has been presented at the 2019 IEEE International Symposium on
Information Theory.} }

\author{Vladislav Z. B. Tadi\'{c} and
Arnaud Doucet
\thanks{
Vladislav Z. B. Tadi\'{c} is with School of Mathematics, University of Bristol,
Bristol BS8 1TW, United Kingdom
(e-mail: v.b.tadic@bristol.ac.uk).

Arnaud Doucet is with Department of Statistics,
University of Oxford, Oxford OX1 3LB, United Kingdom
(e-mail: doucet@stats.ox.ac.uk).

}}

\date{}

\maketitle

\begin{abstract}
Using stochastic gradient search and the optimal filter derivative,
it is possible to perform recursive maximum likelihood estimation
in a non-linear state-space model.
As the optimal filter and its derivative are analytically intractable for such a model,
they need to be approximated numerically.
In \cite{poyiadjis&doucet&singh},
a recursive maximum likelihood algorithm based on a particle approximation
to the optimal filter derivative has been proposed and studied through numerical simulations.
This algorithm and its asymptotic behavior are here analyzed theoretically.
Under regularity conditions,
we show that the algorithm accurately estimates maxima of the underlying log-likelihood rate 
when the number of particles is sufficiently large.
We also provide qualitative upper bounds on the estimation error in terms of the number of particles.
\end{abstract}


\section{Introduction}

State-space models (also known as continuous-state hidden Markov models)
are a class of stochastic processes capable of modeling
complex time-series data and stochastic dynamical systems.
These models can be viewed as a discrete-time Markov process which can be observed only
through noisy measurements of its states.

In many applications, a state-space model depends on a parameter whose value needs to be
estimated given a set of state-observations.
Due to its practical and theoretical importance,
parameter estimation in state-space and hidden Markov models
has been extensively studied in the engineering and statistics literature
(see e.g. \cite{cappe&moulines&ryden}, \cite{douc&moulines&stoffer} and references cited therein).
Among them, the methods based on maximum likelihood principle
have gained much attention.
Their asymptotic properties (convergence, convergence rate and asymptotic normality)
have been analyzed thoroughly in a number of papers 
(see e.g. \cite{bickel&ritov&ryden},
\cite{douc&matias}, \cite{douc&moulines&ryden}, \cite{leroux},
\cite{mevel&finesso}, \cite{ryden1} -- \cite{tadic2}).
Unfortunately, to the best of our knowledge,
the existing results do not offer much information about recursive (i.e., online)
maximum likelihood estimation in non-linear state-space models.
However, in a number of different scenarios,
the parameter indexing a state-space model needs to be estimated recursively.
For example, this is much more computationally efficient for long observation sequences.
In the maximum likelihood approach,
this can be achieved using stochastic gradient search and the optimal filter derivative.
Since the optimal filter and its derivative are not analytically tractable for a non-linear state-space model,
they need to be approximated numerically.
In \cite{poyiadjis&doucet&singh},
a recursive maximum likelihood algorithm based on a particle approximation to the optimal filter derivative
has been proposed and it has been shown experimentally that the algorithm is stable and efficient.
We show here that the algorithm proposed in \cite{poyiadjis&doucet&singh}
produces asymptotically accurate estimates of maxima to the underlying log-likelihood rate.
More specifically, we show that these estimates converge almost surely
to a close vicinity of stationary points of the underlying log-likelihood.
We also provide qualitative upper bounds on the radius of this vicinity.
These bounds are expressed in terms of the number of particles used to approximate
the filter and its derivative and directly characterize
the (asymptotic) error of the recursive particle maximum likelihood algorithm
proposed in \cite{poyiadjis&doucet&singh}.
The obtained results hold under strong mixing assumptions
which are very commonly used in the particle filtering literature
(see e.g. \cite{cappe&moulines&ryden}, \cite{crisan&rozovskii}, \cite{delmoral&doucet&singh},
\cite{douc&moulines&stoffer}).
To the best of our knowledge, the results presented here are the first
to offer a rigorous analysis of recursive maximum likelihood estimation in non-linear state-space models.

The rest of this paper is organized as follows.
In Section \ref{section1}, non-linear state-space models and
the corresponding recursive maximum likelihood algorithm are specified.
In the same section, the main results of the paper are presented.
In Section \ref{section2}, a non-trivial example illustrating the main results is provided.
The main results are proved in Sections \ref{section1*} -- \ref{section3*}.

\section{Main Results} \label{section1}

\subsection{State-Space Models and Parameter Estimation}\label{ssection1.1}

To define state-space models,
we use the following notation.
For a set ${\cal Z}$ in a metric space,
${\cal B}({\cal Z})$ denotes the collection of Borel subsets of ${\cal Z}$.
$d_{x}\geq 1$ and $d_{y}\geq 1$ are integers,
while ${\cal X}\in{\cal B}(\mathbb{R}^{d_{x} })$ and ${\cal Y}\in{\cal B}(\mathbb{R}^{d_{y} })$.
$P(x,dx')$ is a transition kernel on ${\cal X}$,
while $Q(x,dy)$ is a conditional probability measure on ${\cal Y}$ given $x\in{\cal X}$.
$(\Omega,{\cal F}, P)$ is a probability space.
A state-space model can be defined as
the ${\cal X}\times{\cal Y}$-valued stochastic process
$\{ (X_{n}, Y_{n} ) \}_{n\geq 0}$
on $(\Omega,{\cal F}, P)$ which satisfies
\begin{align*}
	&
	P\left( (X_{n+1}, Y_{n+1} )\in B
	|X_{0:n}, Y_{0:n} \right)
	\\
	&
	=
	\iint I_{B}(x,y) Q(x,dy) P(X_{n}, dx )
\end{align*}
almost surely for each $B\in{\cal B}({\cal X}\times{\cal Y})$, $n\geq 0$.
$\{X_{n} \}_{n\geq 0}$ are the unobservable model states,
while $\{Y_{n} \}_{n\geq 0}$ are the observations.
States $\{X_{n} \}_{n\geq 0}$ form a Markov chain,
while $P(x,dx')$ is their transition kernel.
The observations $\{Y_{n} \}_{n\geq 0}$ are mutually independent
conditionally on $\{X_{n} \}_{n\geq 0}$,
while $Q(X_{n}, dy )$ is the conditional distribution of $Y_{n}$ given $X_{0:n}$.

In this paper, we assume that the model $\{ (X_{n}, Y_{n} ) \}_{n\geq 0}$,
can be accurately approximated by a parametric family of state-space models.
To define such a family,
we rely on the following notation.
$d\geq 1$ is an integer, while $\Theta\subset\mathbb{R}^{d}$ is an open set.
${\cal P}({\cal X})$ is the set of probability measures on ${\cal X}$.
$\mu(dx)$ and $\nu(dy)$ are measures on ${\cal X}$ and ${\cal Y}$ (respectively).
$p_{\theta}(x'|x)$ and $q_{\theta}(y|x)$
are functions which map
$\theta\in\Theta$, $x,x'\in{\cal X}$, $y\in{\cal Y}$
to $[0,\infty )$ and satisfy
\begin{align*}
	\int_{\cal X} p_{\theta}(x'|x) \mu(dx')
	=
	\int_{\cal Y} q_{\theta}(y|x) \nu(dy)
	=
	1
\end{align*}
for all $\theta\in\Theta$, $x\in{\cal X}$.
A parametric family of state-space models can then be defined as a collection of
${\cal X}\times{\cal Y}$-valued stochastic processes
$\left\{ (X_{n}^{\theta,\lambda}, Y_{n}^{\theta,\lambda} ) \right\}_{n\geq 0}$
on $(\Omega, {\cal F}, P)$
which are parameterized by $\theta\in\Theta$, $\lambda\in{\cal P}({\cal X})$ and satisfy
\begin{align*}
	&
	P\left( (X_{0}^{\theta,\lambda}, Y_{0}^{\theta,\lambda} ) \in B \right)
	=
	\iint I_{B}(x,y) q_{\theta}(y|x) \lambda(dx),
	\\
	&
	P\left(\left. (X_{n+1}^{\theta,\lambda}, Y_{n+1}^{\theta,\lambda} ) \in B\right|
	X_{0:n}^{\theta,\lambda}, Y_{0:n}^{\theta,\lambda} \right)
	\\
	&
	=
	\iint I_{B}(x,y)
	q_{\theta}(y|x) p_{\theta}(x|X_{n}^{\theta,\lambda} ) \mu(dx) \nu(dy)
\end{align*}
almost surely for each $B\in{\cal B}({\cal X}\times{\cal Y})$, $n\geq 0$.

We are interested in the identification of model parameters.
This problem can be formulated as the estimation of the transition kernel $P(x,dx')$
and the conditional probability $Q(x,dy)$
given a realization of state-observations $\{Y_{n} \}_{n\geq 0}$.
If the identification is based on the recursive maximum likelihood approach
and the parametric model
$\left\{ ( X_{n}^{\theta,\lambda}, Y_{n}^{\theta,\lambda} ) \right\}_{n\geq 0}$,
the estimation of $P(x,dx')$ and $Q(x,dy)$ reduces to the maximization of
the log-likelihood rate associated with models
$\{(X_{n}, Y_{n} ) \}_{n\geq 0}$ and
$\left\{ ( X_{n}^{\theta,\lambda}, Y_{n}^{\theta,\lambda} ) \right\}_{n\geq 0}$.
Here, $\{(X_{n}, Y_{n} ) \}_{n\geq 0}$ is considered as the true system,
while the parametric model
$\left\{ ( X_{n}^{\theta,\lambda}, Y_{n}^{\theta,\lambda} ) \right\}_{n\geq 0}$
is regarded as the candidate model.

To define the log-likelihood rate associated with models
$\{(X_{n}, Y_{n} ) \}_{n\geq 0}$ and
$\left\{ ( X_{n}^{\theta,\lambda}, Y_{n}^{\theta,\lambda} ) \right\}_{n\geq 0}$,
we use the following notation. $q_{\theta}^{n}(y_{1:n}|\lambda )$ is the density of $Y_{1:n}^{\theta,\lambda}$,
i.e.,
\begin{align*}
	q_{\theta}^{n}(y_{1:n}|\lambda )
	=
	\int\cdots\iint
	&
	\left(
	\prod_{k=1}^{n} \big( q_{\theta}(y_{k}|x_{k}) p_{\theta}(x_{k}|x_{k-1}) \big)
	\right)
	\\
	&\cdot 
	\mu(dx_{n})\cdots\mu(dx_{1})\lambda(dx_{0})
\end{align*}
for $\theta\in\Theta$, $\lambda\in{\cal P}({\cal X})$,
$y_{1:n}=(y_{1},\dots,y_{n} )\in{\cal Y}^{n}$, $n\geq 1$.
$l_{n}(\theta,\lambda)$ is the expected (average) log-likelihood of
$Y_{1:n}$ given model $\left\{ ( X_{n}^{\theta,\lambda}, Y_{n}^{\theta,\lambda} ) \right\}_{n\geq 0}$,
i.e.,
\begin{align*}
	l_{n}(\theta,\lambda)
	=
	E\left( \frac{1}{n} \log q_{\theta}^{n}(Y_{1:n}|\lambda ) \right).
\end{align*}
Then, the log-likelihood rate for models
$\{(X_{n}, Y_{n} ) \}_{n\geq 0}$ and
$\left\{ ( X_{n}^{\theta,\lambda}, Y_{n}^{\theta,\lambda} ) \right\}_{n\geq 0}$
can be defined as the limit
$\lim_{n\rightarrow\infty} l_{n}(\theta,\lambda)$.
Under the assumptions adopted in this paper,
$\lim_{n\rightarrow\infty} l_{n}(\theta,\lambda)$ exists and does not depend on $\lambda$
(see Lemmas \ref{lemma1.3} and \ref{lemma3.3}).
Throughout this paper, $l(\theta)$ denotes the log-likelihood rate for models
$\{(X_{n}, Y_{n} ) \}_{n\geq 0}$ and
$\left\{ ( X_{n}^{\theta,\lambda}, Y_{n}^{\theta,\lambda} ) \right\}_{n\geq 0}$,
i.e.,
\begin{align*}
	l(\theta)
	=
	\lim_{n\rightarrow\infty} l_{n}(\theta,\lambda).
\end{align*}

\subsection{Recursive Maximum Likelihood Algorithm}

Recursive maximum likelihood estimation in state-space models
can be described as an online process maximizing the log-likelihood rate
$l(\theta)$.
As $l(\theta)$ and its gradient do not admit closed-form expressions for any non-linear state-space model,
they need to be approximated numerically.
We analyze here the recursive maximum likelihood algorithm proposed in
\cite{poyiadjis&doucet&singh}.
In this algorithm, $\nabla l(\theta)$ is approximated by a particle method,
while $l(\theta)$ is maximized by stochastic gradient search.

The recursive particle maximum likelihood algorithm proposed in
\cite[Sections 3.2, 2.2, Equations (26), (27), (18) -- (22)]{poyiadjis&doucet&singh} is defined by the following equations: 
\begin{align}
	\label{1.1}
	&\begin{aligned}[b]
	W_{n+1,i}
	=&
	\frac{\sum_{j=1}^{N}
	p_{\theta_{n}}(\hat{X}_{n+1,i}|\hat{X}_{n,j} )
	\nabla_{\theta} q_{\theta_{n}}(Y_{n}|\hat{X}_{n,j} ) }
	{\sum_{j=1}^{N} p_{\theta_{n}}(\hat{X}_{n+1,i}|\hat{X}_{n,j} ) q_{\theta_{n}}(Y_{n}|\hat{X}_{n,j} ) }
	\\
	&+
	\frac{\sum_{j=1}^{N}
	\nabla_{\theta} p_{\theta_{n}}(\hat{X}_{n+1,i}|\hat{X}_{n,j} )
	q_{\theta_{n}}(Y_{n}|\hat{X}_{n,j} ) }
	{\sum_{j=1}^{N} p_{\theta_{n}}(\hat{X}_{n+1,i}|\hat{X}_{n,j} ) q_{\theta_{n}}(Y_{n}|\hat{X}_{n,j} ) }
	\\
	&+
	\frac{\sum_{j=1}^{N} p_{\theta_{n}}(\hat{X}_{n+1,i}|\hat{X}_{n,j} )
	q_{\theta_{n}}(Y_{n}|\hat{X}_{n,j} ) W_{n,j} }
	{\sum_{j=1}^{N} p_{\theta_{n}}(\hat{X}_{n+1,i}|\hat{X}_{n,j} ) q_{\theta_{n}}(Y_{n}|\hat{X}_{n,j} ) },
	\end{aligned}
	\\
	&\begin{aligned}[b]
	\hat{W}_{n+1,i}
	=&
	W_{n+1,i}-\frac{1}{N}\sum_{j=1}^{N} W_{n+1,j},
	\end{aligned}
	\\
	\label{1.3}
	&\begin{aligned}[b]
	\theta_{n+1}
	=
	\theta_{n}
	+
	\alpha_{n}
	&
	\left(
	\frac{\sum_{j=1}^{N}
	q_{\theta_{n}}(Y_{n+1}|\hat{X}_{n+1,j} )
	\hat{W}_{n+1,j} }
	{\sum_{j=1}^{N} q_{\theta_{n}}(Y_{n+1}|\hat{X}_{n+1,j}) }
	\right. 
	\\
	&+
	\left.  
	\frac{\sum_{j=1}^{N}
	\nabla_{\theta} q_{\theta_{n}}(Y_{n+1}|\hat{X}_{n+1,j} ) }
	{\sum_{j=1}^{N} q_{\theta_{n}}(Y_{n+1}|\hat{X}_{n+1,j}) }
	\right)
	\end{aligned}
\end{align}
for $n\geq 0$, $1\leq i\leq N$. Here, $N\geq 1$ is an integer corresponding to the number of particles and $\{\alpha_{n} \}_{n\geq 0}$ is a sequence of positive real numbers.
$\big\{ \hat{X}_{n+1,i}: 1\leq i\leq N \big\}$ are the particles generated
through the sequential Monte Carlo scheme
\begin{align}\label{1.5}
	\hat{X}_{n+1,i}
	\sim
	\frac{\sum_{j=1}^{N} p_{\theta_{n}}(x|\hat{X}_{n,j}) q_{\theta_{n}}(Y_{n}|\hat{X}_{n,j})
	\mu(dx) }
	{\sum_{j=1}^{N} q_{\theta_{n}}(Y_{n}|\hat{X}_{n,j}) }.
\end{align}
In (\ref{1.5}), $\big\{ \hat{X}_{n+1,i}: 1\leq i\leq N \big\}$ are sampled independently
one from another and independently from
$\big\{ X_{k}: 0\leq k\leq n \big\}$,
$\big\{ \theta_{k},Y_{k},\hat{X}_{k,i}: 0\leq k< n, 1\leq i\leq N \big\}$.
Moreover, in (\ref{1.1}) -- (\ref{1.5}),
$\theta_{0}\in\Theta$, $\{\hat{X}_{0,i}: 1\leq i\leq N \}\subset{\cal X}$
and $\{W_{0,i}: 1\leq i\leq N \}\subset\mathbb{R}^{d\times N}$ are
selected independently from $(X_{0},Y_{0} )$.

\begin{remark}
Recursion (\ref{1.3}) usually involves a device which keeps $\{\theta_{n} \}_{n\geq 0}$
within a compact subset of $\Theta$.
This device is usually based on the projection to a compact domain
(for details, see \cite[Section 5.4]{borkar}, 
\cite[Sections 5.1, 5.2]{kushner&yin} 
and references cited therein).
As algorithm (\ref{1.1}) -- (\ref{1.5}) is already a very complex procedure,
this aspect is not considered here.
Instead, similarly
as in \cite[Part II]{benveniste}, \cite{borkar}, \cite{kushner&yin}, \cite{tadic2},
our results on the asymptotic
behavior of the algorithm (Theorem \ref{theorem2}, below) are expressed
in a local form.
\end{remark}

The variables appearing in algorithm (\ref{1.1}) -- (\ref{1.5})
have the following meaning.
$\theta_{0}$, $\hat{X}_{0,1},\dots,\hat{X}_{0,N}$ and $W_{0,1},\dots,W_{0,N}$
are the initial conditions.
$\hat{X}_{n,1},\dots,\hat{X}_{n,N}$ are particles whose empirical distribution approximates the predictive distribution of $X_n$ given $Y_1,...,Y_{n-1}$ and its derivative
(computed using parameter $\theta_{n}$ at time $n$),
while
$W_{n,1},\dots,W_{n,N}$ are vector-valued weights in the particle approximation
to this derivative.\footnote
{The empirical measures
\begin{align*}
	\frac{1}{N} \sum_{i=1}^{N} {\delta}_{\hat{X}_{n,i}}(dx),
	\;\;\;\;\;
	\frac{1}{N}
	\sum_{i=1}^{N} \left(W_{n,i} - \frac{1}{N} \sum_{j=1}^{N} W_{n,j} \right)
	{\delta}_{\hat{X}_{n,i}}(dx)
\end{align*}
can be viewed as particle approximations (respectively) to
the optimal (one-step) predictor and its gradient at discrete-time $n$.
Here and throughout the paper, ${\delta}_{x}(dx')$ denotes the Dirac measure centered at $x\in{\cal X}$.
}
$\theta_{n}$ is an estimate of maxima to the log-likelihood rate $l(\theta)$.
$\alpha_{n}$ is the step-size in recursion (\ref{1.1}).
Recursion (\ref{1.3}) is a stochastic gradient search maximizing $l(\theta)$.\footnote
{The fraction on the right-hand side of (\ref{1.3}) is a Monte Carlo
estimate of $l(\theta)$.}
Recursions (\ref{1.1}) and (\ref{1.5}) are procedures
through which the particle approximations
to the optimal filter and its derivative are updated.
More details on the recursive particle maximum likelihood algorithm can be found in \cite{poyiadjis&doucet&singh}.

\subsection{Convergence Results}
To formulate the assumptions under which the analysis of the recursive particle maximum likelihood estimation procedure is carried out,
we introduce further notation.
$\mathbb{N}_{0}$ is the set of non-negative integers,
while $\mathbb{C}^{d}$ is the set $d$-dimensional complex valued vectors.
For $\boldsymbol\alpha=(a_{1},\dots,a_{d})\in\mathbb{N}_{0}^{d}$,
$\theta=(t_{1},\dots,t_{d})\in\mathbb{R}^{d}$,
notation $|\boldsymbol\alpha|$ and $\partial_{\theta}^{\boldsymbol\alpha}$ stand for
\begin{align*}
	|\boldsymbol\alpha|
	=
	a_{1} + \dots + a_{d},
	\;\;\;\;\;
	\partial_{\theta}^{\boldsymbol\alpha}
	=
	\frac{\partial_{\theta}^{|\boldsymbol\alpha|} }{\partial t_{1}^{a_{1}} \cdots t_{d}^{a_{d}} }.
\end{align*}
For $\eta\in\mathbb{C}^{d}$, $\|\eta\|$ denotes the Euclidean norm of $\eta$.
For $\delta\in(0,\infty)$, $\eta\in\mathbb{C}^{d}$ and $A\subset\mathbb{C}^{d}$,
$V_{\delta}(A)$, $d(\eta,A)$ denote the $\delta$-vicinity of $A$
and the distance between $\eta$, $A$,
i.e.,
\begin{align*}
	&
	d(\eta,A)=
	\inf_{\eta'\in A} \|\eta-\eta'\|,
	\\
	&
	V_{\delta}(A)
	=
	\{\eta'\in\mathbb{C}^{d}: \exists\eta''\in A, \|\eta'-\eta''\|<\delta \}.
\end{align*}

Let $Q$ be any compact set satisfying $Q\subset\Theta$.
The asymptotic properties of algorithm (\ref{1.1}) -- (\ref{1.5})
are analyzed under the following assumptions.

\begin{assumption}\label{a1}
$\sum_{n=0}^{\infty}\alpha_{n}=\infty$,
$\sum_{n=0}^{\infty}\alpha_{n}^{2}<\infty$
and
$\sum_{n=0}^{\infty} |\alpha_{n}-\alpha_{n+1}|<\infty$.
\end{assumption}

\begin{assumption}\label{a2}
There exist a probability measure $\pi(dx)$
and real numbers $\rho\in(0,1)$, $K\in[1,\infty)$
such that
\begin{align*}
	|P^{n}(x,B)-\pi(B)|\leq K\rho^{n}
\end{align*}
for all $x\in{\cal X}$, $B\in{\cal B}({\cal X})$, $n\geq 0$.
\end{assumption}

\begin{assumption}\label{a3}
There exists a real number $\varepsilon_{Q}\in(0,1)$ such that
\begin{align*}
	\varepsilon_{Q} \leq p_{\theta}(x'|x) \leq \frac{1}{\varepsilon_{Q} },
	\;\;\;\;\;
	\varepsilon_{Q} \leq q_{\theta}(y|x) \leq \frac{1}{\varepsilon_{Q} }
\end{align*}
for all $\theta\in Q$, $x,x'\in{\cal X}$, $y\in{\cal Y}$.
\end{assumption}

\begin{assumption}\label{a4}
There exists a real number $K_{1,Q}\in[1,\infty )$ such that
\begin{align*}
	&
	\max\{\|\nabla_{\theta}p_{\theta}(x'|x) \|, \|\nabla_{\theta}q_{\theta}(y|x) \| \}
	\leq
	K_{1,Q},
	\\
	&
	\max\{|p_{\theta}(x'|x)-p_{\theta'}(x'|x) |,
	\|\nabla_{\theta}p_{\theta}(x'|x) - \nabla_{\theta}p_{\theta'}(x'|x) \|
	\}
	\\
	&\leq
	K_{1,Q}\|\theta-\theta'\|,
	\\
	&
	\max\{|q_{\theta}(y|x)-q_{\theta'}(y|x) |,
	\|\nabla_{\theta}q_{\theta}(y|x) - \nabla_{\theta}q_{\theta'}(y|x) \| \}
	\\
	&\leq
	K_{1,Q}\|\theta-\theta'\|
\end{align*}
for all $\theta,\theta'\in Q$, $x,x'\in{\cal X}$, $y\in{\cal Y}$.
\end{assumption}

\begin{assumption}\label{a5}
$p_{\theta}(x'|x)$ and $q_{\theta}(y|x)$ are $p$-times differentiable in $\theta$
for each $\theta\in\Theta$, $x,x'\in{\cal X}$, $y\in{\cal Y}$,
where $p>d$.
Moreover, there exists a real number $K_{2,Q}\in[1,\infty )$ such that
\begin{align*}
	|\partial_{\theta}^{\boldsymbol\alpha} p_{\theta}(x'|x) | \leq K_{2,Q},
	\;\;\;\;\;
	|\partial_{\theta}^{\boldsymbol\alpha} q_{\theta}(y|x) | \leq K_{2,Q}
\end{align*}
for all $\theta\in Q$, $x,x'\in{\cal X}$, $y\in{\cal Y}$,
$\boldsymbol\alpha\in \mathbb{N}_{0}^{d}$
satisfying $|\boldsymbol\alpha|\leq p$.
\end{assumption}

\begin{assumption}\label{a6}
$p_{\theta}(x'|x)$ and $q_{\theta}(y|x)$ are real-analytic in $\theta$
for each $\theta\in\Theta$, $x,x'\in{\cal X}$, $y\in{\cal Y}$.
Moreover, $p_{\theta}(x'|x)$ and $q_{\theta}(y|x)$
have (complex-valued) continuations
$\hat{p}_{\eta}(x'|x)$ and $\hat{q}_{\eta}(y|x)$ with the following properties:

(i) $\hat{p}_{\eta}(x'|x)$ and $\hat{q}_{\eta}(y|x)$ map $\eta\in\mathbb{C}^{d}$,
$x,x'\in{\cal X}$, $y\in{\cal Y}$ to $\mathbb{C}$.

(ii) $\hat{p}_{\theta}(x'|x)=p_{\theta}(x'|x)$ and
$\hat{q}_{\theta}(y|x)=q_{\theta}(y|x)$ for all $\theta\in\Theta$,
$x,x'\in{\cal X}$, $y\in{\cal Y}$.

(iii) There exists a real number $\delta_{Q}\in(0,1)$ such that
$\hat{p}_{\eta}(x'|x)$ and $\hat{q}_{\eta}(y|x)$ are analytic in $\eta$
for each $\eta\in V_{\delta_{Q}}(Q)$, $x,x'\in{\cal X}$, $y\in{\cal Y}$.

(iv) There exists a real number $K_{3,Q}\in(0,1)$ such that
\begin{align*}
	|\hat{p}_{\eta}(x'|x) | \leq K_{3,Q},
	\;\;\;\;\;
	|\hat{q}_{\eta}(y|x) | \leq K_{3,Q}
\end{align*}
for all $\eta\in V_{\delta_{Q}}(Q)$, $x,x'\in{\cal X}$, $y\in{\cal Y}$.
\end{assumption}

Assumption \ref{a1} corresponds to the step-size sequence $\{\alpha_{n} \}_{n\geq 0}$
and its asymptotic properties. This assumption is standard in any asymptotic analysis
of stochastic gradient search and stochastic approximation
(see e.g., \cite{benveniste}, \cite{borkar}, \cite{kushner&yin}).
It holds when $\alpha_{n}=1/n^{a}$ for $n\geq 1$,
where $a\in(1/2,1]$.

Assumption \ref{a2} is related to the stability of the true system $\{ (X_{n},Y_{n}) \}_{n\geq 0}$.
It requires $\{X_{n} \}_{n\geq 0}$ to be uniformly ergodic. Assumption \ref{a3} implies the stability of the optimal filter for
$\left\{ (X_{n}^{\theta,\lambda}, Y_{n}^{\theta,\lambda} ) \right\}_{n\geq 0}$ 
(i.e., it ensures that the optimal filter forgets its initial condition exponentially fast).
Assumption \ref{a2} and \ref{a3} are restrictive from the theoretical point of view
as they implicitly require the state and observation spaces ${\cal X}$ and ${\cal Y}$ to be bounded. However, these assumptions are very commonly used in the literature 
(see, e.g., \cite{cappe&moulines&ryden}, \cite{crisan&rozovskii}, \cite{douc&moulines&stoffer}).

Assumptions \ref{a4} -- \ref{a6} are related to the parameterization of the candidate models
$\left\{ (X_{n}^{\theta,\lambda}, Y_{n}^{\theta,\lambda} ) \right\}_{n\geq 0}$
and its analytical properties 
(i.e. to the analytical properties of conditional densities $p_{\theta}(x'|x)$ and $q_{\theta}(y|x)$).
The purpose of Assumption \ref{a4} is to ensure that the Poisson equation associated with
algorithm (\ref{1.1}) -- (\ref{1.5}) has a locally Lipschitz solution
(see Lemma \ref{lemma2.4}).
Assumption \ref{a4} also ensures that the log-likelihood rate $l(\theta)$
is Lipschitz continuously differentiable (see Lemmas \ref{lemma1.3}, \ref{lemma3.3}).
This Poisson equation plays a crucial role in the analysis of the asymptotic error
in the Monte Carlo estimation of $\nabla l(\theta)$
(see Lemma \ref{lemma3.2} and (\ref{3*.1})),
while the Lipschitz continuity of $\nabla l(\theta)$ allows us
to analyze algorithm (\ref{1.1}) -- (\ref{1.5}) using the results on
Lipschitz gradient flows
(see Theorem \ref{theorem2}, Part (i)).
The purpose of Assumption \ref{a5} is to provide for $l(\theta)$ to be at least $(d+1)$-times
differentiable (see Lemma \ref{lemma3.3}, Part (ii)),
while Assumption \ref{a6} ensures for $l(\theta)$ to be real-analytic
(see Lemma \ref{lemma3.3}, Part (iii)).
These analytical properties of $l(\theta)$ allows us to establish
qualitative upper bounds on the asymptotic error in the estimation of maxima to $l(\theta)$
(see Theorem \ref{theorem2}, Parts (ii), (iii)).

In order to state the main results of the paper, we need further notation.
${\cal S}$ and $l({\cal S})$ are the sets of stationary points and critical values of $l(\theta)$
(respectively),
i.e.,
\begin{align}\label{1.51}
	{\cal S}
	=
	\{\theta\in\Theta: \nabla l(\theta)=0 \},
	\;\;\;\;\;
	l({\cal S})
	=
	\{l(\theta): \theta\in{\cal S} \}.
\end{align}
$\pi:\mathbb{R}\times\Theta\rightarrow\Theta$ is the solution to the ODE
$d\theta/dt=\nabla l(\theta)$ which satisfies the initial condition $\pi(0,\theta)=\theta$
for $\theta\in\Theta$. 
${\cal R}$ is the set of chain-recurrent points of the ODE $d\theta/dt=\nabla l(\theta)$,
i.e., $\theta\in{\cal R}$ if and only if
for any $\delta,t\in(0,\infty)$, there exist an integer $n\geq 1$,
real numbers $t_{1},\dots,t_{n}\in[t,\infty)$ and vectors $\vartheta_{1},\dots,\vartheta_{n}\in\Theta$
(each of which can depend on $\theta$, $\delta$, $t$) such that
$\|\vartheta_{1} - \theta \|\leq\delta$, 
$\|\pi(t_{n},\vartheta_{n}) -\theta \|\leq\delta$ and  
\begin{align*}
	\|\vartheta_{k+1} - \pi(t_{k},\vartheta_{k} ) \|\leq\delta
\end{align*}
for $1\leq k< n$.

\begin{remark}
Chain-recurrent points ${\cal R}$ can be interpreted as limit points of slightly perturbed solutions to
the ODE $d\theta/dt=\nabla l(\theta)$.
Since the piecewise linear interpolation of sequence $\{\theta_{n} \}_{n\geq 0}$
is such a solution
(see (\ref{3*.1}) and Lemma \ref{lemma3.2}; see also \cite[Section 5]{tadic&doucet21}),
the chain-recurrence is closely related to the asymptotic behavior
of algorithm (\ref{1.1}) -- (\ref{1.5}).
Regarding stationary and chain-recurrent points,
the following relationship can be established.
If $l(\theta)$ is Lipschitz continuously differentiable,
then all stationary points ${\cal S}$ are chain-recurrent for
the ODE $d\theta/dt=\nabla l(\theta)$
(i.e., ${\cal S}\subseteq{\cal R}$).
If additionally $l({\cal S})$ is of a zero Lebesgue measure
(which holds when $l(\theta )$ is $d$-times continuously differentiable),
then all chain-recurrent points ${\cal R}$ are stationary for
the ODE $d\theta/dt=\nabla l(\theta)$
(i.e., ${\cal S}={\cal R}$).
However, if $l(\theta)$ is only Lipschitz continuously differentiable,
then ${\cal S}={\cal R}$ does not necessarily hold and ${\cal R}\setminus{\cal S}\neq\emptyset$ is quite possible
(for details, see \cite[Section 4]{hurley}).
For more details on chain-recurrence, see
\cite{benaim2}, \cite{benaim3}, \cite{borkar}.
\end{remark}

Let $Q$ be any compact set satisfying $Q\subset\Theta$, while
$\Lambda_{Q}$ is the event defined by
\begin{align}\label{1.301}
	\Lambda_{Q}
	=
	\liminf_{n\rightarrow\infty} \{\theta_{n}\in Q\}
	=
	\bigcup_{n=0}^{\infty} \bigcap_{k=n}^{\infty} \{\theta_{k}\in Q \}.
\end{align}
Then, the main results of this paper are summarized in the next theorem.

\begin{theorem}\label{theorem2}
(i) If Assumptions  \ref{a1} -- \ref{a4} hold, then there exists a non-decreasing function 
$\psi_{Q}:[0,\infty)\rightarrow[0,\infty)$
depending only on $l(\theta)$, $p_{\theta}(x'|x)$, $q_{\theta}(y|x)$)
such that $\lim_{t\rightarrow 0} \psi_{Q}(t) = \psi_{Q}(0) = 0$ and
\begin{align*}
	\limsup_{n\rightarrow\infty } d(\theta_{n}, {\cal R} )
	\leq
	\psi_{Q}\left(\frac{1}{N} \right)
\end{align*}
almost surely on $\Lambda_{Q}$.

(ii) If Assumptions \ref{a1} -- \ref{a5} hold,
then there exists a real number $L_{1,Q}\in[1,\infty)$
(independent of $N$ and depending only on $l(\theta)$, $p_{\theta}(x'|x)$, $q_{\theta}(y|x)$) such that
\begin{align*}
	&
	\limsup_{n\rightarrow\infty}
	\|\nabla l(\theta_{n} ) \|
	\leq
	\frac{L_{1,Q}}{N^{q/2} },
	\\
	&
	\limsup_{n\rightarrow\infty} l(\theta_{n} ) - \liminf_{n\rightarrow\infty} l(\theta_{n} )
	\leq
	\frac{L_{1,Q}}{N^{q} }
\end{align*}
almost surely on $\Lambda_{Q}$, where $q=(p-d)/(p-1)$.

(iii) If Assumptions  \ref{a1} -- \ref{a4} and \ref{a6} hold,
then there exist real numbers $r_{Q}\in(0,1)$, $L_{2,Q}\in[1,\infty)$
(independent of $N$ and depending only on $l(\theta)$, $p_{\theta}(x'|x)$, $q_{\theta}(y|x)$) such that
\begin{align*}
	&
	\limsup_{n\rightarrow\infty} d(\theta_{n}, {\cal S} )
	\leq
	\frac{L_{2,Q}}{N^{r_{Q} } },
	\\
	&
	\limsup_{n\rightarrow\infty}
	\|\nabla l(\theta_{n} ) \|
	\leq
	\frac{L_{2,Q}}{N^{1/2} },
	\\
	&
	\limsup_{n\rightarrow\infty} d(l(\theta_{n} ), l({\cal S} ) )
	\leq
	\frac{L_{2,Q}}{N}
\end{align*}
almost surely on $\Lambda_{Q}$.
\end{theorem}

Theorem \ref{theorem2} is proved in Section \ref{section3*}.

\begin{remark}
The function $\psi_{Q}(t)$ and the real numbers $L_{1,Q}$, $L_{2,Q}$
depend on $p_{\theta}(x'|x)$, $q_{\theta}(y|x)$ through constants
$\varepsilon_{Q}$, $K_{1,Q}$ (specified in Assumptions \ref{a3}, \ref{a4}).
$\psi_{Q}(t)$ also depends on $l(\theta)$ through
a Lipschitz constant of $\nabla l(\theta)$,
an upper bound of $\|\nabla l(\theta) \|$ and the geometric properties of ${\cal R}$.
$L_{1,Q}$, $L_{2,Q}$ depend on $l(\theta)$ through
a Lipschitz constant of $\nabla l(\theta)$ and
an upper bound of $\|\nabla l(\theta) \|$.
Additionally, $L_{1,Q}$, $L_{2,Q}$ also depend on $l(\theta)$ through
the Yomdin and Lojasiewicz constants for $l(\theta)$.\footnote
{If $l(\theta)$ is $p$-times differentiable and $p>d$
(which is true under Assumptions \ref{a2} -- \ref{a5}; see Part (ii) of Lemma \ref{lemma3.3}),
there exists a real number $M_{1,Q}\in[1,\infty)$ such that
\begin{align*}
	m\left(\{l(\theta): \theta\in Q, \|\nabla l(\theta) \|\leq \varepsilon \} \right)
	\leq
	M_{1,Q}\varepsilon^{q}
\end{align*}
for all $\varepsilon\in[1,\infty)$,
where $m(\cdot)$ is the Lebesgue measure on $\mathbb{R}^{d}$
($q$ is specified in the statement of Theorem \ref{theorem2}).
This result is known as the Yomdin theorem (see \cite[Theorem 1.2]{yomdin}),
while $M_{1,Q}$ is referred to as the Yomdin exponent.

If $l(\theta)$ is real-analytic on $\Theta$
(which holds under Assumptions \ref{a2} -- \ref{a4}, \ref{a6}; see Part (iii) of Lemma \ref{lemma3.3}),
there exist real numbers $r_{Q}\in(0,1)$, $M_{2,Q}, M_{3,Q}\in[1,\infty)$ such that
\begin{align*}
	d(\theta,{\cal S} )
	\leq
	M_{2,Q} \|\nabla l(\theta) \|^{r_{Q} },
	\;\;\;\;\;
	d(l(\theta), l({\cal S} ) )
	\leq
	M_{3,Q} \|\nabla l(\theta ) \|
\end{align*}
for all $\theta\in Q$.
These inequalities are known as the Lojasiewicz inequalities
(see \cite[Theorem 6.4, Remark 6.5]{bierstone&milman},
\cite[Theorem \L I, Page 775]{kurdyka}).
$M_{2,Q}$, $M_{3,Q}$ are referred to as the Lojasiewicz constants,
while $r_{Q}$ is called the Lojasiewicz exponent.
}
$r_{Q}$ is the Lojasiewicz exponent for $l(\theta)$.
For further details on how $\psi_{Q}(t)$, $r_{Q}$, $L_{1,Q}$, $L_{2,Q}$
depend on $l(\theta)$, $p_{\theta}(x'|x)$, $q_{\theta}(y|x)$,
see \cite{tadic&doucet21}.
\end{remark}

As algorithm (\ref{1.1}) -- (\ref{1.5}) is a stochastic gradient search
maximizing the log-likelihood rate $l(\theta)$,
the asymptotic properties of sequences
$\{\theta_{n} \}_{n\geq 0}$,
$\{ l(\theta_{n} ) \}_{n\geq 0}$ and
$\{ \nabla l(\theta_{n} ) \}_{n\geq 0}$
provide a natural way to characterize the asymptotic behavior of this algorithm.
If the estimation of $\nabla l(\theta)$ in algorithm (\ref{1.1}) -- (\ref{1.5})
were based on the exact optimal filter
instead of a particle approximation,
the corresponding estimator would be asymptotically consistent.
Then, according to the existing results on stochastic optimization,
sequences $\{\theta_{n} \}_{n\geq 0}$, $\{l(\theta_{n} ) \}_{n\geq 0}$ and
$\{\nabla l(\theta_{n} ) \}_{n\geq 0}$ would exhibit the following behavior.
If $\nabla l(\theta)$ was estimated using the exact optimal filter
and $l(\theta)$ was Lipschitz continuously differentiable,
then the limits
\begin{align}
	&\label{1.7a}
	\lim_{n\rightarrow\infty} d(\theta_{n}, {\cal R} ) = 0,
	\\
	&\label{1.7b}
	\lim_{n\rightarrow\infty} d(l(\theta_{n}), l({\cal R}) ) =0,
	\\
	&\label{1.7c}
	\liminf_{n\rightarrow\infty} d(\theta_{n}, {\cal S} ) = 0
\end{align}
would hold almost surely on the event
$\{\sup_{n\geq 0} \|\theta_{n}\|<\infty,$ $\inf_{n\geq 0} d(\theta_{n}, \Theta^{c} ) > 0 \}$ 
(see, e.g., \cite[Proposition 4.1, Theorem 5.7]{benaim2}).
If $l(\theta)$ was additionally $(d+1)$-times differentiable,
then the limits
\begin{align}
	\label{1.9a}
	\lim_{n\rightarrow\infty} d(\theta_{n}, {\cal S} ) \!=\! 0,
	&\;\;\;
	\lim_{n\rightarrow\infty} \nabla l(\theta_{n} ) \!=\! 0,
	\\
	\label{1.9b}
	\lim_{n\rightarrow\infty} d(l(\theta_{n}), l({\cal S}) ) \!=\! 0,
	&\;\;\:
	\limsup_{n\rightarrow\infty} l(\theta_{n})
	\!=\!
	\liminf_{n\rightarrow\infty} l(\theta_{n})
\end{align}
would hold almost surely on
$\{\sup_{n\geq 0} \|\theta_{n}\|<\infty,$ $\inf_{n\geq 0} d(\theta_{n}, \Theta^{c} ) > 0 \}$
(see e.g., \cite[Corollary 6.7]{benaim2}).
Since algorithm (\ref{1.1}) -- (\ref{1.5}) estimates $\nabla l(\theta)$
using a particle approximation,
the corresponding estimator is biased.
Consequently, the limits (\ref{1.7a}) -- (\ref{1.9b}) do not hold for algorithm (\ref{1.1}) -- (\ref{1.5}).
Instead, the following limits
\begin{align}
	\label{1.11.a}
	\limsup_{n\rightarrow\infty} d(\theta_{n}, {\cal R} ),
	&\;\;\;
	\limsup_{n\rightarrow\infty} \|\nabla l(\theta_{n} ) \|,
	\\
	\label{1.11.b}
	\limsup_{n\rightarrow\infty} d(l(\theta_{n}), l({\cal R}) ),
	&\;\;\; 
	\limsup_{n\rightarrow\infty} l(\theta_{n})
	-
	\liminf_{n\rightarrow\infty} l(\theta_{n})
\end{align}
take strictly positive values.
These limits directly depend on the accuracy of the particle
approximations to the optimal filter and its derivative.

Theorem \ref{theorem2} provides qualitative upper bounds on the limits (\ref{1.11.a}), (\ref{1.11.b})
in terms of the number of particles $N$ and the analytical properties of the log-likelihood rate $l(\theta)$.
These bounds are of the almost sure type
and based on the strong mixing condition (Assumption \ref{a3}).
As such, they can be considered as of the worst-case type.
Moreover, these bounds can be rather loose in scenarios
for which the strong mixing condition is too conservative or even undesirable
(e.g., when $q_{\theta}(\cdot|x)$ is concentrated, 
while $p_{\theta}(\cdot|x)$ is diffuse). 
This is partly due to the fact that the recursive particle maximum likelihood algorithm 
analyzed in the paper is based on 
the bootstrap particle filter (\ref{1.5}),  
which is well-known to perform poorly in such scenarios. 
We believe that the bounds could be improved
using more sophisticated schemes 
which sample particles relying on a distribution dependent on the observations 
\cite{pitt&shepard}. 
These bounds could also be made
tighter using the (non-mixing) assumptions on the optimal filter stability adopted in
\cite{douc&fort&moulines&priouret}. 
However, this would require substantial generalization of the existing results
on the stability of the optimal filter derivatives and their particle approximations
(i.e., of the results of \cite{tadic&doucet1}, \cite{tadic&doucet4}).
Since the analysis of the optimal filter derivatives and their particle approximations
would be very difficult under non-mixing stability conditions
and since the results presented here are already complex,
this generalization is left for future research.

\subsection{Outline of Proofs of Convergence Results}

An outline/summary of the main steps and key ideas in the proof of Theorem \ref{theorem2} is provided here.
These steps and ideas can be described as follows. 

{\em Step 1:}
Algorithm (\ref{1.1}) -- (\ref{1.5}) is transformed to stochastic approximation
with Markovian dynamics.
More specifically, it is rewritten as
\begin{align}
	\label{2*.1}
	&
	W_{n+1}
	= 
	W_{n} A_{\theta_{n} }(V_{n},V_{n+1} )
	+
	B_{\theta_{n} }(V_{n},V_{n+1} ),
	\\
	\label{2*.9}
	&
	\theta_{n+1}
	=
	\theta_{n} + \alpha_{n} \left(W_{n+1}C_{\theta_{n}}(V_{n+1}) + D_{\theta_{n}}(V_{n+1}) \right).
\end{align}
The same algorithm is also rewritten as
\begin{align}\label{2*.9'}
	\theta_{n+1}
	=&
	\theta_{n} + \alpha_{n}H(\theta_{n},Z_{n+1} ).
\end{align}
Here, $\{V_{n} \}_{n\geq 0}$, $\{W_{n} \}_{n\geq 0}$ and $\{Z_{n} \}_{n\geq 0}$ 
are the stochastic processes
defined by 
$W_{n}=\big( W_{n,1},\dots,W_{n,N} \big)$ 
and 
\begin{align}\label{2*.403}
	V_{n}
	=
	\big( Y_{n}, X_{n}, \hat{X}_{n} \big),
	\;\;\;\;\; 
	Z_{n}
	=
	\big( V_{n}, W_{n} \big),
\end{align}
while $\hat{X}_{n}$ is the vector of particles $\hat{X}_{n} =	\big( \hat{X}_{n,1},\dots, \hat{X}_{n,N} \big)$
($W_{n}$ is the $d\times N$ matrix whose $j$-th column is $W_{n,j}$).
$A_{\theta}(v,v')$, $B_{\theta}(v,v')$, $C_{\theta}(v)$, $D_{\theta}(v)$ and $H(\theta,z)$
are suitably chosen functions which are defined precisely in
(\ref{2*.101.a}) -- (\ref{2*.401}).
Equations (\ref{2*.1}) -- (\ref{2*.9'}) are a compact form of (\ref{1.1}) -- (\ref{1.3}),
while terms
\begin{align}\label{2*.403'}
	W_{n+1}C_{\theta_{n}}(V_{n+1}) + D_{\theta_{n}}(V_{n+1}),
	\;\;\;\;\;
	H(\theta_{n},Z_{n+1} )
\end{align}
can be viewed as Monte Carlo estimators of $\nabla l(\theta_{n} )$.
Aggregate process $\{(\theta_{n}, Z_{n} ) \}_{n\geq 0}$ is a Markov chain,
while stochastic processes $\{ V_{n} \}_{n\geq 0}$ and $\{Z_{n} \}_{n\geq 0}$
can be interpreted as Markov chains controlled by estimates $\{\theta_{n} \}_{n\geq 0}$
(see (\ref{2*.7'}), (\ref{2*.7})).

{\em Step 2:}
We analyze conditional probability measure of $V_{n+1}$ given $V_{n}=v$, $\theta_{n}=\theta$,
which is denoted by $T_{\theta}(v,dv')$ and precisely defined in (\ref{2*.103'}).
It is shown that $T_{\theta}(v,dv')$ is geometrically ergodic with a rate (locally) uniform in $\theta$.
It is also established that $T_{\theta}(v,dv')$ is (locally) Lipschitz continuous in $\theta$. The details are included in Lemma \ref{lemma2.1}.

{\em Step 3:} We consider the conditional expectations of the products
\begin{align}
	&\label{2*.901.a}
	A_{\theta_{0}}(V_{0},V_{1})\cdots A_{\theta_{n-1}}(V_{n-1},V_{n})
	C_{\theta_{n}}(V_{n}),
	\\
	&\label{2*.901.b}
	B_{\theta_{0}}(V_{0},V_{1})A_{\theta_{1}}(V_{1},V_{2})\cdots A_{\theta_{n-1}}(V_{n-1},V_{n})
	C_{\theta_{n}}(V_{n})
\end{align}
given $\theta_{0}=\theta,\dots,\theta_{n}=\theta$,
$V_{0}=v$.
These conditional expectations are denoted by $\Phi_{\theta}^{n}(v)$, $\Psi_{\theta}^{n}(v)$
and defined precisely in (\ref{2*.107.a}) -- (\ref{2*.107''.b}).
Using results on stochastic matrices
(see Appendix \ref{appendix1}) and the results of Step 2,
it is shown that functions $\Phi_{\theta}^{n}(v)$, $\Psi_{\theta}^{n}(v)$
converge exponentially to zero as $n\rightarrow\infty$
with rates (locally) uniform in $\theta$.
The same functions are also shown to be (locally) Lipschitz continuous with Lipschitz constants
tending exponentially to zero as $n\rightarrow\infty$.
The details are provided in Lemma \ref{lemma2.3}.

{\em Step 4:}
Function $(\Pi^{n}H)(\theta,z)$ and its properties are analyzed, where
$(\Pi^{n}H)(\theta,z)$ is the conditional expectation of $H(\theta_{n},Z_{n+1})$
given $\theta_{0}=\theta,\dots,\theta_{n}=\theta$, $Z_{0}=z$ (see (\ref{l2.4.303})).
$\Pi_{\theta}(z,dz')$ is the conditional probability measure of $Z_{n+1}$
given $Z_{n}=z$, $\theta_{n}=\theta$,
which is defined precisely in (\ref{2*.103''}).
Relying on the results of Step 3,
it is shown that there exists a function $h(\theta)$ such that
$(\Pi^{n}H)(\theta,z)$ converges exponentially to $h(\theta)$
as $n\rightarrow\infty$ at a rate (locally) uniform in $\theta$.
It is also shown that $(\Pi^{n}H)(\theta,z)-h(\theta)$
is (locally) Lipschitz continuous in $\theta$ with a Lipschitz constant
tending exponentially to zero as $n\rightarrow\infty$.
The details are included in Lemma \ref{lemma2.4} (see (\ref{2*.33}), (\ref{l2.4.501}), (\ref{l2.4.23'})).

{\em Step 5:}
The Poisson equation associated with algorithm (\ref{1.1}) -- (\ref{1.5})
(i.e., with functions $H(\theta,z)$, $h(\theta)$ and the transition kernel
$\Pi_{\theta}(z,dz')$)
and its properties are considered.
Relying on the results of Step 4, it is shown that the Poisson equation
has a solution and that the solution is (locally) Lipschitz continuous in $\theta$.
The details are provided in Lemma \ref{lemma2.4}.

{\em Step 6:}
The weight sequence $\{W_{n} \}_{n\geq 0}$ and its stability are studied.
Using results on stochastic matrices (see Appendix \ref{appendix1}),
it is shown that $W_{n}$ is (deterministically) bounded in $n$.
The details are contained in Lemma \ref{lemma3.1}.

{\em Step 7:}
The Monte Carlo estimators (\ref{2*.403'}) and their statistical properties are analyzed. By exploiting the results of \cite{tadic&doucet4}, it is shown that the asymptotic bias of these estimators are inversely proportional to $N$ with a multiplicative constant uniform in $\theta$.
The details are included in Lemma \ref{lemma2.5}.

{\em Step 8:}
Algorithm (\ref{1.1}) -- (\ref{1.5}) is transformed to
a stochastic gradient search with additive noise.
More specifically, it is rewritten as
\begin{align}\label{3*.1}
	\theta_{n+1}
	=
	\theta_{n}+\alpha_{n}(\nabla l(\theta_{n})+\xi_{n}).
\end{align}
Moreover, the additive noise
$\xi_{n}$ is decomposed as $\xi_{n}=\zeta_{n}+\eta_{n}$,
where 
\begin{align*}
	\xi_{n}=H(\theta_{n},Z_{n+1} ) - h(\theta_{n} ), 
	\;\;\;\;\; 
	\zeta_{n}=h(\theta_{n})-\nabla l(\theta_{n} ).
\end{align*}
$\nabla l(\theta_{n})+\xi_{n}$ can be interpreted as an estimator of $\nabla l(\theta_{n})$,
while $\zeta_{n}$ and $\eta_{n}$
can be considered as the variance and bias of this estimator.
Using results of martingale limit theory and the results of Steps 5 and 6,
it is shown that $\{\zeta_{n}\}_{n\geq 0}$ satisfy the Kushner-Clark noise condition.
Relying on the results of Step 7, it is also shown that the asymptotic magnitude of
$\{\eta_{n} \}_{n\geq 0}$ is inversely proportional to $N$ with a deterministic multiplicative
constant.
The details are provided in Lemma \ref{lemma3.2}.

{\em Step 9:}
Using the results presented in Section \ref{section1*} and
the results of \cite{tadic&doucet3}, \cite{tadic&doucet2}, Lemma \ref{lemma3.3} is proved.
Then, relying on the results obtained at Step 8 and the results of \cite{tadic&doucet21},
Theorem \ref{theorem2} is established.

\section{Example}\label{section2}

To illustrate the main results and their applicability,
we use them to study recursive maximum likelihood estimation for the following
non-linear state-space model:
\begin{align}
	&\label{2.1.a}
	X_{n+1}^{\theta,\lambda}
	=
	A_{\theta}(X_{n}^{\theta,\lambda} ) + B_{\theta}(X_{n}^{\theta,\lambda} ) V_{n},
	\\
	&\label{2.1.b}
	Y_{n}^{\theta,\lambda}
	=
	C_{\theta}(X_{n}^{\theta,\lambda} ) + D_{\theta}(X_{n}^{\theta,\lambda} ) W_{n},
	\;\;\;\;\;
	n\geq 0.
\end{align}
Here, $\theta\in\Theta$ and $\lambda\in{\cal P}({\cal X})$ 
are the parameters indexing the state-space model (\ref{2.1.a}), (\ref{2.1.b}). 
$A_{\theta}(x)$ and $B_{\theta}(x)$
are functions which map $\theta\in\Theta$, $x\in\mathbb{R}^{d_{x}}$ (respectively) to
$\mathbb{R}^{d_{x}}$ and $\mathbb{R}^{d_{x}\times d_{x}}$, 
while 
$C_{\theta}(x)$ and $D_{\theta}(x)$ map $\theta\in\Theta$, $x\in\mathbb{R}^{d_{x}}$ (respectively) to
$\mathbb{R}^{d_{y}}$ and $\mathbb{R}^{d_{y}\times d_{y}}$. 
$X_{0}^{\theta,\lambda}$ is an $\mathbb{R}^{d_{x}}$-valued random variable
defined on a probability space $(\Omega,{\cal F},P)$
and distributed according to $\lambda$.
$\{V_{n}\}_{n\geq 0}$ are $\mathbb{R}^{d_{x}}$-valued i.i.d. random variables distributed according the probability density $v(x)$ with respect to the Lebesgue measure while
$\{W_{n}\}_{n\geq 0}$ are $\mathbb{R}^{d_{y}}$-valued i.i.d. random variables distributed according the probability density $w(y)$ with respect to the Lebesgue measure. We also assume that $X_{0}^{\theta,\lambda}$,
$\{V_{n} \}_{n\geq 0}$ and $\{W_{n} \}_{n\geq 0}$ are (jointly) independent.

Let $p_{\theta}(x'|x)$ and $q_{\theta}(y|x)$
be the functions defined by
\begin{align}
	&\label{2.3.a}
	p_{\theta}(x'|x)
	=
	\frac{v\left(B_{\theta}^{-1}(x)(x'-A_{\theta}(x) ) \right) 1_{\cal X}(x') }
	{\int_{\cal X} v\left(B_{\theta}^{-1}(x)(x''-A_{\theta}(x) ) \right) dx'' },
	\\
	&\label{2.3.b}
	q_{\theta}(y|x)
	=
	\frac{w\left(D_{\theta}^{-1}(x)(y-C_{\theta}(x) ) \right) 1_{\cal Y}(y) }
	{\int_{\cal Y} w\left(D_{\theta}^{-1}(x)(y'-C_{\theta}(x) ) \right) dy' }
\end{align}
for $\theta\in\Theta$, $x,x'\in\mathbb{R}^{d_{x}}$, $y\in\mathbb{R}^{d_{y}}$,
where ${\cal X}\in{\cal B}(\mathbb{R}^{d_{x}} )$,
${\cal Y}\in{\cal B}(\mathbb{R}^{d_{y}} )$.
If ${\cal X}=\mathbb{R}^{d_{x}}$, ${\cal Y}=\mathbb{R}^{d_{y}}$,
then $p_{\theta}(x'|x)$ and $q_{\theta}(y|x)$
reduce to the conditional densities
of $X_{n+1}^{\theta,\lambda}$ and $Y_{n}^{\theta,\lambda}$ (respectively)
given $X_{n}^{\theta,\lambda}=x$.
When ${\cal X}\neq\mathbb{R}^{d_{x}}$, ${\cal Y}\neq\mathbb{R}^{d_{y}}$,
$p_{\theta}(x'|x)$ and $q_{\theta}(y|x)$ can be viewed as
a truncation of state-space model (\ref{2.1.a}), (\ref{2.1.b})
to domains ${\cal X}$ and ${\cal Y}$.
Due to the finite precision of digital computers,
this kind of truncation is involved (explicitly or implicitly) in the implementation of
any numerical approximation to the optimal filter
for state-space model (\ref{2.1.a}), (\ref{2.1.b}).
In \cite{heine&crisan}, a truncation scheme similar to (\ref{2.3.a}), (\ref{2.3.b}) 
has been theoretically analyzed
and the choice of the corresponding truncation domain has been addressed.
In the context of algorithm (\ref{1.1}) -- (\ref{1.5}),
the choice of domains ${\cal X}$ and ${\cal Y}$ is much more complex as it involves
many factors such as the stability, accuracy, convergence and convergence rate of
algorithm (\ref{1.1}) -- (\ref{1.5}), as well as
the stability and accuracy of the optimal filter for model (\ref{2.3.a}), (\ref{2.3.b}).
As such, the choice of ${\cal X}$ and ${\cal Y}$ is beyond the scope of this paper.

In this section, we rely on the following assumptions.

\begin{assumption}\label{b1}
${\cal X}$ and ${\cal Y}$ are compact sets with non-empty interiors.
\end{assumption}

\begin{assumption}\label{b2}
$v(x)>0$ and $w(y)>0$ for each $x\in\mathbb{R}^{d_{x}}$, $y\in\mathbb{R}^{d_{y}}$.
$B_{\theta}(x)$ and $D_{\theta}(x)$ are invertible for each $\theta\in\Theta$, $x\in\mathbb{R}^{d_{x}}$.
\end{assumption}

\begin{assumption}\label{b3}
$v(x)$ and $w(y)$ are differentiable for each $x\in\mathbb{R}^{d_{x}}$, $y\in\mathbb{R}^{d_{y}}$.
The first order derivatives of $v(x)$ and $w(y)$ are locally Lipschitz continuous on
$\mathbb{R}^{d_{x}}$, $\mathbb{R}^{d_{y}}$.
$A_{\theta}(x)$, $B_{\theta}(x)$, $C_{\theta}(x)$ and $D_{\theta}(x)$
are differentiable in $\theta$ for each $\theta\in\Theta$, $x\in\mathbb{R}^{d_{x}}$.
The first order derivatives in $\theta$ of
$A_{\theta}(x)$, $B_{\theta}(x)$, $C_{\theta}(x)$ and $D_{\theta}(x)$
are locally Lipschitz continuous in $(\theta,x)$ on $\Theta\times\mathbb{R}^{d_{x}}$.
\end{assumption}

\begin{assumption}\label{b4}
$v(x)$ and $w(y)$ are $p$-times differentiable for each $x\in\mathbb{R}^{d_{x}}$, $y\in\mathbb{R}^{d_{y}}$,
where $p>d$.
The $p$-th order derivatives of $v(x)$ and $w(y)$ are locally bounded on $\mathbb{R}^{d_{x}}$, $\mathbb{R}^{d_{y}}$.
$A_{\theta}(x)$, $B_{\theta}(x)$, $C_{\theta}(x)$ and $D_{\theta}(x)$
are $p$-times differentiable in $\theta$ for each $\theta\in\Theta$, $x\in\mathbb{R}^{d_{x}}$.
The $p$-th order derivatives in $\theta$ of
$A_{\theta}(x)$, $B_{\theta}(x)$, $C_{\theta}(x)$ and $D_{\theta}(x)$
are locally bounded in $(\theta,x)$ on $\Theta\times\mathbb{R}^{d_{x}}$.
\end{assumption}

\begin{assumption}\label{b5}
$v(x)$ and $w(y)$ are real-analytic for each $x\in\mathbb{R}^{d_{x}}$, $y\in\mathbb{R}^{d_{y}}$.
$A_{\theta}(x)$, $B_{\theta}(x)$, $C_{\theta}(x)$ and $D_{\theta}(x)$ are real-analytic
in $(\theta,x)$ for each $\theta\in\Theta$, $x\in\mathbb{R}^{d_{x}}$.
\end{assumption}

Regarding Assumptions \ref{a3} -- \ref{a6} and \ref{b1} -- \ref{b5},
the following relationships can be established.
Assumptions \ref{b1} -- \ref{b3} imply Assumptions \ref{a3} and \ref{a4},
while
Assumptions \ref{b4} and \ref{b5} are particular cases of Assumptions \ref{a5} and \ref{a6}
(respectively).
For the proof of these relationships, see
\cite[Corollary 4.1]{tadic&doucet3}, \cite[Corollary 4.1]{tadic&doucet2} (and the arguments used therein).
Assumptions \ref{b1} -- \ref{b5} are relevant for several practically important
classes of state-space models and cover, for example, stochastic volatility models, dynamic probit models
and their truncated versions.
For other models satisfying (\ref{2.1.a}), (\ref{2.1.b}) and Assumptions \ref{b1} -- \ref{b5},
see \cite{cappe&moulines&ryden}, \cite{crisan&rozovskii}, \cite{douc&moulines&stoffer}
and references cited therein.

As a direct consequence of the relationships between Assumptions \ref{a3} -- \ref{a6} and \ref{b1} -- \ref{b5},
we get the following corollary to Theorem \ref{theorem2}.

\begin{corollary}
(i) If Assumptions \ref{a1}, \ref{a2} and \ref{b1} -- \ref{b3} are satisfied,
then the conclusions of Part (i) of Theorem \ref{theorem2} hold.

(ii) If Assumptions \ref{a1}, \ref{a2} and \ref{b1} -- \ref{b4} are fulfilled,
then the conclusions of Part (ii) of Theorem \ref{theorem2} hold.

(iii) If Assumptions \ref{a1}, \ref{a2}, \ref{b1} -- \ref{b3} and \ref{b5} are satisfied,
then the conclusions of Part (iii) of Theorem \ref{theorem2} hold.
\end{corollary}

\section{Results Related to Optimal Filter and Log-Likelihood Rate}\label{section1*}

In this section, we study the stability and analytical properties of the optimal filter
and its derivative as well as some regularity properties of the log-likelihood rate.
The results presented here are a prerequisite for Lemmas \ref{lemma2.5} and \ref{lemma3.3}.
Note that we only consider here the results
which are essential for the proof of Theorem \ref{theorem2}
and not well-covered in the existing literature on optimal filtering.

Throughout this section and the whole paper, we use the following notation.
$Q$ stands for any compact set satisfying $Q\subset\Theta$.
${\cal M}_{s}({\cal X})$ is the collection of signed measures on ${\cal X}$,
while ${\cal M}_{s}^{d}({\cal X})$ is the set of $d$-dimensional vector measures on ${\cal X}$.
For $\xi\in{\cal M}_{s}({\cal X})$, $|\xi|(dx)$ and $\|\xi\|$
denote (respectively) the total variation and the total variation norm of $\xi$.
For $\zeta\in{\cal M}_{s}^{d}({\cal X})$,
$|\zeta|(dx)$ and $\|\zeta\|$
denote (respectively) the total variation and the total variation norm of $\zeta$
induced by $l_{1}$ vector norm.\footnote
{If $\zeta\in{\cal M}_{s}^{d}({\cal X})$,
then $|\zeta|(dx) = \sum_{i=1}^{d} |e_{i}^{T}\zeta|(dx)$
and $\|\zeta\| = \sum_{i=1}^{d} \|e_{i}^{T}\zeta\|$,
where $e_{i}$ is the $i$-th standard unit vector in $\mathbb{R}^{d}$.
}
$r_{\theta}(x'|y,x)$ is the function defined by
\begin{align*}
	r_{\theta}(x'|y,x)
	=
	p_{\theta}(x'|x)q_{\theta}(y|x)
\end{align*}
for $\theta\in\Theta$, $x,x'\in{\cal X}$, $y\in{\cal Y}$,
while $h_{\theta,y}(x|\xi,\zeta)$ and $H_{\theta,y}(\xi,\zeta)$
are defined for $\xi\in{\cal P}({\cal X})$, $\zeta\in{\cal M}_{s}^{d}({\cal X})$ as
\begin{align}
	&h_{\theta,y}(x|\xi,\zeta)
	=
	\frac{\int r_{\theta}(x|y,x')\zeta(dx') + \int \nabla_{\theta} r_{\theta}(x|y,x')\xi(dx') }
	{\int q_{\theta}(y|x')\xi(dx') },
	\nonumber\\
	\label{1.3*b}
	&H_{\theta,y}(\xi,\zeta)
	=
	\int h_{\theta,y}(x|\xi,\zeta) \mu(dx).
\end{align}
$r_{\theta,\boldsymbol y}^{m:n}(x'|x)$ is the function recursively defined by
\begin{align*}
	&r_{\theta,\boldsymbol y}^{m:m+1}(x'|x)
	=
	r_{\theta}(x'|y_{m},x),
	\\
	&r_{\theta,\boldsymbol y}^{m:n+1}(x'|x)
	=
	\int r_{\theta,\boldsymbol y}^{n:n+1}(x'|x'')
	r_{\theta,\boldsymbol y}^{m:n}(x''|x) \mu(dx'')
\end{align*}
for $n>m\geq 0$
and a sequence $\boldsymbol y = \{y_{n} \}_{n\geq 0}$ in ${\cal Y}$.
$h_{\theta,\boldsymbol y}^{m:n}(x|\xi,\zeta)$ and
$H_{\theta,\boldsymbol y}^{m:n}(\xi,\zeta)$
are the functions defined by
\begin{align*}
	&h_{\theta,\boldsymbol y}^{m:n}(x|\xi,\zeta)
	\!=\!
	\frac{\int r_{\theta,\boldsymbol y}^{m:n}(x|x') \zeta(dx')
	+ \int \nabla_{\theta} r_{\theta,\boldsymbol y}^{m:n}(x|x') \xi(dx')}
	{\iint r_{\theta,\boldsymbol y}^{m:n}(x''|x') \xi(dx')\mu(dx'') },
	\\
	&H_{\theta,\boldsymbol y}^{m:n}(\xi,\zeta)
	\!=\!
	\int h_{\theta,\boldsymbol y}^{m:n}(x|\xi,\zeta) \mu(dx),
\end{align*}
while
$f_{\theta,\boldsymbol y}^{m:n}(x|\xi)$ and $g_{\theta,\boldsymbol y}^{m:n}(x|\xi,\zeta)$
are defined as
\begin{align*}
	&
	f_{\theta,\boldsymbol y}^{m:n}(x|\xi)
	=
	\frac{\int r_{\theta,\boldsymbol y}^{m:n}(x|x') \xi(dx') }
	{\iint r_{\theta,\boldsymbol y}^{m:n}(x''|x') \xi(dx')\mu(dx'') },
	\\
	&
	g_{\theta,\boldsymbol y}^{m:n}(x|\xi,\zeta)
	=
	h_{\theta,\boldsymbol y}^{m:n}(x|\xi,\zeta)
	-
	f_{\theta,\boldsymbol y}^{m:n}(x|\xi)
	H_{\theta,\boldsymbol y}^{m:n}(\xi,\zeta).
\end{align*}
$F_{\theta,\boldsymbol y}^{m:m}(dx|\xi)$,
$F_{\theta,\boldsymbol y}^{m:n}(dx|\xi)$
and
$G_{\theta,\boldsymbol y}^{m:m}(dx|\xi,\zeta)$,
$G_{\theta,\boldsymbol y}^{m:n}(dx|\xi,\zeta)$
are the measures defined for $B\in{\cal B}({\cal X})$ by
$F_{\theta,\boldsymbol y}^{m:m}(B|\xi)=\xi(B)$, 
$G_{\theta,\boldsymbol y}^{m:m}(B|\xi,\zeta)=\zeta(B)$ and 
\begin{align}
	\label{1.401*}
	&
	F_{\theta,\boldsymbol y}^{m:n}(B|\xi)
	=
	\int_{B} f_{\theta,\boldsymbol y}^{m:n}(x|\xi) \mu(dx),
	\\
	\label{1.403*}
	&
	G_{\theta,\boldsymbol y}^{m:n}(B|\xi,\zeta)
	=
	\int_{B} g_{\theta,\boldsymbol y}^{m:n}(x|\xi,\zeta) \mu(dx).
\end{align}
Throughout this paper, the measures $F_{\theta,\boldsymbol y}^{m:n}(dx|\xi)$
and $G_{\theta,\boldsymbol y}^{m:n}(dx|\xi,\zeta)$
are also denoted by $F_{\theta,\boldsymbol y}^{m:n}(\xi)$
and $G_{\theta,\boldsymbol y}^{m:n}(\xi,\zeta)$ (short-hand notation).
Then, it is easy to show that
$F_{\theta,\boldsymbol y}^{m:n}(\xi)$ and $G_{\theta,\boldsymbol y}^{m:n}(\xi,\zeta)$
are the optimal (one-step) predictor and its gradient,
i.e.,
\begin{align*}
	&
	F_{\theta,\boldsymbol y}^{0:n}(B|\lambda)
	=
	P\left(\left. X_{n}^{\theta,\lambda}\in B \right|Y_{0:n-1}^{\theta,\lambda}=y_{0:n-1} \right),
	\\
	&
	G_{\theta,\boldsymbol y}^{0:n}(B|\lambda,\boldsymbol 0)
	=
	\nabla_{\theta} F_{\theta,\boldsymbol y}^{0:n}(B|\lambda)
\end{align*}
for each $\lambda\in{\cal P}({\cal X})$, $n\geq 1$.
Here, $\boldsymbol 0$ is the $d$-dimensional zero-measure
(i.e., $\boldsymbol 0\in{\cal M}_{s}^{d}({\cal X})$, $\|\boldsymbol 0\|=0$).

\begin{lemma}\label{lemma1.3}
Let Assumptions \ref{a2} -- \ref{a4} hold.
Then, the following is true:

(i) $l(\theta)$ is well-defined and differentiable on $\Theta$.

(ii) $\nabla l(\theta)$ is locally Lipschitz continuous on $\Theta$
and satisfies
\begin{align}\label{l1.3.1*}
	\nabla l(\theta)
	=
	\lim_{n\rightarrow\infty}
	E\left(
	H_{\theta,Y_{n} }\big( F_{\theta,\boldsymbol Y}^{0:n}(\xi),
	G_{\theta,\boldsymbol Y}^{0:n}(\xi,\zeta) \big)
	\right)
\end{align}
for all $\theta\in\Theta$, $\xi\in{\cal P}({\cal X})$, $\zeta\in{\cal M}_{s}^{d}({\cal X})$,
where $\boldsymbol Y = \{Y_{n} \}_{n\geq 0}$.

(iii) There exists a real number $C_{1,Q}\in[1,\infty)$
(depending only on $p_{\theta}(x'|x)$, $q_{\theta}(y|x)$)
such that
\begin{align*}
	\big\| G_{\theta,\boldsymbol y}^{0:n}(\xi,\zeta) \big\|
	\leq
	C_{1,Q} (1 + \|\zeta\| )
\end{align*}
for all $\theta\in Q$, $\xi\in{\cal P}({\cal X})$, $\zeta\in{\cal M}_{s}^{d}({\cal X})$,
$n\geq 0$ and any sequence $\boldsymbol y = \{y_{n} \}_{n\geq 0}$.
\end{lemma}

Lemma \ref{lemma1.3} is proved in Appendix \ref{appendix2}.

\section{Results Related to Sequential Monte Carlo Approximations}\label{section2*}

In this section, we study the asymptotic properties of the particles
$\{\hat{X}_{n,i}: n\geq 0, 1\leq i\leq N \}$
and their weights $\{W_{n,i}: n\geq 0,1\leq i\leq N\}$.
Using these properties, we show that the Poisson equation
associated with algorithm (\ref{1.1}) -- (\ref{1.5})
has a Lipschitz continuous solution
(see Lemma \ref{lemma2.4}).
The results presented here are needed to analyze the error in the Monte Carlo estimation of $\nabla l(\theta)$
(see Lemma \ref{lemma3.2} and (\ref{3*.1})).

Throughout this section, we use the following notation.
${\cal V}$ and ${\cal Z}$ are the sets defined by
${\cal V} = {\cal Y} \times {\cal X} \times {\cal X}^{N}$
and ${\cal Z} = {\cal V} \times \mathbb{R}^{d\times N}$.
$e$ is the $N$-dimensional vector whose all elements are one
(i.e., $e=(1,\dots,1)^{T}\in\mathbb{R}^{N}$).
$I$ is the $N\times N$ unit matrix, while $\Lambda$ is the $N\times N$ matrix defined as
\begin{align}\label{2*.701}
	\Lambda
	=
	I - \frac{ee^{T}}{N}.
\end{align}
$A_{\theta}(v,v')$ and $B_{\theta}(v,v')$
are respectively $\mathbb{R}^{N\times N}$ and $\mathbb{R}^{d\times N}$-valued functions
defined by
\begin{align}
	&\label{2*.101.a}
	A_{\theta}^{i,j}(v,v')
	=
	\frac{r_{\theta}(x'_{j}|y,x_{i} ) }
	{\sum_{k=1}^{N} r_{\theta}(x'_{j}|y,x_{k} ) },
	\\
	&\label{2*.101.b}
	B_{\theta}^{j}(v,v')
	=
	\frac{\sum_{k=1}^{N} \nabla_{\theta} r_{\theta}(x'_{j}|y,x_{k} ) }
	{\sum_{k=1}^{N} r_{\theta}(x'_{j}|y,x_{k} ) }
\end{align}
for $\theta\in\Theta$, $x,x'\in{\cal X}$, $y,y'\in{\cal Y}$,
$\hat{x}=(x_{1},\dots,x_{N} )\in{\cal X}^{N}$,
$\hat{x}'=(x'_{1},\dots,x'_{N} )\in{\cal X}^{N}$,
$1\leq i,j\leq N$
and $v=(y,x,\hat{x})$, $v'=(y',x',\hat{x}')$,
where $A_{\theta}^{i,j}(v,v')$ and $B_{\theta}^{j}(v,v')$
are the $(i,j)$-entry of $A_{\theta}(v,v')$
and the $j$-th column of $B_{\theta}(v,v')$
(respectively).
$C_{\theta}(v)$ and $D_{\theta}(v)$ are respectively $\mathbb{R}^{N}$ and
$\mathbb{R}^{d}$-valued functions defined by
\begin{align}
	&\label{2*.101'.a}
	C_{\theta}^{i}(v)
	=
	\frac{q_{\theta}(y|x_{i}) }
	{\sum_{k=1}^{N} q_{\theta}(y|x_{k}) }
	-
	\frac{1}{N},
	\\
	&\label{2*.101'.b}
	D_{\theta}(v)
	=
	\frac{\sum_{k=1}^{N} \nabla_{\theta} q_{\theta}(y|x_{k}) }
	{\sum_{k=1}^{N} q_{\theta}(y|x_{k}) },
\end{align}
where $C_{\theta}^{i}(v)$ is the $i$-th element of $C_{\theta}(v)$.
$H(\theta,z)$ is the function defined by
\begin{align}\label{2*.401}
	H(\theta,z)
	=
	W C_{\theta}(v) + D_{\theta}(v)
\end{align}
for $v\in{\cal V}$, $W\in\mathbb{R}^{d\times N}$
and $z=(v,W)$.
Then, it is straightforward to verify
\begin{align}\label{2*.5}
	e^{T}A_{\theta}(v,v')=e^{T},
	\;\;\;\;\;
	e^{T}C_{\theta}(v)=0
\end{align}
for all $\theta\in\Theta$, $v,v'\in{\cal V}$,
where $A_{\theta}(v,v')$, $C_{\theta}(v)$ are defined in (\ref{2*.101.a}), (\ref{2*.101'.a}).

We rely here on the following notation, too.
$s_{\theta}(x|y,\hat{x})$ is the function defined by
\begin{align}\label{2*.103}
	s_{\theta}(x|y,\hat{x} )
	=
	\frac{\sum_{k=1}^{N} p_{\theta}(x|x_{k}) q_{\theta}(y|x_{k}) }
	{\sum_{k=1}^{N} q_{\theta}(y|x_{k}) }.
\end{align}
For $\hat{x}=(x_{1},\dots,x_{N} )\in{\cal X}^{N}$,
$S_{\theta}(d\hat{x}'|y,\hat{x})$ is the conditional probability measure on ${\cal X}^{N}$ defined for $B\in{\cal B}({\cal X}^{N})$ as
\begin{align*}
	S_{\theta}(B|y,\hat{x})
	=
	\int\cdots\int
	&
	I_{B}(x'_{1},\dots,x'_{N} )
	\left( \prod_{k=1}^{N} s_{\theta}(x'_{k}|y,\hat{x} ) \right)
	\\
	&\cdot 
	\mu(dx'_{1} )\cdots \mu(dx'_{N} ),
\end{align*}
where $I_{B}$ denotes the indicator of $B$.
$T_{\theta}(v,dv')$ is the kernel on ${\cal V}$ defined for $B\in{\cal B}({\cal V})$ and $v=(y,x,\hat{x})$
by
\begin{align}\label{2*.103'}
	T_{\theta}(v,B)
	=
	\iiint
	&
	I_{B}(y',x',\hat{x}') Q(x',dy') 
	\nonumber\\
	&\cdot 
	P(x,dx') S_{\theta}(d\hat{x}'|y,\hat{x}).
\end{align}
$\Pi_{\theta}(z,dz')$ is the kernel on ${\cal Z}$ defined for $B\in{\cal B}({\cal Z})$, $W\in\mathbb{R}^{d\times N}$ and $z=(v,W)$ as
\begin{align}\label{2*.103''}
	\Pi_{\theta}(z,B)
	=
	\int I_{B}\left(v', W A_{\theta}(v,v') + B_{\theta}(v,v') \right) T_{\theta}(v,dv').
\end{align}
Then, it is straightforward to verify that
$\{V_{n} \}_{n\geq 0}$ and
$\{Z_{n} \}_{n\geq 0}$ defined in (\ref{2*.403}) satisfy
\begin{align}
	&\label{2*.7'}
	P(V_{n+1}\in A|\theta_{0},V_{0},\dots,\theta_{n},V_{n} )
	=
	T_{\theta_{n} }(V_{n},A),
	\\
	&\label{2*.7}
	P(Z_{n+1}\in B|\theta_{0},Z_{0},\dots,\theta_{n},Z_{n} )
	=
	\Pi_{\theta_{n} }(Z_{n},B)
\end{align}
almost surely for each $A\in{\cal B}({\cal V})$, $B\in{\cal B}({\cal Z})$, $n\geq 0$.

Using functions $A_{\theta}(v,v')$, $B_{\theta}(v,v')$
and $S_{\theta}(d\hat{x}'|y,\hat{x})$,
we introduce the following notation.
$\{\hat{X}_{n,i}^{\theta}: n\geq 0,1\leq i\leq N \}$
are ${\cal X}$-valued random variables generated through the
sequential Monte Carlo scheme
\begin{align}\label{2*.21}
	\hat{X}_{n+1,i}^{\theta}
	\sim
	s_{\theta}\left(x\big|
	Y_{n}, \big( \hat{X}_{n,1}^{\theta},\dots,\hat{X}_{n,N}^{\theta} \big) \right) \mu(dx),
\end{align}
while $\hat{X}_{n}^{\theta}$, $V_{n}^{\theta}$
are the random variables defined by
\begin{align*}
	\hat{X}_{n}^{\theta}
	=
	\big(\hat{X}_{n,1}^{\theta}, \dots, \hat{X}_{n,N}^{\theta} \big),
	\;\;\;\;\;
	V_{n}^{\theta}
	=
	\big(Y_{n}, X_{n}, \hat{X}_{n}^{\theta} \big).
\end{align*}
In (\ref{2*.21}), $\{\hat{X}_{n+1,i}^{\theta}: 1\leq i\leq N\}$
are sampled independently from one another and independently of
$\{X_{k}: 0\leq k\leq n\}$, $\{Y_{k},\hat{X}_{k,i}^{\theta}: 0\leq k<n, 1\leq i\leq N \}$,
while $\{\hat{X}_{0,i}^{\theta}: 1\leq i\leq N\}$ are selected independently of $(X_{0},Y_{0})$.
$\{W_{n}^{\theta} \}_{n\geq 0}$ are $d\times N$ random matrices
generated by the recursion
\begin{align}\label{2*.23}
	W_{n+1}^{\theta}
	=
	W_{n}^{\theta} A_{\theta}(V_{n}^{\theta},V_{n+1}^{\theta} )
	+
	B_{\theta}(V_{n}^{\theta},V_{n+1}^{\theta} ),
\end{align}
while
$Z_{n}^{\theta}$ is the random variable
defined by $Z_{n}^{\theta} = (V_{n}^{\theta}, W_{n}^{\theta} )$.
In (\ref{2*.23}), $W_{0}^{\theta}$ is selected independently of $(X_{0},Y_{0})$.
Then, it can easily be shown that
$\{V_{n}^{\theta} \}_{n\geq 0}$ and $\{Z_{n}^{\theta} \}_{n\geq 0}$
are Markov chains whose transition kernels are
$T_{\theta}(v,dv')$ and $\Pi_{\theta}(z,dz')$
(respectively).

Using functions $A_{\theta}(v,v')$, $B_{\theta}(v,v')$, $C_{\theta}(v)$, $D_{\theta}(v)$
(defined in (\ref{2*.101.a}) -- (\ref{2*.101'.b})) and
stochastic process $\{V_{n}^{\theta} \}_{n\geq 0}$,
we introduce the following notation.
$T_{\theta}^{n}(v,dv')$ and $\tau_{\theta}(dv)$
are (respectively) the $n$-th step transition kernel and the invariant probability measure of
$\{V_{n}^{\theta} \}_{n\geq 0}$
(the existence and uniqueness of $\tau_{\theta}(dv)$ are guaranteed by Lemma \ref{lemma2.1}).
$\tilde{T}_{\theta}^{n}(v,dv')$ is the kernel on ${\cal V}$ defined for $B\in{\cal B}({\cal V})$ by
\begin{align}\label{2*.105}
	\tilde{T}_{\theta}^{n}(v,B)
	=
	T_{\theta}^{n}(v,B)
	-
	\tau_{\theta}(B).
\end{align}
$\tilde{A}_{\theta}^{0}(v)$ and $\Phi_{\theta}^{0}(v)$ are the functions defined
by
\begin{align}\label{2*.107'}
	\tilde{A}_{\theta}^{0}(v) = I,
	\;\;\;\;\;
	\Phi_{\theta}^{0}(v) = C_{\theta}(v).
\end{align}
$\tilde{A}_{\theta}^{n}(v_{0},\dots,v_{n})$
and $\Phi_{\theta}^{n}(v)$ are the functions defined for $v,v_{0},\dots,v_{n}\in{\cal V}$, $n\geq 1$ by
\begin{align}
	&\label{2*.107.a}
	\tilde{A}_{\theta}^{n}(v_{0},\dots,v_{n})
	=
	A_{\theta}(v_{0},v_{1})\cdots A_{\theta}(v_{n-1},v_{n}),
	\\
	&\label{2*.107.b}
	\Phi_{\theta}^{n}(v)
	=
	E\left(\left.
	\tilde{A}_{\theta}^{n}(V_{0}^{\theta},\dots,V_{n}^{\theta} )
	C_{\theta}(V_{n}^{\theta} )
	\right|V_{0}^{\theta}=v\right).
\end{align}
$\tilde{B}_{\theta}^{n}(v_{0},\dots,v_{n})$
and $\Psi_{\theta}^{n}(v)$ are the functions defined by
\begin{align}
	&\label{2*.107''.a}
	\tilde{B}_{\theta}^{n}(v_{0},\dots,v_{n})
	=
	B_{\theta}(v_{0},v_{1}) \tilde{A}_{\theta}^{n-1}(v_{1},\dots,v_{n}),
	\\
	&\label{2*.107''.b}
	\Psi_{\theta}^{n}(v)
	=
	E\left(\left.
	\tilde{B}_{\theta}^{n}(V_{0}^{\theta},\dots,V_{n}^{\theta} )
	C_{\theta}(V_{n}^{\theta} )
	\right|V_{0}^{\theta}=v\right).
\end{align}
$h(\theta)$ is the function defined by
\begin{align}\label{2*.405}
	h(\theta)
	=
	\int D_{\theta}(v)\tau_{\theta}(dv)
	+
	\sum_{n=1}^{\infty} \int \Psi_{\theta}^{n}(v)\tau_{\theta}(dv).
\end{align}
Then, for each $\theta\in\Theta$, $v\in{\cal V}$, $n\geq 1$,
it is straightforward to verify
\begin{align}\label{2*.33}
	\Psi_{\theta}^{n}(v)
	=
	E\left(\left.
	B_{\theta}(V_{0}^{\theta},V_{1}^{\theta})
	\Phi_{\theta}^{n}(V_{1}^{\theta})
	\right|V_{0}^{\theta}=v\right).
\end{align}

\begin{remark}
Throughout this and subsequent sections, the following convention is applied.
Diacritic $\tilde{}$ is used to denote a locally defined quantity,
i.e., a quantity whose definition holds only within the proof where
the quantity appears.
We also recall here that $Q$ stands for any compact set satisfying
$Q\subset\Theta$.
\end{remark}

\begin{lemma}\label{lemma2.1}
Let Assumptions \ref{a2} -- \ref{a4} hold.
Then, the following is true:

(i) $\{ V_{n}^{\theta} \}_{n\geq 0}$ is geometrically ergodic for each $\theta\in\Theta$.

(ii) There exist real numbers $\rho_{1,Q}\in(0,1)$, $C_{2,Q}\in[1,\infty)$ (possibly depending on $N$)
such that
\begin{align}
	&\label{l2.1.1*}
	|\tilde{T}_{\theta}^{n}(v,B) |
	\leq
	C_{2,Q}\rho_{1,Q}^{n},
	\\
	&\label{l2.1.3*}
	|\tilde{T}_{\theta}^{n}(v,B) - \tilde{T}_{\theta'}^{n}(v,B) |
	\leq
	C_{2,Q}\rho_{1,Q}^{n} \|\theta-\theta'\|,
	\\
	&\label{l2.1.5*}
	\begin{aligned}[b]
	&
	\max\{
	|\tau_{\theta}(B) - \tau_{\theta'}(B) |,
	|T_{\theta}(v,B) - T_{\theta'}(v,B) |
	\}
	\\
	&\leq
	C_{2,Q} \|\theta-\theta'\|
	\end{aligned}
\end{align}
for all $\theta,\theta'\in Q$,
$v\in{\cal V}$, $B\in{\cal B}({\cal V})$, $n\geq 0$.
\end{lemma}

\begin{proof}
Using Assumption \ref{a2} and \cite[Theorem 16.0.2]{meyn&tweedie},
we conclude that there exist an integer $n_{0}\geq 1$, a real number $\gamma\in(0,1)$
and a probability measure $\xi(dx)$ on ${\cal X}$ such that
$P^{n_{0}}(x,A)\geq \gamma\xi(A)$
for all $x\in{\cal X}$, $A\subseteq{\cal B}({\cal X})$.

Throughout the proof, the following notation is used.
$\theta$, $\theta'$ are any elements of $Q$.
$x,x_{1},\dots,x_{N}$ are any elements of ${\cal X}$,
while $\hat{x}=(x_{1},\dots,x_{N} )$.
%(notice that $\hat{x}$ is any element of ${\cal X}^{N}$).
$y$ is any element of ${\cal Y}$, while $v=(y,x,\hat{x} )$.
$B$ is any element of ${\cal B}({\cal V})$,
while $n$ is any non-negative integer.
$\zeta(d\hat{x})$ is the probability measure on ${\cal X}^{N}$ defined for $A\in{\cal B}({\cal X}^{N})$ by
\begin{align*}
	\zeta(A)
	\!=\!
	\left(\frac{1}{\mu({\cal X}) } \right)^{\!N}
	\!\!\int\!\cdots\!\int\! 
	I_{A}(x_{1},\dots,x_{N}) \mu(dx_{1})\cdots\mu(dx_{N}).
\end{align*}

Let $\beta_{Q}=(\varepsilon_{Q}\mu({\cal X}) )^{N}$
($\varepsilon_{Q}$ is specified in Assumption \ref{a3}, while $\mu(dx)$ is defined
in Subsection \ref{ssection1.1}).
Relying on Assumption \ref{a3}, we deduce
\begin{align}\label{l2.1.5}
	\varepsilon_{Q} \leq s_{\theta}(x|y,\hat{x}) \leq \frac{1}{\varepsilon_{Q} }.
\end{align}
Consequently, for $A\in{\cal B}({\cal X}^{N})$, we get
\begin{align*}
	S_{\theta}(A|y,\hat{x})
	\geq&
	\varepsilon_{Q}^{N}
	\int\cdots\int I_{A}(x'_{1},\dots,x'_{N}) \mu(dx'_{1})\cdots\mu(dx'_{N})
	\\
	=&
	\beta_{Q} \zeta(A).
\end{align*}	
Hence, we have
\begin{align*}
	T_{\theta}(v,B)
	=&
	\begin{aligned}[t]
	\iiint 
	&I_{B}(y',x',\hat{x}') Q(x',dy') 
	\\
	&\cdot
	P(x,dx') S_{\theta}(d\hat{x}'|y,\hat{x} )
	\end{aligned}
	\\
	\geq &
	\beta_{Q}
	\iiint I_{B}(y',x',\hat{x}')Q(x,dy')P(x,dx')\zeta(d\hat{x}').
\end{align*}
Therefore, we get
\begin{align*}
	&
	T_{\theta}^{n+1}(v,B)
	=
	E\left( T_{\theta}(V_{n}^{\theta},B) |V_{0}^{\theta}=v \right)
	\\
	&\geq 
	\begin{aligned}[t]
	\!\beta_{Q}
	E\Bigg(\!
	\iiint 
	&\!
	I_{B}(y',x',\hat{x}' ) Q(x',dy') 
	\\
	&\!\!\!\!\!\!\!
	\cdot 
	P(X_{n},dx') \zeta(d\hat{x}' )
	\Bigg|
	Y_{0}=y,X_{0}=x,\hat{X}_{0}^{\theta}=\hat{x}
	\Bigg)
	\end{aligned}
	\\
	&=
	\!\beta_{Q}
	\iiint I_{B}(y',x',\hat{x}' ) Q(x',dy') P^{n+1}(x,dx') \zeta(d\hat{x}' ).
\end{align*}
Since $P^{n_{0}}(x,A)\geq \gamma\xi(A)$
for any $A\subseteq{\cal B}({\cal X})$, we get
\begin{align}\label{l2.1.1}
	T_{\theta}^{n_{0}}(v,B)
	\geq
	\beta_{Q}\gamma
	\iiint I_{B}(y',x',\hat{x}' ) Q(x',dy') \xi(dx') \zeta(d\hat{x}' ).
\end{align}

Let $\rho_{1,Q}=(1-\beta_{Q}\gamma)^{1/(2n_{0})}$.
As $v$ is any element in ${\cal V}$, \cite[Theorem 16.0.2]{meyn&tweedie}) and (\ref{l2.1.1}) imply that
$\{V_{n}^{\theta} \}_{n\geq 0}$ is geometrically ergodic.
The same arguments also imply
\begin{align}\label{l2.1.3}
	|\tilde{T}_{\theta}^{n}(v,B) |
	=
	|T_{\theta}^{n}(v,B) - \tau_{\theta}(B) |
	\leq
	\rho_{1,Q}^{2n}.
\end{align}
Since $Q$ is any compact set in $\Theta$, we conclude that (i) is true.

Let $\tilde{C}_{1,Q}=3\varepsilon_{Q}^{-2}K_{1,Q}$,
$\tilde{C}_{2,Q}=\varepsilon_{Q}^{-N}\tilde{C}_{1,Q}N$,
$\tilde{C}_{3,Q}=(\mu({\cal X}) )^{N} \tilde{C}_{2,Q}$
($K_{1,Q}$ is specified in Assumption \ref{a4}).
Owing to Assumptions \ref{a3}, \ref{a4}, we have
\begin{align*}
	&
	|s_{\theta}(x|y,\hat{x} ) - s_{\theta'}(x|y,\hat{x} ) |
	\\
	&
	\begin{aligned}
	\leq &
	\frac{\sum_{i=1}^{N} |p_{\theta}(x|x_{i}) - p_{\theta'}(x|x_{i} ) | q_{\theta}(y|x_{i} ) }
	{\sum_{i=1}^{N} q_{\theta}(y|x_{i} ) }
	\\
	&+
	\frac{\sum_{i=1}^{N} p_{\theta'}(x|x_{i} ) |q_{\theta}(y|x_{i}) - q_{\theta'}(y|x_{i} ) |  }
	{\sum_{i=1}^{N} q_{\theta}(y|x_{i} ) }
	\\
	&+
	\frac{s_{\theta'}(x|y,\hat{x} ) \sum_{i=1}^{N} |q_{\theta}(y|x_{i}) - q_{\theta'}(y|x_{i} ) |  }
	{\sum_{i=1}^{N} q_{\theta}(y|x_{i} ) }
	\end{aligned}
	\\
	&\leq 
	\frac{3K_{1,Q}\|\theta-\theta'\|}{\varepsilon_{Q}^{2} }
	=
	\tilde{C}_{1,Q}\|\theta-\theta'\|.
\end{align*}
Consequently, for any $x'_{1},\dots,x'_{N}\in{\cal X}$, (\ref{l2.1.5}) yields
\begin{align*}
	&
	\left|
	\prod_{i=1}^{N} s_{\theta}(x'_{i}|y,\hat{x} )
	-
	\prod_{i=1}^{N} s_{\theta'}(x'_{i}|y,\hat{x} )
	\right|
	\\
	&
	\leq 
	\begin{aligned}[t]
	\sum_{i=1}^{N}
	&
	\left(\prod_{j=1}^{i-1} s_{\theta}(x'_{j}|y,\hat{x} ) \right)
	\left(\prod_{j=i+1}^{N} s_{\theta'}(x'_{j}|y,\hat{x} ) \right)
	\\
	&\cdot
	|s_{\theta}(x'_{i}|y,\hat{x} ) - s_{\theta'}(x'_{i}|y,\hat{x} ) |
	\end{aligned}
	\\
	&\leq 
	\frac{\tilde{C}_{1,Q}N\|\theta-\theta'\|}{\varepsilon_{Q}^{N-1} }
	\leq 
	\tilde{C}_{2,Q} \|\theta-\theta'\|.
\end{align*}
Here and throughout the paper, we use the convention that the product $\prod_{i=k}^{l}$ is one whenever $k>l$.
Hence, we have
\begin{align*}
	&
	|S_{\theta}(B|y,\hat{x} ) - S_{\theta'}(B|y,\hat{x} ) |
	\\
	&\leq 
	\begin{aligned}[t]
	\!\int\!\cdots\!\int
	&
	I_{B}(x_{1},\dots,x_{N})
	\left|
	\prod_{i=1}^{N} s_{\theta}(x_{i}|y,\hat{x} )
	\!-\!
	\prod_{i=1}^{N} s_{\theta'}(x_{i}|y,\hat{x} )
	\right|
	\\
	&
	\cdot
	\mu(dx_{1})\cdots\mu(dx_{N})
	\end{aligned}
	\\
	&\leq
	\tilde{C}_{2,Q} (\mu({\cal X}) )^{N} \|\theta-\theta'\|
	=
	\tilde{C}_{3,Q} \|\theta-\theta'\|.
\end{align*}
Therefore, we get
\begin{align}\label{l2.1.7}
	|T_{\theta}(v,B) - T_{\theta'}(v,B) |
	\leq &
	\begin{aligned}[t]
	\iiint
	&
	I_{B}(y',x',\hat{x}' ) Q(x',dy') 
	\\
	&\cdot 
	P(x,dx')|S_{\theta}\!-\!S_{\theta'}|(d\hat{x}'|y,\hat{x} )
	\end{aligned}
	\nonumber\\
	\leq&
	\tilde{C}_{3,Q} \|\theta-\theta'\|.
\end{align}
Here, $|S_{\theta}-S_{\theta'}|(d\hat{x}'|y,\hat{x} )$
denotes the total variation of the signed measure
$S_{\theta}(d\hat{x}'|y,\hat{x} ) - S_{\theta'}(d\hat{x}'|y,\hat{x} )$.

Let $\tilde{C}_{4,Q}\in[1,\infty)$ be an upper bound of
sequence $\{n\rho_{1,Q}^{n-1} \}_{n\geq 1}$,
while $C_{2,Q}=2\tilde{C}_{3,Q}\tilde{C}_{4,Q}(1-\rho_{1,Q})^{-1}$.
Using (\ref{l2.1.3}), (\ref{l2.1.7}), we conclude
\begin{align}\label{l2.1.21}
	&
	|T_{\theta}^{n+1}(v,B) - T_{\theta'}^{n+1}(v,B) |
	\nonumber\\
	&=
	\left|
	\sum_{i=0}^{n}
	\iint
	\tilde{T}_{\theta}^{i}(v'',B)
	(T_{\theta}-T_{\theta'} )(v',dv'')
	T_{\theta'}^{n-i}(v,dv')
	\right|
	\nonumber\\
	&\leq
	\sum_{i=0}^{n}
	\iint
	|\tilde{T}_{\theta}^{i}(v'',B) |
	|T_{\theta}-T_{\theta'} |(v',dv'')
	T_{\theta'}^{n-i}(v,dv')
	\nonumber\\
	&\leq
	\tilde{C}_{3,Q} \|\theta-\theta'\| \sum_{i=0}^{n} \rho_{1,Q}^{2i}
	\leq
	C_{2,Q}\|\theta-\theta'\|.
\end{align}
Similarly, we deduce
\begin{align}\label{l2.1.23}
	&
	|\tilde{T}_{\theta}^{n+1}(v,B) - \tilde{T}_{\theta'}^{n+1}(v,B) |
	\nonumber\\
	&=
	\left|
	\sum_{i=0}^{n}
	\iint
	\tilde{T}_{\theta}^{i}(v'',B)
	(T_{\theta}-T_{\theta'} )(v',dv'')
	\tilde{T}_{\theta'}^{n-i}(v,dv')
	\right|
	\nonumber\\
	&\leq
	\sum_{i=0}^{n}
	\iint
	|\tilde{T}_{\theta}^{i}(v'',B) |
	|T_{\theta}-T_{\theta'} |(v',dv'')
	|\tilde{T}_{\theta'}^{n-i} |(v,dv')
	\nonumber\\
	&\leq
	\tilde{C}_{3,Q} \rho_{1,Q}^{2n} (n+1) \|\theta-\theta'\|
	\leq
	C_{2,Q}\rho_{1,Q}^{n+1} \|\theta-\theta'\|.
\end{align}
Combining (\ref{l2.1.3}), (\ref{l2.1.21}), we get
\begin{align}\label{l2.1.25}
	|\tau_{\theta}(B)-\tau_{\theta'}(B)|
	\leq &
	|T_{\theta}^{n}(v,B) - T_{\theta'}^{n}(v,B) |
	+
	|\tilde{T}_{\theta}^{n}(v,B) |
	\nonumber\\
	&+
	|\tilde{T}_{\theta'}^{n}(v,B) |
	\nonumber\\
	\leq&
	C_{2,Q}\|\theta-\theta'\| + 2\rho_{1,Q}^{n}.
\end{align}
Letting $n\rightarrow\infty$ in (\ref{l2.1.25}) and
using (\ref{l2.1.3}), (\ref{l2.1.7}), (\ref{l2.1.23}),
we conclude that (\ref{l2.1.1*}) -- (\ref{l2.1.5*}) hold.
\end{proof}

\begin{lemma}\label{lemma2.2}
Let Assumptions \ref{a3} and \ref{a4} hold.
Then, the following is true:

(i) There exists a real number $\rho_{2,Q}\in(0,1)$
(independent of $N$ and depending only on $p_{\theta}(x'|x)$,
$q_{\theta}(y|x)$)
such that $A_{\theta}^{i,j}(v,v')\geq\rho_{2,Q}/N$
for all $\theta\in Q$, $v,v'\in{\cal V}$, $1\leq i,j\leq N$.

(ii) There exists a real number $C_{3,Q}\in[1,\infty)$ (possibly depending on $N$)
such that
\begin{align}
	&\label{l2.2.1*}
	\max\{\|A_{\theta}(v,v')\|, \|B_{\theta}(v,v')\|, \|C_{\theta}(v)\|, \|D_{\theta}(v)\| \}
	\leq
	C_{3,Q},
	\\
	&\label{l2.2.3*}
	\max\{\|A_{\theta}(v,v') - A_{\theta'}(v,v') \|, \|B_{\theta}(v,v') - B_{\theta'}(v,v') \| \}
	\nonumber\\
	&\leq
	C_{3,Q} \|\theta-\theta'\|,
	\\
	&\label{l2.2.5*}
	\max\{\|C_{\theta}(v) - C_{\theta'}(v) \|, \|D_{\theta}(v) - D_{\theta'}(v) \| \}
	\nonumber\\
	&\leq
	C_{3,Q} \|\theta-\theta'\|
\end{align}
for all $\theta,\theta'\in Q$, $v,v'\in{\cal V}$.
\end{lemma}

\begin{proof}
Throughout the proof, the following notation is used.
$\theta$, $\theta'$ are any elements of $Q$.
$x,x'$,
$x_{1},x'_{1},\dots,x_{N},x'_{N}$ are any elements of ${\cal X}$,
while $\hat{x}=(x_{1},\dots,x_{N} )$, $\hat{x}'=(x'_{1},\dots,x'_{N} )$.
$y$, $y'$ are any elements of ${\cal Y}$,
while $v=(y,x,\hat{x} )$, $v'=(y',x',\hat{x}' )$.
$i,j$ are any integers satisfying $1\leq i,j\leq N$.

Let $\rho_{2,Q}=\varepsilon_{Q}^{4}$
($\varepsilon_{Q}$ is specified in Assumption \ref{a3}).
Owing to Assumption \ref{a3}, we have
$\varepsilon_{Q}^{2}\leq r_{\theta}(x'|y,x)\leq 1/\varepsilon_{Q}^{2}$.
Therefore, we get
$A_{\theta}^{i,j}(v,v') \geq \varepsilon_{Q}^{4}/N = \rho_{2,Q}/N$.
Hence, (i) is true.

Due to Assumptions \ref{a3} and \ref{a4}, we have
\begin{align*}
	|r_{\theta}(x'|y,x) - r_{\theta'}(x'|y,x) |
	\leq &
	|p_{\theta}(x'|x) - p_{\theta'}(x'|x) | q_{\theta}(y|x)
	\\
	&+
	p_{\theta'}(x'|x) |q_{\theta}(y|x) - q_{\theta'}(y|x) |
	\\
	\leq &
	\frac{2K_{1,Q}\|\theta-\theta'\|}{\varepsilon_{Q} }.
\end{align*}
Then, we get
\begin{align}\label{l2.2.1}
	&
	|A_{\theta}^{i,j}(v,v') - A_{\theta'}^{i,j}(v,v') |
	\nonumber\\
	&
	\begin{aligned}
	\leq & 
	\frac{|r_{\theta}(x'_{j}|y,x_{i}) - r_{\theta'}(x'_{j}|y,x_{i}) | }
	{\sum_{k=1}^{N} r_{\theta}(x'_{j}|y,x_{k}) }
	\nonumber\\
	&+
	A_{\theta'}^{i,j}(v,v')
	\frac{\sum_{k=1}^{N} 	|r_{\theta}(x'_{j}|y,x_{k}) - r_{\theta'}(x'_{j}|y,x_{k}) | }
	{\sum_{k=1}^{N} r_{\theta}(x'_{j}|y,x_{k}) }
	\end{aligned} 
	\nonumber\\
	&\leq 
	\frac{2K_{1,Q}\|\theta-\theta'\|}{\varepsilon_{Q}^{3} }
	\left(\frac{1}{N} + A_{\theta'}^{i,j}(v,v') \right).
\end{align}
Since $\sum_{i=1}^{N} A_{\theta}^{i,j}(v,v')=1$
(due to (\ref{2*.5})), (\ref{l2.2.1}) implies
\begin{align}\label{l2.2.3}
	&
	\sum_{i=1}^{N}
	|A_{\theta}^{i,j}(v,v') - A_{\theta'}^{i,j}(v,v') |
	\nonumber\\
	&
	\leq
	\frac{2K_{1,Q}\|\theta-\theta'\|}{\varepsilon_{Q}^{3} }
	\left(1 + \sum_{i=1}^{N} A_{\theta'}^{i,j}(v,v') \right)
	\nonumber\\
	&=
	\frac{4K_{1,Q}\|\theta-\theta'\|}{\varepsilon_{Q}^{3} }.
\end{align}

It is straightforward to verify
\begin{align}\label{l2.2.101}
	B_{\theta}^{j}(v,v')
	=
	\sum_{i=1}^{N} A_{\theta}^{i,j}(v,v')
	\frac{\nabla_{\theta}r_{\theta}(x'_{j}|y,x_{i}) }{r_{\theta}(x'_{j}|y,x_{i}) }.
\end{align}
Moreover, using Assumptions \ref{a3} and \ref{a4}, we conclude
\begin{align}\label{l2.2.103}
	\left\|
	\frac{\nabla_{\theta}r_{\theta}(x'|y,x) }{r_{\theta}(x'|y,x) }
	\right\|
	\leq&
	\left\|
	\frac{\nabla_{\theta}p_{\theta}(x'|x) }{p_{\theta}(x'|x) }
	\right\|
	+
	\left\|
	\frac{\nabla_{\theta}q_{\theta}(y|x) }{q_{\theta}(y|x) }
	\right\|
	\nonumber\\
	\leq&
	\frac{2K_{1,Q} }{\varepsilon_{Q} }.
\end{align}
Relying on the same assumptions, we deduce
\begin{align}\label{l2.2.105}
	&
	\left\|
	\frac{\nabla_{\theta}r_{\theta}(x'|y,x) }{r_{\theta}(x'|y,x) }
	-
	\frac{\nabla_{\theta}r_{\theta'}(x'|y,x) }{r_{\theta'}(x'|y,x) }
	\right\|
	\nonumber\\
	&
	\begin{aligned}
	\leq &
	\frac{\|\nabla_{\theta}p_{\theta}(x'|x) \!-\! \nabla_{\theta}p_{\theta'}(x'|x) \| }
	{p_{\theta}(x'|x) }
	\\
	&+
	\left\|
	\frac{\nabla_{\theta}p_{\theta'}(x'|x) }{p_{\theta'}(x'|x) }
	\right\|
	\frac{|p_{\theta}(x'|x) \!-\! p_{\theta'}(x'|x) | }
	{p_{\theta}(x'|x) }
	\nonumber\\
	&+
	\frac{\|\nabla_{\theta}q_{\theta}(y|x) - \nabla_{\theta}q_{\theta'}(y|x) \| }
	{q_{\theta}(y|x) }
	\\
	&+
	\left\|
	\frac{\nabla_{\theta}q_{\theta'}(y|x) }{q_{\theta'}(y|x) }
	\right\|
	\frac{|q_{\theta}(y|x) - q_{\theta'}(y|x) | }
	{q_{\theta}(y|x) }
	\end{aligned} 
	\nonumber\\
	&\leq 
	\frac{4K_{1,Q}^{2}\|\theta-\theta'\| }{\varepsilon_{Q}^{2} }.
\end{align}
Then, (\ref{2*.5}), (\ref{l2.2.101}), (\ref{l2.2.103}) imply
\begin{align}\label{l2.2.5}
	\|B_{\theta}^{j}(v,v') \|
	\leq&
	\sum_{i=1}^{N}
	A_{\theta}^{i,j}(v,v')
	\left\|\frac{\nabla_{\theta}r_{\theta}(x'_{j}|y,x_{i}) }{r_{\theta}(x'_{j}|y,x_{i}) } \right\|
	\nonumber\\
	\leq&
	\frac{2K_{1,Q}}{\varepsilon_{Q} }.
\end{align}
Similarly, (\ref{2*.5}), (\ref{l2.2.3}) -- (\ref{l2.2.105}) yield
\begin{align}\label{l2.2.7}
	&
	\|B_{\theta}^{j}(v,v') - B_{\theta'}^{j}(v,v') \|
	\nonumber\\
	&
	\begin{aligned}
	\leq &
	\sum_{i=1}^{N}
	|A_{\theta}^{i,j}(v,v') - A_{\theta'}^{i,j}(v,v') |
	\left\|\frac{\nabla_{\theta}r_{\theta}(x'_{j}|y,x_{i} ) }{r_{\theta}(x'_{j}|y,x_{i}) } \right\|
	\nonumber\\
	&+
	\sum_{i=1}^{N}
	A_{\theta'}^{i,j}(v,v')
	\left\|\frac{\nabla_{\theta}r_{\theta}(x'_{j}|y,x_{i} ) }{r_{\theta}(x'_{j}|y,x_{i}) }
	-
	\frac{\nabla_{\theta}r_{\theta'}(x'_{j}|y,x_{i} ) }{r_{\theta'}(x'_{j}|y,x_{i}) }\right\|
	\end{aligned}
	\nonumber\\
	&\leq 
	\frac{12K_{1,Q}^{2}\|\theta-\theta'\|}{\varepsilon_{Q}^{4} }.
\end{align}

Due to Assumptions \ref{a3}, \ref{a4}, we have
\begin{align}
	&\label{l2.2.9.a}
	|C_{\theta}^{i}(v)|
	\leq
	\max\left\{
	\frac{1}{N},
	\frac{q_{\theta}(y|x_{i}) }{\sum_{k=1}^{N} q_{\theta}(y|x_{k} ) }
	\right\}
	\leq
	1,
	\\
	&\label{l2.2.9.b}
	\|D_{\theta}(v)\|
	\leq
	\frac{\sum_{k=1}^{N} \|\nabla_{\theta}q_{\theta}(y|x_{k}) \| }
	{\sum_{k=1}^{N} q_{\theta}(y|x_{k}) }
	\leq
	\frac{K_{1,Q} }{\varepsilon_{Q} }.
\end{align}
Combining Assumptions \ref{a3}, \ref{a4} and (\ref{l2.2.9.a}), we get
\begin{align}\label{l2.2.21}
	|C_{\theta}^{i}(v) - C_{\theta'}^{i}(v) |
	\leq &
	\frac{|q_{\theta}(y|x_{i}) - q_{\theta'}(y|x_{i}) |}
	{\sum_{k=1}^{N} q_{\theta}(y|x_{k}) }
	\nonumber\\
	&+
	\frac{|C_{\theta'}^{i}(v)| \sum_{k=1}^{N} |q_{\theta}(y|x_{k}) - q_{\theta'}(y|x_{k}) | }
	{\sum_{k=1}^{N} q_{\theta}(y|x_{k}) }
	\nonumber\\
	\leq &
	\frac{2K_{1,Q}\|\theta-\theta'\|}{\varepsilon_{Q} }.
\end{align}
Moreover, using Assumptions \ref{a3}, \ref{a4} and (\ref{l2.2.9.b}), we get
\begin{align}\label{l2.2.23}
	\|D_{\theta}(v) \!-\! D_{\theta'}(v) \|
	\leq &
	\frac{\sum_{k=1}^{N}
	\|\nabla_{\theta}q_{\theta}(y|x_{k}) - \nabla_{\theta}q_{\theta'}(y|x_{k}) \| }
	{\sum_{k=1}^{N} q_{\theta}(y|x_{k}) }
	\nonumber\\
	&+\!
	\frac{\|D_{\theta'}(v)\|
	\sum_{k=1}^{N}\! |q_{\theta}(y|x_{k}) \!-\! q_{\theta'}(y|x_{k}) | }
	{\sum_{k=1}^{N} q_{\theta}(y|x_{k}) }
	\nonumber\\
	\leq &
	\frac{2K_{1,Q}^{2}\|\theta-\theta'\|}{\varepsilon_{Q}^{2} }.
\end{align}

Let $C_{3,Q}=12\varepsilon_{Q}^{-4}K_{1,Q}^{2}N$.
Then, relying on (\ref{2*.5}), (\ref{l2.2.3}) -- (\ref{l2.2.23}),
we deduce that (\ref{l2.2.1*}) -- (\ref{l2.2.5*}) hold.
Hence, (ii) is true.
\end{proof}

\begin{lemma}\label{lemma2.3}
Let Assumptions \ref{a3} and \ref{a4} hold.
Then, there exist real numbers $\rho_{3,Q}\in (0,1)$, $C_{4,Q}\in[1,\infty)$
(possibly depending on $N$) such that
\begin{align*}
	&
	\max\{\|\Phi_{\theta}^{n}(v)\|, \|\Psi_{\theta}^{n}(v)\| \}
	\leq
	C_{4,Q}\rho_{3,Q}^{n},
	\\
	&
	\max\{\|\Phi_{\theta}^{n}(v) - \Phi_{\theta'}^{n}(v) \|,
	\|\Psi_{\theta}^{n}(v) - \Phi_{\theta'}^{n}(v) \| \}
	\\
	&\leq
	C_{4,Q} \rho_{3,Q}^{n} \|\theta-\theta' \|
\end{align*}
for all $\theta,\theta'\in Q$, $v\in{\cal V}$, $n\geq 1$.
\end{lemma}

\begin{proof}
Throughout the proof, the following notation is used.
$\theta$, $\theta'$ are any elements of $Q$.
$v$ is any element of ${\cal V}$,
while $\{v_{n} \}_{n\geq 0}$ is any sequence in ${\cal V}$.
$n$ is any positive integer.

Let $\rho_{3,Q}=(1-\rho_{2,Q} )^{1/2}$,
$\tilde{C}_{1,Q}=4\rho_{3,Q}^{-2}N$, $\tilde{C}_{2,Q}=2\tilde{C}_{1,Q}C_{3,Q}^{2}$,
$\tilde{C}_{3,Q}=\rho_{3,Q}^{-2}\tilde{C}_{2,Q}C_{3,Q}$
($\rho_{2,Q}$, $C_{3,Q}$ are specified in Lemma \ref{lemma2.2}).
Owing to Lemmas \ref{lemma2.2}, \ref{lemmaa1} (see Appendix)
and (\ref{2*.5}),
we have
\begin{align}\label{l2.3.1}
 \left\|
 \tilde{A}_{\theta}^{n}(v_{0},\dots,v_{n}) C_{\theta}(v_{n})
 \right\|
 \leq&
 \tilde{C}_{1,Q} \rho_{3,Q}^{2n} \|C_{\theta}(v_{n}) \|
	\nonumber\\
 \leq&
 \tilde{C}_{1,Q} C_{3,Q} \rho_{3,Q}^{2n}
\nonumber\\
 \leq&
 \tilde{C}_{2,Q} \rho_{3,Q}^{2n}.
\end{align}
Since $A_{\theta}^{0}(v)=I$ (due to (\ref{2*.107'})),
Lemma \ref{lemma2.2} and (\ref{l2.3.1}) imply
\begin{align}\label{l2.3.3}
	\left\|
	\tilde{B}_{\theta}^{n}(v_{0},\dots,v_{n}) C_{\theta}(v_{n} )
	\right\|
	\leq&
	\left\|
	\tilde{A}_{\theta}^{n-1}(v_{1},\dots,v_{n}) C_{\theta}(v_{n})
	\right\|
	\nonumber\\
	&\cdot 
	\|B_{\theta}(v_{0},v_{1}) \|
	\nonumber\\
 \leq&
 \tilde{C}_{2,Q}C_{3,Q}\rho_{3,Q}^{2(n-1)}
	\nonumber\\
	=&
 \tilde{C}_{3,Q} \rho_{3,Q}^{2n}.
\end{align}
Moreover, due to Lemmas \ref{lemma2.2}, \ref{lemmaa1} (see Appendix \ref{appendix1}),
we have
\begin{align}\label{l2.3.5}
	&
	\left\|
	\tilde{A}_{\theta}^{n}(v_{0},\dots,v_{n}) C_{\theta}(v_{n})
	-
	\tilde{A}_{\theta'}^{n}(v_{0},\dots,v_{n}) C_{\theta'}(v_{n})
	\right\|
	\nonumber\\
	&
	\leq
	\begin{aligned}[t]
	&
	\tilde{C}_{1,Q} \rho_{3,Q}^{2n}
	\left( \|C_{\theta}(v_{n}) \| + \|C_{\theta'}(v_{n}) \| \right)
	\\
	&\cdot 
	\sum_{k=0}^{n-1}
	\|A_{\theta}(v_{k},v_{k+1}) - A_{\theta'}(v_{k},v_{k+1}) \|
	\\
	&
	+
	\tilde{C}_{1,Q}\rho_{3,Q}^{2n} \|C_{\theta}(v_{n}) - C_{\theta'}(v_{n}) \|
	\end{aligned}
	\nonumber\\
	&
	\leq
	2\tilde{C}_{1,Q} C_{3,Q}^{2} \rho_{3,Q}^{2n}(n+1) \|\theta-\theta'\|
	\nonumber\\
	&
	=
	\tilde{C}_{2,Q} \rho_{3,Q}^{2n}(n+1) \|\theta-\theta'\|.
\end{align}
Combining this with Lemma \ref{lemma2.2} and (\ref{l2.3.1}), we get
\begin{align}\label{l2.3.7}
	&
	\left\|
	\tilde{B}_{\theta}^{n}(v_{0},\dots,v_{n}) C_{\theta}(v_{n})
	-
	\tilde{B}_{\theta'}^{n}(v_{0},\dots,v_{n}) C_{\theta'}(v_{n})
	\right\|
	\nonumber\\
	&
	\leq
	\begin{aligned}[t]
	&
	\|B_{\theta}(v_{0},v_{1}) - B_{\theta'}(v_{0},v_{1})\|
	\left\|
	\tilde{A}_{\theta}^{n-1}(v_{1},\dots,v_{n}) C_{\theta}(v_{n})
	\right\|
	\\
	&+
	\!\left\|
	\tilde{A}_{\theta}^{n-1}(v_{1},\dots,v_{n}) C_{\theta}(v_{n})
	\!-\!
	\tilde{A}_{\theta'}^{n-1}(v_{1},\dots,v_{n}) C_{\theta'}(v_{n})
	\right\|
	\\
	&\cdot 
	\|B_{\theta'}(v_{0},v_{1})\|
	\end{aligned}
	\nonumber\\
	&
	\leq
	\tilde{C}_{2,Q}C_{3,Q}\rho_{3,Q}^{2(n-1)}(n+1)\|\theta-\theta'\|
	\nonumber\\
	&
	=
	\tilde{C}_{3,Q} \rho_{3,Q}^{2n}(n+1) \|\theta-\theta'\|.
\end{align}

Let $U_{\theta}^{n}(dv_{1},\dots,dv_{n}|v)$ be the conditional probability measure defined for 
$B\in{\cal B}({\cal V}^{n} )$ by
\begin{align*}
	U_{\theta}^{n}(B|v)
	=
	E\left(\left.I_{B}(V_{1}^{\theta},\dots,V_{n}^{\theta}) \right|V_{0}^{\theta}=v \right).
\end{align*}
Moreover, let $u_{\theta,\theta'}^{n}(v)$ be the function defined by
\begin{align*}
	u_{\theta,\theta'}^{n}(v)
	=
	\sup_{B\in{\cal B}({\cal V}^{n} ) }
	\left| U_{\theta}^{n}(B|v) - U_{\theta'}^{n}(B|v) \right|.
\end{align*}
Then, for $B\in{\cal B}({\cal V}^{n+1})$, we have
\begin{align*}
	U_{\theta}^{n+1}(B|v)
	\!=\!
	\iint\dots\int\!
	&
	I_{B}(v_{1},\dots,v_{n},v_{n+1})
	T_{\theta}(v_{n},dv_{n+1})
	\\
	&\cdot 
	U_{\theta}^{n}(dv_{1},\dots,dv_{n}|v).
\end{align*}
Consequently, Lemma \ref{lemma2.1} implies
\begin{align*}
	&
	\left|
	U_{\theta}^{n+1}(B|v) - U_{\theta'}^{n+1}(B|v)
	\right|
	\\
	&
	\begin{aligned}
	\leq &
	\begin{aligned}[t]
	\iint\cdots\int
	&
	I_{B}(v_{1},\dots,v_{n},v_{n+1})
	|T_{\theta}-T_{\theta'}|(v_{n},dv_{n+1})
	\\
	&\cdot 
	U_{\theta}^{n}(dv_{1},\dots,dv_{n}|v)
	\end{aligned}
	\\
	&
	+
	\begin{aligned}[t]
	\iint\cdots\int
	&
	I_{B}(v_{1},\dots,v_{n},v_{n+1})
	T_{\theta'}(v_{n},dv_{n+1})
	\\
	&\cdot 
	|U_{\theta}^{n}-U_{\theta'}^{n}|(dv_{1},\dots,dv_{n}|v)
	\end{aligned}
	\end{aligned}
	\\
	&\leq 
	C_{2,Q}\|\theta-\theta'\| + u_{\theta,\theta'}^{n}(v)
\end{align*}
($C_{2,Q}$ is specified in Lemma \ref{lemma2.1}).
For $B\in{\cal B}({\cal V})$,
Lemma \ref{lemma2.1} also yields
\begin{align*}
	\left| U_{\theta}^{1}(B|v) - U_{\theta'}^{1}(B|v) \right|
	=&
	|T_{\theta}(v,B) - T_{\theta'}(v,B) |
	\nonumber\\
	\leq&
	C_{2,Q} \|\theta-\theta'\|.
\end{align*}
Hence, we have
\begin{align}
	&
	u_{\theta,\theta'}^{1}(v)
	\leq
	C_{2,Q} \|\theta-\theta'\|,
	\nonumber\\
	&\label{l2.3.901}
	u_{\theta,\theta'}^{n+1}(v)
	\leq
	u_{\theta,\theta'}^{n}(v)
	+
	C_{2,Q} \|\theta-\theta'\|.
\end{align}
Then, iterating (\ref{l2.3.901}), we conclude
\begin{align}\label{l2.3.31}
	u_{\theta,\theta'}^{n}(v)
	\leq
	C_{2,Q}n\|\theta-\theta'\|.
\end{align}

Let $\tilde{C}_{4,Q}\in[1,\infty)$ be an upper bound of
sequence $\{n\rho_{3,Q}^{n-1} \}_{n\geq 1}$,
while $C_{4,Q}=4\tilde{C}_{3,Q}\tilde{C}_{4,Q}C_{2,Q}$.
It is straightforward to verify
\begin{align}
	&\label{l2.3.23}
	\begin{aligned}[b]
	\Phi_{\theta}^{n}(v)
	=
	\int\cdots\int
	&
	\tilde{A}_{\theta}^{n}(v,v_{1},\dots,v_{n}) C_{\theta}(v_{n})
	\\
	&\cdot 
	U_{\theta}^{n}(dv_{1},\dots,dv_{n}|v),
	\end{aligned}
	\\
	&\label{l2.3.25}
	\begin{aligned}[b]
	\Psi_{\theta}^{n}(v)
	=
	\int\cdots\int
	&
	\tilde{B}_{\theta}^{n}(v,v_{1},\dots,v_{n}) C_{\theta}(v_{n})
	\\
	&\cdot 
	U_{\theta}^{n}(dv_{1},\dots,dv_{n}|v).
	\end{aligned}
\end{align}
Combining this with (\ref{l2.3.1}), (\ref{l2.3.3}), we get
\begin{align*}
	\|\Phi_{\theta}^{n}(v) \|
	\leq &
	\begin{aligned}[t]
	\int\cdots\int
	&
	\left\|
	\tilde{A}_{\theta}^{n}(v,v_{1},\dots,v_{n}) C_{\theta}(v_{n})
	\right\|
	\\
	&\cdot 
	U_{\theta}^{n}(dv_{1},\dots,dv_{n}|v)
	\end{aligned}
	\\
	\leq &
	\tilde{C}_{2,Q}\rho_{3,Q}^{2n}
	\leq
	C_{4,Q}\rho_{3,Q}^{n},
	\\
	\|\Psi_{\theta}^{n}(v) \|
	\leq &
	\begin{aligned}[t]
	\int\cdots\int
	&
	\left\|
	\tilde{B}_{\theta}^{n}(v,v_{1},\dots,v_{n}) C_{\theta}(v_{n})
	\right\|
	\\
	&\cdot 
	U_{\theta}^{n}(dv_{1},\dots,dv_{n}|v)
	\end{aligned}
	\\
	\leq&
	\tilde{C}_{3,Q}\rho_{3,Q}^{2n}
	\leq
	C_{4,Q}\rho_{3,Q}^{n}.
\end{align*}
Moreover, (\ref{l2.3.5}), (\ref{l2.3.23}), (\ref{l2.3.31}) imply
\begin{align*}
	\|\Phi_{\theta}^{n}(v) - \Phi_{\theta'}^{n}(v) \|
	\leq &
	\begin{aligned}[t]
	\int\!\cdots\!\int
	&
	\begin{aligned}[t]
	\Big\|
	&\tilde{A}_{\theta}^{n}(v,v_{1},\dots,v_{n}) C_{\theta}(v_{n})
	\\
	&-
	\tilde{A}_{\theta'}^{n}(v,v_{1},\dots,v_{n}) C_{\theta'}(v_{n})
	\Big\|
	\end{aligned}
	\\
	&\cdot 
	U_{\theta}^{n}(dv_{1},\dots,dv_{n}|v)
	\end{aligned}
	\\
	&
	+
	\begin{aligned}[t]
	\int\!\cdots\!\int
	&
	\Big\|
	\tilde{A}_{\theta'}^{n}(v,v_{1},\dots,v_{n}) C_{\theta'}(v_{n})
	\Big\|
	\\
	&\cdot 
	|U_{\theta}^{n}-U_{\theta'}^{n}|(dv_{1},\dots,dv_{n}|v)
	\end{aligned}
	\\
	\leq &
	2\tilde{C}_{2,Q}C_{2,Q} \rho_{3,Q}^{2n} (n+1) \|\theta-\theta'\|
	\\
	\leq &
	C_{4,Q}\rho_{3,Q}^{n} \|\theta-\theta'\|.
\end{align*}
Similarly, (\ref{l2.3.7}), (\ref{l2.3.25}), (\ref{l2.3.31}) yield
\begin{align*}
	\|\Psi_{\theta}^{n}(v) - \Psi_{\theta'}^{n}(v) \|
	\leq &
	\begin{aligned}[t]
	\int\!\cdots\!\int
	&
	\begin{aligned}[t]
	\Big\|
	&\tilde{B}_{\theta}^{n}(v,v_{1},\dots,v_{n}) C_{\theta}(v_{n})
	\\
	&-
	\tilde{B}_{\theta'}^{n}(v,v_{1},\dots,v_{n}) C_{\theta'}(v_{n})
	\Big\|
	\end{aligned}
	\\
	&\cdot 
	U_{\theta}^{n}(dv_{1},\dots,dv_{n}|v)
	\end{aligned}
	\\
	&
	+
	\begin{aligned}[t]
	\int\!\cdots\!\int
	&
	\Big\|
	\tilde{B}_{\theta'}^{n}(v,v_{1},\dots,v_{n}) C_{\theta'}(v_{n})
	\Big\|
	\\
	&\cdot 
	|U_{\theta}^{n}-U_{\theta'}^{n}|(dv_{1},\dots,dv_{n}|v)
	\end{aligned}
	\\
	\leq &
	2\tilde{C}_{3,Q}C_{2,Q} \rho_{3,Q}^{2n} (n+1) \|\theta-\theta'\|
	\\
	\leq &
	C_{4,Q}\rho_{3,Q}^{n} \|\theta-\theta'\|.
\end{align*}
\end{proof}

\begin{lemma}\label{lemma2.4}
Let Assumptions \ref{a2} -- \ref{a4} hold.
Then, the following is true:

(i) $h(\theta)$ is well-defined on $\Theta$.

(ii) $h(\theta) = \lim_{n\rightarrow\infty} E\left( H(\theta,Z_{n}^{\theta} ) \right)$
for each $\theta\in\Theta$ satisfying $E(\|W_{0}^{\theta}\Lambda\|)<\infty$.

(iii) There exists a function $\tilde{H}(\theta,z)$
mapping $\theta\in\Theta$, $z\in{\cal Z}$ to $\mathbb{R}^{d}$
such that
\begin{align}\label{l2.4.1*}
	H(\theta,z)-h(\theta)
	=
	\tilde{H}(\theta,z)
	-
	(\Pi\tilde{H})(\theta,z)
\end{align}
for all $\theta\in\Theta$, $z\in{\cal Z}$.
Here, $(\Pi\tilde{H})(\theta,z)$ denotes $\int \tilde{H}(\theta,z')\Pi_{\theta}(z,dz')$.

(iv) There exists a real number $C_{5,Q}\in[1,\infty)$ (possibly depending on $N$)
such that
\begin{align*}
	&
	\max\{ \|H(\theta,z)\|, \|\tilde{H}(\theta,z)\|, \|(\Pi\tilde{H})(\theta,z)\| \}
	\\
	&\leq
	C_{5,Q}(1+\|W\Lambda\|),
	\\
	&
	\|(\Pi\tilde{H})(\theta,z) - (\Pi\tilde{H})(\theta',z) \|
	\leq
	C_{5,Q}\|\theta-\theta'\|(1+\|W\Lambda\|)
\end{align*}
for all $\theta,\theta'\in Q$, $v\in{\cal V}$, $W\in\mathbb{R}^{d\times N}$
and $z=(v,W)$
($\Lambda$ is defined in (\ref{2*.701})).
\end{lemma}

\begin{proof}
Throughout the proof, the following notation is used.
$\theta$, $\theta'$ are any elements of $Q$.
$v$, $W$ are any elements of ${\cal V}$, $\mathbb{R}^{d\times N}$ (respectively),
while $z=(v,W)$.
$n,k$ are any positive integers.

Owing to (\ref{2*.5}), we have
\begin{align*}
	e^{T} \tilde{A}_{\theta}^{n}(V_{0}^{\theta},\dots,V_{n}^{\theta} ) C_{\theta}(V_{n}^{\theta} )
	=
	e^{T} C_{\theta}(V_{n}^{\theta} )
	=
	0.
\end{align*}
Therefore, we get
\begin{align}\label{l2.4.301}
	&
	\Lambda \tilde{A}_{\theta}^{n}(V_{0}^{\theta},\dots,V_{n}^{\theta} ) C_{\theta}(V_{n}^{\theta} )
	\nonumber\\
	&=
	\tilde{A}_{\theta}^{n}(V_{0}^{\theta},\dots,V_{n}^{\theta} ) C_{\theta}(V_{n}^{\theta} )
	-
	\frac{e}{N}
	e^{T} \tilde{A}_{\theta}^{n}(V_{0}^{\theta},\dots,V_{n}^{\theta} ) C_{\theta}(V_{n}^{\theta} )
	\nonumber\\
	&=
	\tilde{A}_{\theta}^{n}(V_{0}^{\theta},\dots,V_{n}^{\theta} ) C_{\theta}(V_{n}^{\theta} ).
\end{align}
Moreover, iterating (\ref{2*.23}), it is straightforward to verify
\begin{align}\label{l2.4.301'}
	W_{n}^{\theta}
	=
	W_{0}^{\theta}
	\tilde{A}_{\theta}^{n}(V_{0}^{\theta},\dots,V_{n}^{\theta} )
	+
	\sum_{k=0}^{n-1}
	\tilde{B}_{\theta}^{n-k}(V_{k}^{\theta},\dots,V_{n}^{\theta} ).
\end{align}
Combining this with (\ref{l2.4.301}), we conclude
\begin{align*}
	H(\theta,Z_{n}^{\theta} )
	=&
	D_{\theta}(V_{n}^{\theta} )
	+
	W_{0}^{\theta}
	\tilde{A}_{\theta}^{n}(V_{0}^{\theta},\dots,V_{n}^{\theta} ) C_{\theta}(V_{n}^{\theta} )
	\\
	&
	+
	\sum_{k=0}^{n-1}
	\tilde{B}_{\theta}^{n-k}(V_{k}^{\theta},\dots,V_{n}^{\theta} ) C_{\theta}(V_{n}^{\theta} )
	\\
	=&
	D_{\theta}(V_{n}^{\theta} )
	+
	W_{0}^{\theta} \Lambda
	\tilde{A}_{\theta}^{n}(V_{0}^{\theta},\dots,V_{n}^{\theta} ) C_{\theta}(V_{n}^{\theta} )
	\\
	&
	+
	\sum_{k=0}^{n-1}
	\tilde{B}_{\theta}^{n-k}(V_{k}^{\theta},\dots,V_{n}^{\theta} ) C_{\theta}(V_{n}^{\theta} ).
\end{align*}
Consequently, we have
\begin{align}\label{l2.4.303}
	&
	(\Pi^{n}H)(\theta,z)
	=
	E\left(\left. H(\theta,Z_{n}^{\theta} ) \right|Z_{0}^{\theta}=z \right)
	\nonumber\\
	&
	\begin{aligned}[t]
	=&
	E\left(\left. D_{\theta}(V_{n}^{\theta} ) \right| V_{0}^{\theta}=v \right)
	\\
	&+
	E\left(\left.
	W\Lambda \tilde{A}_{\theta}^{n}(V_{0}^{\theta},\dots,V_{n}^{\theta} ) C_{\theta}(V_{n}^{\theta} )
	\right| V_{0}^{\theta}=v \right)
	\nonumber\\
	&+
	\sum_{k=0}^{n-1}
	\!E\!\left(\left.
	E\!\left(\left.
	\tilde{B}_{\theta}^{n-k}(V_{k}^{\theta},\dots,V_{n}^{\theta} ) C_{\theta}(V_{n}^{\theta} )
	\right| V_{k}^{\theta} \right)
	\right| V_{0}^{\theta} = v \right)
	\end{aligned}
	\nonumber\\
	&=
	(T^{n}D)_{\theta}(v)
	+
	W\Lambda\Phi_{\theta}^{n}(v)
	+
	\sum_{k=0}^{n-1}
	E\left(\left. \Psi_{\theta}^{n-k}(V_{k}^{\theta}) \right| V_{0}^{\theta}=v \right)
	\nonumber\\
	&=
	(T^{n}D)_{\theta}(v)
	+
	W\Lambda\Phi_{\theta}^{n}(v)
	+
	\sum_{k=0}^{n-1}
	(T^{k}\Psi^{n-k})_{\theta}(v).
\end{align}
Here, $(\Pi^{n}H)(\theta,z)$ denotes $\int H(\theta,z')\Pi_{\theta}^{n}(z,dz')$,
while
$(T^{n}D)_{\theta}(v)$,
$(T^{k}\Psi^{l})_{\theta}(v)$ stand for
$\int D_{\theta}(v')T_{\theta}^{n}(v,dv')$,
$\int \Psi_{\theta}^{l}(v') T_{\theta}^{k}(v,dv')$
(respectively).

Let $\beta_{Q}=\max\{\rho_{1,Q}^{1/2},\rho_{3,Q}^{1/2} \}$,
$\tilde{C}_{1,Q}=4C_{2,Q}C_{3,Q}C_{4,Q}$,
$\tilde{C}_{2,Q}=2\tilde{C}_{1,Q}(1-\beta_{Q} )^{-1}$
($\rho_{1,Q}$, $\rho_{3,Q}$, $C_{2,Q}$, $C_{3,Q}$, $C_{4,Q}$
are specified in Lemmas \ref{lemma2.1} -- \ref{lemma2.3}).
Owing to Lemma \ref{lemma2.3}, we have
\begin{align}\label{l2.4.21}
	\int \|\Psi_{\theta}^{n}(v)\| \tau_{\theta}(dv)
	\leq
	C_{4,Q}\rho_{3,Q}^{n}
	\leq
	\tilde{C}_{1,Q}\beta_{Q}^{2n}.
\end{align}
Consequently, Lemma \ref{lemma2.2} yields
\begin{align*}
	&
	\int \|D_{\theta}(v)\|\tau_{\theta}(dv)
	+
	\sum_{n=1}^{\infty}
	\int \|\Psi_{\theta}^{n}(v)\|\tau_{\theta}(dv)
	\\
	&\leq
	C_{3,Q} + \tilde{C}_{1,Q} \sum_{n=1}^{\infty} \beta_{Q}^{2n}
	\leq
	2\tilde{C}_{1,Q}(1-\beta_{Q})^{-1}
	\leq
	\tilde{C}_{2,Q}
	\!<\!\infty.
\end{align*}
Hence, $h(\theta)$ is well-defined
and satisfies $\|h(\theta)\|\leq\tilde{C}_{2,Q}$.
Since $Q$ is any compact set in $\Theta$, we conclude that (i) holds.
Moreover, using (\ref{l2.4.303}), we deduce
\begin{align}\label{l2.4.305}
	(\Pi^{n}H)(\theta,z)
	-
	h(\theta)
	=&
	(\tilde{T}^{n}D)_{\theta}(v)
	+
	W\Lambda\Phi_{\theta}^{n}(v)
	\nonumber\\
	&+
	\sum_{k=0}^{n-1}
	(\tilde{T}^{k}\Psi^{n-k})_{\theta}(v)
	\nonumber\\
	&-
	\sum_{k=n+1}^{\infty}
	\int \Psi_{\theta}^{k}(v')\tau_{\theta}(dv').
\end{align}
Here, $(\tilde{T}^{n}D)_{\theta}(v)$,
$(\tilde{T}^{k}\Psi^{l})_{\theta}(v)$ denote
$\int D_{\theta}(v')\tilde{T}_{\theta}^{n}(v,dv')$,
$\int \Psi_{\theta}^{l}(v') \tilde{T}_{\theta}^{k}(v,dv')$
(respectively).

Let $\tilde{C}_{3,Q}\in[1,\infty)$ be an upper bound of
sequence $\{n\beta_{Q}^{n-1} \}_{n\geq 1}$,
while $\tilde{C}_{4,Q}=2\tilde{C}_{2,Q}\tilde{C}_{3,Q}(1-\beta_{Q} )^{-1}$,
$C_{5,Q}=\tilde{C}_{4,Q}(1-\beta_{Q} )^{-1}$.
Owing to Lemmas \ref{lemma2.1}, \ref{lemma2.2}, we have
\begin{align}\label{l2.4.25}
	\|(\tilde{T}^{n}D)_{\theta}(v) \|
	\leq&
	\int \|D_{\theta}(v')\| |\tilde{T}_{\theta}^{n}|(v,dv')
	\nonumber\\
	\leq&
	C_{2,Q}C_{3,Q}\rho_{1,Q}^{n}
	\leq
	\tilde{C}_{1,Q} \beta_{Q}^{2n}.
\end{align}
Similarly, due to Lemmas \ref{lemma2.1}, \ref{lemma2.3},
we have
\begin{align}\label{l2.4.27}
	\|(\tilde{T}^{n-k}\Psi^{k})_{\theta}(v) \|
	\leq&
	\int \|\Psi_{\theta}^{k}(v')\| |\tilde{T}_{\theta}^{n-k}|(v,dv')
	\nonumber\\
	\leq&
	C_{2,Q} C_{4,Q} \rho_{1,Q}^{n-k} \rho_{3,Q}^{k}
	\leq
	\tilde{C}_{1,Q} \beta_{Q}^{2n}.
\end{align}
Combining Lemma \ref{lemma2.3} and (\ref{l2.4.21}), (\ref{l2.4.305}) -- (\ref{l2.4.27}),
we get
\begin{align}\label{l2.4.501}
	\|(\Pi^{n}H)(\theta,z)-h(\theta) \|
	\leq &
	\|(\tilde{T}^{n}D)_{\theta}(v)\|
	+
	\|\Phi_{\theta}^{n}(v)\| \|W\Lambda\|
	\nonumber\\
	&+
	\sum_{k=1}^{n}
	\|(\tilde{T}^{n-k}\Psi^{k})_{\theta}(v) \|
	\nonumber\\
	&+
	\sum_{k=n+1}^{\infty}
	\left\|
	\Psi_{\theta}^{k}(v')
	\right\|
	\tau_{\theta}(dv')
	\nonumber\\
	\leq &
	\tilde{C}_{1,Q}\beta_{Q}^{2n} (n+1)
	\!+\!
	C_{4,Q}\rho_{3,Q}^{n} \|W\Lambda\|
	\nonumber\\
	&+
	\tilde{C}_{1,Q} \sum_{k=n+1}^{\infty} \beta_{Q}^{2k}
	\nonumber\\
	\leq &
	\tilde{C}_{1,Q}\beta_{Q}^{2n} (n+1) (1 + \|W\Lambda\| )
	\nonumber\\
	&+
	\tilde{C}_{1,Q}\beta_{Q}^{2n}(1-\beta_{Q})^{-1}
	\nonumber\\
	\leq &
	\tilde{C}_{4,Q}\beta_{Q}^{n} (1+\|W\Lambda\| ).
\end{align}
Since $\|h(\theta)\|\leq\tilde{C}_{2,Q}$ and
$\Lambda C_{\theta}(v)
= C_{\theta}(v)$ (due to (\ref{2*.5})),
Lemma \ref{lemma2.2} yields
\begin{align*}
	\|(\Pi^{0}H)(\theta,z)\!-\!h(\theta) \|
	\leq &
	\|C_{\theta}(v)\| \|W\Lambda\| \!+\! \|D_{\theta}(v)\| \!+\! \|h(\theta)\|
	\\
	\leq &
	C_{3,Q}(1+\|W\Lambda\|) + \tilde{C}_{2,Q}
	\\
	\leq &
	\tilde{C}_{4,Q}(1+\|W\Lambda\|).
\end{align*}
Hence, we have
\begin{align}\label{l2.4.29}
	\sum_{n=0}^{\infty}
	\|(\Pi^{n}H)(\theta,z)-h(\theta) \|
	\leq&
	\tilde{C}_{4,Q}(1+\|W\Lambda\|)
	\sum_{n=0}^{\infty}\beta_{Q}^{n}
	\nonumber\\
	\leq&
	C_{5,Q}(1+\|W\Lambda\|).
\end{align}

Owing to Lemmas \ref{lemma2.1}, \ref{lemma2.2}, we have
\begin{align}\label{l2.4.31}
	&
	\|(\tilde{T}^{n}D)_{\theta}(v) - (\tilde{T}^{n}D)_{\theta'}(v) \|
	\nonumber\\
	&
	\begin{aligned}
	\leq &
	\int \|D_{\theta}(v')-D_{\theta'}(v') \| |\tilde{T}_{\theta}^{n}|(v,dv')
	\\
	&+
	\int \|D_{\theta'}(v') \| |\tilde{T}_{\theta}^{n}-\tilde{T}_{\theta'}^{n} |(v,dv')
	\end{aligned}
	\nonumber\\
	&\leq 
	2C_{2,Q}C_{3,Q}\rho_{1,Q}^{n}\|\theta-\theta'\|
	\leq 
	\tilde{C}_{1,Q}\beta_{Q}^{2n} \|\theta-\theta'\|.
\end{align}
Moreover, Lemmas \ref{lemma2.1}, \ref{lemma2.3} imply
\begin{align}\label{l2.4.33}
	&
	\|(\tilde{T}^{n-k}\Psi^{k})_{\theta}(v) - (\tilde{T}^{n-k}\Psi^{k})_{\theta'}(v) \|
	\nonumber\\
	&\begin{aligned}
	\leq &
	\int \|\Psi_{\theta}^{k}(v') - \Psi_{\theta'}^{k}(v') \| |\tilde{T}_{\theta}^{n-k}|(v,dv')
	\nonumber\\
	&+
	\int \|\Psi_{\theta'}^{k}(v') \| |\tilde{T}_{\theta}^{n-k}-\tilde{T}_{\theta'}^{n-k} |(v,dv')
	\end{aligned}
	\nonumber\\
	&\leq 
	2C_{2,Q}C_{4,Q}\rho_{1,Q}^{n-k}\rho_{3,Q}^{k} \|\theta-\theta'\|
	\leq 
	\tilde{C}_{1,Q} \beta_{Q}^{2n} \|\theta-\theta'\|
\end{align}
for $n\geq k$.
The same lemmas also yield
\begin{align}\label{l2.4.23}
	&
	\left\|
	\int \Psi_{\theta}^{n}(v) \tau_{\theta}(dv)
	\!-\!
	\int \Psi_{\theta'}^{n}(v) \tau_{\theta'}(dv)
	\right\|
	\nonumber\\
	&
	\begin{aligned}
	\leq &
	\int \|\Psi_{\theta}^{n}(v) - \Psi_{\theta'}^{n}(v) \| \tau_{\theta}(dv)
	+
	\int \|\Psi_{\theta'}^{n}(v) \| |\tau_{\theta}-\tau_{\theta'} |(dv)
	\end{aligned}
	\nonumber\\
	&\leq 
	2 C_{2,Q} C_{4,Q} \rho_{3,Q}^{n} \|\theta-\theta'\|
	\leq 
	\tilde{C}_{1,Q} \beta_{Q}^{2n} \|\theta-\theta'\|.
\end{align}
Combining Lemma \ref{lemma2.3} and (\ref{l2.4.305}), (\ref{l2.4.31}) -- (\ref{l2.4.23}),
we get
\begin{align}\label{l2.4.23'}
	&
	\|((\Pi^{n}H)(\theta,z)-h(\theta) ) \!-\! ((\Pi^{n}H)(\theta',z)-h(\theta') )\|
	\nonumber\\
	&
	\leq
	\begin{aligned}[t]
	&
	\|(\tilde{T}^{n}D)_{\theta}(v) \!-\! (\tilde{T}^{n}D)_{\theta'}(v) \|
	\!+\!
	\|\Phi_{\theta}^{n}(v) - \Phi_{\theta'}^{n}(v) \| \|W\Lambda\|
	\\
	&
	+
	\sum_{k=1}^{n}
	\|(\tilde{T}^{n-k}\Psi^{k})_{\theta}(v) - (\tilde{T}^{n-k}\Psi^{k})_{\theta'}(v) \|
	\\
	&+
	\sum_{k=n+1}^{\infty}
	\left\|
	\int \Psi_{\theta}^{k}(v')\tau_{\theta}(dv')
	-
	\int \Psi_{\theta'}^{k}(v')\tau_{\theta'}(dv')
	\right\|
	\end{aligned}
	\nonumber\\
	&
	\begin{aligned}
	\leq&
	\tilde{C}_{1,Q}\beta_{Q}^{2n}(n+1) \|\theta-\theta'\|
	\!+\!
	C_{4,Q}\rho_{3,Q}^{n}\|W\Lambda\|\|\theta-\theta'\|
	\\
	&+
	\tilde{C}_{1,Q}\|\theta-\theta'\|
	\sum_{k=n+1}^{\infty} \beta_{Q}^{2k}
	\end{aligned}
	\nonumber\\
	&
	\begin{aligned}
	\leq&
	\tilde{C}_{1,Q}\beta_{Q}^{2n}(n+1) (1+\|W\Lambda\| ) \|\theta-\theta'\|
	\\
	&+
	\tilde{C}_{1,Q}\beta_{Q}^{2n}(1-\beta_{Q} )^{-1} \|\theta-\theta'\|
	\end{aligned}
	\nonumber\\
	&\leq
	\tilde{C}_{4,Q} \beta_{Q}^{n} \|\theta-\theta'\|(1+\|W\Lambda\|).
\end{align}
Hence, we have
\begin{align}\label{l2.4.35}
	&
	\sum_{n=1}^{\infty}
	\|((\Pi^{n}H)(\theta,z)-h(\theta) ) - ((\Pi^{n}H)(\theta',z)-h(\theta') ) \|
	\nonumber\\
	&\leq 
	\tilde{C}_{4,Q} \|\theta-\theta'\|(1+\|W\Lambda\|)
	\sum_{n=1}^{\infty} \beta_{Q}^{n}
	\nonumber\\
	&\leq 
	C_{5,Q}\|\theta-\theta'\|(1+\|W\Lambda\|).
\end{align}

Let $\tilde{H}(\theta,z)$ be the function defined by
\begin{align*}
	\tilde{H}(\theta,z) = \sum_{n=0}^{\infty} ((\Pi^{n}H)(\theta,z)-h(\theta) ).
\end{align*}
Then, (\ref{l2.4.29}) implies that
$\tilde{H}(\theta,z)$, $(\Pi\tilde{H})(\theta,z)$ are  well-defined
and satisfy
\begin{align*}
	(\Pi\tilde{H})(\theta,z) = \sum_{n=1}^{\infty} ((\Pi^{n}H)(\theta,z)-h(\theta) ).
\end{align*}
Consequently, (\ref{l2.4.1*}) holds.
Since $Q$ is any compact set in $\Theta$,
we conclude that (iii) also holds.
Moreover, using (\ref{l2.4.29}), (\ref{l2.4.35}),
we deduce that (iv) is also true.

When $E(\|W_{0}^{\theta}\Lambda\|)<\infty$,
(\ref{l2.4.501}) implies
\begin{align*}
	\left\|E(H(\theta,Z_{n}^{\theta} ) ) - h(\theta) \right\|
	=&
	\left\| E\left( (\Pi^{n}H)(\theta,Z_{0}^{\theta}) - h(\theta) \right) \right\|
	\\
	\leq &
	E\left( \left\| (\Pi^{n}H)(\theta,Z_{0}^{\theta}) - h(\theta) \right\| \right)
	\\
	\leq &
	\tilde{C}_{4,Q}\beta_{Q}^{n}
	(1 + E(\|W_{0}^{\theta}\Lambda \| ) ).
\end{align*}
Therefore,
$
	h(\theta)=\lim_{n\rightarrow\infty} E(H(\theta,Z_{n}^{\theta} ) )
$
if $E(\|W_{0}^{\theta}\Lambda\|)<\infty$.
As $Q$ is any compact set in $\Theta$, we conclude that (ii) holds.
\end{proof}

\begin{lemma}\label{lemma3.1}
Let Assumptions \ref{a3} and \ref{a4} hold.
Then, there exists a real number $C_{6,Q}\in[1,\infty)$ (possibly depending on $N$)
such that
\begin{align*}
	\|W_{n}\Lambda\|
	I_{ \{\tau_{Q}\geq n\} }
	\leq
	C_{6,Q}(1+\|W_{0}\Lambda\| )
\end{align*}
for $n\geq 1$, where $\tau_{Q}$ is the stopping time defined by
\begin{align*}
	\tau_{Q}
	=
	\inf\left(\{n\geq 0:\theta_{n}\not\in Q \} \cup \{\infty \} \right)
\end{align*}
($\Lambda$ is specified in (\ref{2*.701})).
\end{lemma}

\begin{proof}
Throughout the proof, the following notation is used.
$n$ is any positive integer.
$A_{n}$ and $B_{n}$ are the random matrices defined by
\begin{align*}
	A_{n} = A_{\theta_{n-1} }(V_{n-1}, V_{n} ),
	\;\;\;\;\;
	B_{n} = B_{\theta_{n-1} }(V_{n-1}, V_{n} ).
\end{align*}
$A_{k,k}$ and $A_{k,l}$ are the random matrices defined by
\begin{align*}
	A_{k,k} = I, \;\;\;\;\; A_{k,l} = A_{k+1} \cdots A_{l}
\end{align*}
for $l>k\geq 0$.
Then, iterating (\ref{2*.1}), we get
\begin{align}\label{l3.1.151}
	W_{n} = W_{0} A_{0,n} + \sum_{j=1}^{n} B_{j} A_{j,n}.
\end{align}
Moreover, (\ref{2*.5}) implies
$
	e^{T} A_{k,l} \Lambda = e^{T}\Lambda = 0
$
for $l\geq k\geq 0$. Consequently, we have
\begin{align*}
	\Lambda A_{k,l} \Lambda
	=
	A_{k,l} \Lambda
	-
	\frac{e}{N} e^{T} A_{k,l} \Lambda
	=
	A_{k,l} \Lambda.
\end{align*}
Combining this with (\ref{l3.1.151}), we get
\begin{align}\label{l3.1.153}
	W_{n}\Lambda
	=&
	W_{0}A_{0,n}\Lambda
	+
	\sum_{j=1}^{n} B_{j} A_{j,n} \Lambda
	\nonumber\\
	=&
	W_{0}\Lambda A_{0,n}\Lambda
	+
	\sum_{j=1}^{n} B_{k} A_{j,n} \Lambda.
\end{align}

Let $\beta_{Q}=1-\rho_{2,Q}$,
$\tilde{C}_{1,Q}=4\beta_{Q}^{-1}N$,
$\tilde{C}_{2,Q}=\tilde{C}_{1,Q}C_{3,Q}$,
$C_{6,Q}=\tilde{C}_{2,Q}(1-\beta_{Q} )^{-1}$
($\rho_{2,Q}$, $C_{3,Q}$ are specified in Lemma \ref{lemma2.2}).
Since $\theta_{0},\dots,\theta_{n-1}\in Q$ on $\{\tau_{Q}\geq n\}$,
Lemmas \ref{lemma2.2}, \ref{lemmaa1} (see Appendix \ref{appendix1}) and (\ref{2*.5}) imply
\begin{align}\label{l3.1.1}
	&
	\left\|
	A_{k,n}\Lambda
	\right\|
	I_{ \{\tau_{Q}\geq n\} }
	\nonumber\\
	&=
	\left\|
	A_{\theta_{k}}(V_{k},V_{k+1}) \cdots A_{\theta_{n-1}}(V_{n-1},V_{n}) \Lambda
	\right\|
	I_{ \{\tau_{Q}\geq n\} }
	\nonumber\\
	&\leq
	\tilde{C}_{1,Q}\beta_{Q}^{n-k}
\end{align}
for $n>k\geq 0$.
Consequently, Lemma \ref{lemma2.2} yields
\begin{align*}
	&
	\left\|B_{k} A_{k,n} \Lambda \right\| I_{ \{\tau_{Q}\geq n \} }
	\\
	&=
	\left\|B_{\theta_{k-1} }(V_{k-1},V_{k} ) A_{k,n} \Lambda \right\| I_{ \{\tau_{Q}\geq n \} }
	\\
	&\leq 
	\left\|B_{\theta_{k-1} }(V_{k-1},V_{k} ) \right\|
	\left\|A_{k,n} \Lambda \right\| I_{ \{\tau_{Q}\geq n \} }
	\\
	&\leq 
	\tilde{C}_{1,Q}C_{3,Q}\beta_{Q}^{n-k}
	\leq 
	\tilde{C}_{2,Q}\beta_{Q}^{n-k}
\end{align*}
for $n\geq k\geq 1$.
Combining this with (\ref{l3.1.153}), we get
\begin{align*}
	\|W_{n}\Lambda \| I_{ \{\tau_{Q}\geq n \} }
	\leq &
	\|W_{0}\Lambda \| \|A_{0,n}\Lambda \| I_{ \{\tau_{Q}\geq n \} }
	\\
	&+
	\sum_{j=1}^{n} \|B_{j}A_{j,n}\Lambda\| I_{ \{\tau_{Q}\geq n \} }
	\\
	\leq &
	\tilde{C}_{1,Q}\beta_{Q}^{n} \|W_{0}\Lambda\|
	+
	\tilde{C}_{2,Q} \sum_{j=1}^{n} \beta_{Q}^{n-j}
	\\
	\leq &
	\tilde{C}_{2,Q}(1-\beta_{Q} )^{-1} (1 + \|W_{0}\Lambda \| )
	\\
	=&
	C_{6,Q} (1 + \|W_{0}\Lambda \| ).
\end{align*}
\end{proof}

\begin{lemma}\label{lemma2.5}
Let Assumptions \ref{a2} -- \ref{a4} hold.
Then, there exists a real number $M_{Q}\in[1,\infty)$
(independent of $N$ and depending only on $p_{\theta}(x'|x)$,
$q_{\theta}(y|x)$) such that
\begin{align}\label{l2.5.1*}
	\|h(\theta) - \nabla l(\theta) \|
	\leq
	\frac{M_{Q} }{N}
\end{align}
for all $\theta\in Q$.
\end{lemma}

\begin{proof}
Throughout the proof, the following notation is used.
$H'_{\theta,y}(\xi,\zeta)$ and $H''_{\theta,y}(\xi,\zeta)$ are the functions defined by
\begin{align*}
	&
	H'_{\theta,y}(\xi,\zeta)
	=
	\frac{\int \nabla_{\theta} q_{\theta}(y|x) \xi(dx) }
	{\int q_{\theta}(y|x) \xi(dx) },
	\\
	&
	H''_{\theta,y}(\xi,\zeta)
	=
	\frac{\int q_{\theta}(y|x) \zeta(dx) }
	{\int q_{\theta}(y|x) \xi(dx) }
\end{align*}
for $\theta\in\Theta$, $y\in{\cal Y}$,
$\xi\in{\cal P}({\cal X})$, $\zeta\in{\cal M}_{s}^{d}({\cal X})$.
$A_{1,n}^{\theta}$, $A_{2,n}^{\theta}$ and $A_{3,n}^{\theta}$
are the random variables defined by
\begin{align*}
	&
	A_{1,n}^{\theta}
	=
	\int q_{\theta}(Y_{n}|x) \xi_{n}^{\theta}(dx),
	\\
	&
	A_{2,n}^{\theta}
	=
	\int q_{\theta}(Y_{n}|x) \tilde{\xi}_{n}^{\theta}(dx),
	\\
	&
	A_{3,n}^{\theta}
	=
	\int q_{\theta}(Y_{n}|x) F_{\theta,\boldsymbol Y}^{0:n}(dx|\xi_{0}^{\theta} )
\end{align*}
for $n\geq 0$.
$B_{1,n}^{\theta}$, $B_{2,n}^{\theta}$ and $B_{3,n}^{\theta}$
are the random variables defined by
\begin{align*}
	&
	B_{1,n}^{\theta}
	=
	\int \nabla_{\theta} q_{\theta}(Y_{n}|x) \xi_{n}^{\theta}(dx),
	\\
	&
	B_{2,n}^{\theta}
	=
	\int \nabla_{\theta} q_{\theta}(Y_{n}|x) \tilde{\xi}_{n}^{\theta}(dx),
	\\
	&
	B_{3,n}^{\theta}
	=
	\int \nabla_{\theta} q_{\theta}(Y_{n}|x) F_{\theta,\boldsymbol Y}^{0:n}(dx|\xi_{0}^{\theta} ).
\end{align*}
$C_{1,n}^{\theta}$, $C_{2,n}^{\theta}$ and $C_{3,n}^{\theta}$
are the random variables defined by
\begin{align*}
	&
	C_{1,n}^{\theta}
	=
	\int q_{\theta}(Y_{n}|x) \zeta_{n}^{\theta}(dx),
	\\
	&
	C_{2,n}^{\theta}
	=
	\int q_{\theta}(Y_{n}|x) \tilde{\zeta}_{n}^{\theta}(dx),
	\\
	&
	C_{3,n}^{\theta}
	=
	\int q_{\theta}(Y_{n}|x) G_{\theta,\boldsymbol Y}^{0:n}(dx|\xi_{0}^{\theta},\zeta_{0} ).
\end{align*}
For $1\leq i\leq N$,
$W_{n,i}^{\theta}$ is the $i$-th column of $W_{n}^{\theta}$.
$\xi_{n}^{\theta}(dx)$ and $\zeta_{n}^{\theta}(dx)$ are
the (empirical) measures defined by
\begin{align}
	&\label{l2.5.751.a}
	\xi_{n}^{\theta}(B)
	=
	\frac{1}{N} \sum_{i=1}^{N} \delta_{\hat{X}_{n,i}^{\theta} }(B),
	\\
	&\label{l2.5.751.b}
	\zeta_{n}^{\theta}(B)
	=
	\frac{1}{N} \sum_{i=1}^{N}
	\left( W_{n,i}^{\theta} - \frac{1}{N} \sum_{j=1}^{N} W_{n,j}^{\theta} \right)
	\delta_{\hat{X}_{n,i}^{\theta}}(B)
\end{align}
for $B\in{\cal B}({\cal X})$.
$\tilde{\xi}_{n}^{\theta}(dx)$ and
$\tilde{\zeta}_{n}^{\theta}(dx)$
are the (random) measures defined by
\begin{align*}
	&
	\tilde{\xi}_{n}^{\theta}(B)
	=
	\xi_{n}^{\theta}(B)
	-
	F_{\theta,\boldsymbol Y}^{0:n}(B|\xi_{0}^{\theta} ),
	\\
	&
	\tilde{\zeta}_{n}^{\theta}(B)
	=
	\zeta_{n}^{\theta}(B)
	-
	G_{\theta,\boldsymbol Y}^{0:n}(B|\xi_{0}^{\theta},\zeta_{0}^{\theta} )
\end{align*}
($\boldsymbol Y$, $F_{\theta,\boldsymbol y}^{0:n}(dx|\xi)$,
$G_{\theta,\boldsymbol y}^{0:n}(dx|\xi,\zeta)$
are defined in the statement of Lemma \ref{lemma1.3} and (\ref{1.401*}), (\ref{1.403*})).
Throughout the proof, we assume (without loss of generality)
that $\hat{X}_{0}^{\theta}=\hat{x}_{0}$, $W_{0}^{\theta}=\boldsymbol 0$
for each $\theta\in\Theta$,
where $\hat{x}_{0}\in{\cal X}^{N}$ is a deterministic vector
and $\boldsymbol 0$ is the $d\times N$ zero matrix.
Consequently, $\xi_{0}^{\theta}(dx)$, $\zeta_{0}^{\theta}(dx)$
are deterministic, independent of $\theta$ and satisfy $\|\zeta_{0}^{\theta} \|=0$.

In the rest of the proof,
let $\theta$, $y$ be any elements of $Q$, ${\cal Y}$ (respectively),
while $\xi$, $\zeta$ are any elements of ${\cal P}({\cal X})$, ${\cal M}_{s}^{d}({\cal X})$.
Moreover, let $n$ be any non-negative integer,
while $\varphi:{\cal X}\rightarrow[-1,1]$ is any function.
Then, relying on \cite[Theorem 2.1, Proposition 5.1]{tadic&doucet4}, we conclude that
there exists a real number $\tilde{C}_{1,Q}\in[1,\infty)$
(independent of $N$ and depending only on $p_{\theta}(x'|x)$,
$q_{\theta}(y|x)$) such that
$\|\zeta_{n}^{\theta} \| \leq \tilde{C}_{1,Q}$ and
\begin{align}
	&\label{l2.5.753.a}
	\left|
	E\left(\left.
	\int \varphi(x)
	\tilde{\xi}_{n}^{\theta}(dx)
	\right|\boldsymbol Y
	\right)
	\right|
	\leq
	\frac{\tilde{C}_{1,Q}}{N},
	\\
	&\label{l2.5.753.b}
	\left\|
	E\left(\left.
	\int \varphi(x)
	\tilde{\zeta}_{n}^{\theta}(dx)
	\right|\boldsymbol Y
	\right)
	\right\|
	\leq
	\frac{\tilde{C}_{1,Q}}{N}
\end{align}
almost surely.
Similarly, using \cite[Theorem 3.1]{delmoral&doucet&singh}, \cite[Theorem 5.8]{legland&oudjane}
(or \cite[Proposition 6.4]{tadic&doucet4}),
we deduce that there exists a real number $\tilde{C}_{2,Q}\in[1,\infty)$
(independent of $N$ and depending only on $p_{\theta}(x'|x)$,
$q_{\theta}(y|x)$) such that
\begin{align}
	&\label{l2.5.755.a}
	E\left(\left.
	\left|
	\int \varphi(x)
	\tilde{\xi}_{n}^{\theta}(dx)
	\right|^{2}
	\right|\boldsymbol Y
	\right)
	\leq
	\frac{\tilde{C}_{2,Q}}{N},
	\\
	&\label{l2.5.755.b}
	E\left(\left.
	\left\|
	\int \varphi(x)
	\tilde{\zeta}_{n}^{\theta}(dx)
	\right\|^{2}
	\right|\boldsymbol Y
	\right)
	\leq
	\frac{\tilde{C}_{2,Q}}{N}
\end{align}
almost surely.

It is straightforward to verify
\begin{align}\label{l2.5.1}
	H_{\theta,y}(\xi,\zeta)
	=
	H'_{\theta,y}(\xi,\zeta) + H''_{\theta,y}(\xi,\zeta)
\end{align}
($H_{\theta,y}(\xi,\zeta)$ is defined in (\ref{1.3*b})).
It is also easy to show
\begin{align}\label{l2.5.757}
	\frac{1}{A_{1,n}^{\theta} }
	=&
	\frac{1}{A_{3,n}^{\theta} }
	-
	\frac{A_{2,n}^{\theta} }{|A_{3,n}^{\theta} |^{2} }
	+
	\frac{1}{A_{1,n}^{\theta} }
	\left|
	\frac{A_{2,n}^{\theta} }{A_{3,n}^{\theta} }
	\right|^{2}.
\end{align}
Relying on (\ref{l2.5.1}), we conclude
\begin{align}\label{l2.5.3}
	H(\theta,Z_{n}^{\theta} )
	=&
	\frac{\sum_{i=1}^{N} 
	q_{\theta}(Y_{n}|\hat{X}_{n,i}^{\theta} )
	\left( W_{n,i}^{\theta} - N^{-1} \sum_{j=1}^{N} W_{n,j}^{\theta} \right) }
	{\sum_{i=1}^{N} q_{\theta}(Y_{n}|\hat{X}_{n,i}^{\theta} ) }
	\nonumber\\
	&+
	\frac{\sum_{i=1}^{N}
	\nabla_{\theta} q_{\theta}(Y_{n}|\hat{X}_{n,i}^{\theta} ) }
	{\sum_{i=1}^{N} q_{\theta}(Y_{n}|\hat{X}_{n,i}^{\theta} ) }
	\nonumber\\
	=&
	\frac{\int q_{\theta}(Y_{n}|x) \zeta_{n}^{\theta}(dx)
	+
	\int \nabla_{\theta} q_{\theta}(Y_{n}|x) \xi_{n}^{\theta}(dx) }
	{\int q_{\theta}(Y_{n}|x) \xi_{n}^{\theta}(dx) }
	\nonumber\\
	=&
	H_{\theta,Y_{n} }(\xi_{n}^{\theta}, \zeta_{n}^{\theta} ).
\end{align}
Using (\ref{l2.5.757}), we also deduce
\begin{align}
	&\label{l2.5.7}
	H'_{\theta,Y_{n} }(\xi_{n}^{\theta},\zeta_{n}^{\theta} )
	-
	H'_{\theta,Y_{n} }(F_{\theta,\boldsymbol Y}^{0:n}(\xi_{0}^{\theta} ),
	G_{\theta,\boldsymbol Y}^{0:n}(\xi_{0}^{\theta}, \zeta_{0}^{\theta} ) )
	\nonumber\\
	&=
	\frac{B_{2,n}^{\theta} }{A_{3,n}^{\theta} }
	-
	\frac{A_{2,n}^{\theta} B_{2,n}^{\theta} }{|A_{3,n}^{\theta} |^{2} }
	-
	\frac{A_{2,n}^{\theta} B_{3,n}^{\theta} }{|A_{3,n}^{\theta} |^{2} }
	+
	\frac{B_{1,n}^{\theta} }{A_{1,n}^{\theta} }
	\left|
	\frac{A_{2,n}^{\theta} }{A_{3,n}^{\theta} }
	\right|^{2},
	\\
	&\label{l2.5.5}
	H''_{\theta,Y_{n} }(\xi_{n}^{\theta},\zeta_{n}^{\theta} )
	-
	H''_{\theta,Y_{n} }(F_{\theta,\boldsymbol Y}^{0:n}(\xi_{0}^{\theta} ),
	G_{\theta,\boldsymbol Y}^{0:n}(\xi_{0}^{\theta}, \zeta_{0}^{\theta} ) )
	\nonumber\\
	&=
	\frac{C_{2,n}^{\theta} }{A_{3,n}^{\theta} }
	-
	\frac{A_{2,n}^{\theta} C_{2,n}^{\theta} }{|A_{3,n}^{\theta} |^{2} }
	-
	\frac{A_{2,n}^{\theta} C_{3,n}^{\theta} }{|A_{3,n}^{\theta} |^{2} }
	+
	\frac{C_{1,n}^{\theta} }{A_{1,n}^{\theta} }
	\left|
	\frac{A_{2,n}^{\theta} }{A_{3,n}^{\theta} }
	\right|^{2}.
\end{align}
Moreover, due to Assumption \ref{a3}, we have
\begin{align}
	&\label{l2.5.51.a}
	A_{1,n}^{\theta}
	=
	\int q_{\theta}(Y_{n}|x) \xi_{n}^{\theta}(dx)
	\geq
	\varepsilon_{Q},
	\\
	&\label{l2.5.51.b}
	A_{3,n}^{\theta}
	=
	\int q_{\theta}(Y_{n}|x) F_{\theta,\boldsymbol Y}^{0:n}(dx|\xi_{0}^{\theta} )
	\geq
	\varepsilon_{Q}.
\end{align}
Similarly, owing to Assumption \ref{a4}, we have
\begin{align}
	&\label{l2.5.53.a}
	\|B_{1,n}^{\theta} \|
	\leq
	\int \|\nabla_{\theta} q_{\theta}(Y_{n}|x) \| \xi_{n}^{\theta}(dx)
	\leq
	K_{1,Q},
	\\
	&\label{l2.5.53.b}
	\|B_{3,n}^{\theta} \|
	\leq
	\int \|\nabla_{\theta} q_{\theta}(Y_{n}|x) \|
	F_{\theta,\boldsymbol Y}^{0:n}(dx|\xi_{0}^{\theta} )
	\leq
	K_{1,Q}.
\end{align}
Since $\|\zeta_{n}^{\theta}\|\leq\tilde{C}_{1,Q}$,
Assumption \ref{a3} and Lemma \ref{lemma1.3} yield
\begin{align}
	&\label{l2.5.55.a}
	\left\|C_{1,n}^{\theta} \right\|
	\leq
	\int q_{\theta}(Y_{n}|x) |\zeta_{n}^{\theta}|(dx)
	\leq
	\frac{\tilde{C}_{1,Q} }{\varepsilon_{Q} },
	\\
	&\label{l2.5.55.b}
	\|C_{3,n}^{\theta} \|
	\leq
	\int q_{\theta}(Y_{n}|x)
	\big|G_{\theta,\boldsymbol Y}^{0:n}\big|(dx|\xi_{0}^{\theta},\zeta_{0}^{\theta} )
	\leq
	\frac{C_{1,Q} }{\varepsilon_{Q} }.
\end{align}

Let $\tilde{C}_{3,Q}=\varepsilon_{Q}^{-4}\sqrt{d}\tilde{C}_{1,Q}C_{1,Q}K_{1,Q}$.
Due to Assumptions \ref{a3}, \ref{a4}, we have
\begin{align*}
	0\leq \varepsilon_{Q} q_{\theta}(Y_{n}|x) \leq 1,
	\;\;\;\;\;
	\left\|\frac{\nabla_{\theta} q_{\theta}(Y_{n}|x) }{K_{1,Q} } \right\| \leq 1.
\end{align*}
Then, using (\ref{l2.5.753.a}), (\ref{l2.5.753.b}), we conclude
\begin{align}
	&\label{l2.5.151}
	\begin{aligned}[b]
	\left| E\left(\left. A_{2,n}^{\theta} \right|\boldsymbol Y\right)\right|
	=&
	\left|
	E\left(\left.
	\int q_{\theta}(Y_{n}|x)
	\tilde{\xi}_{n}^{\theta}(dx)
	\right|\boldsymbol Y
	\right)
	\right|
	\\
	\leq&
	\frac{\tilde{C}_{1,Q}}{\varepsilon_{Q} N },
	\end{aligned}
	\\
	&\label{l2.5.153}
	\begin{aligned}[b]
	\left\| E\left(\left. B_{2,n}^{\theta} \right|\boldsymbol Y\right)\right\|
	=&
	\left\|
	E\left(\left.
	\int \nabla_{\theta} q_{\theta}(Y_{n}|x)
	\tilde{\xi}_{n}^{\theta}(dx)
	\right|\boldsymbol Y
	\right)
	\right\|
	\\
	\leq&
	\frac{\sqrt{d} \tilde{C}_{1,Q}K_{1,Q}}{N},
	\end{aligned}
	\\
	&\label{l2.5.155}
	\begin{aligned}[b]
	\left\| E\left(\left. C_{2,n}^{\theta} \right|\boldsymbol Y\right)\right\|
	=&
	\left\|
	E\left(\left.
	\int q_{\theta}(Y_{n}|x)
	\tilde{\zeta}_{n}^{\theta}(dx)
	\right|\boldsymbol Y
	\right)
	\right\|
	\\
	\leq&
	\frac{\tilde{C}_{1,Q}}{\varepsilon_{Q} N }
	\end{aligned}
\end{align}
almost surely.
As $A_{3,n}^{\theta}$ is measurable with respect to $\boldsymbol Y$,
(\ref{l2.5.51.b}), (\ref{l2.5.153}), (\ref{l2.5.155}) imply
\begin{align}
	&\label{l2.5.157}
	\begin{aligned}[b]
	\left\|
	E\left(
	\frac{B_{2,n}^{\theta} }{A_{3,n}^{\theta} }
	\right)
	\right\|
	\leq&
	E\left(
	\frac{\left\|E\left(\left.
	B_{2,n}^{\theta} \right|\boldsymbol Y \right)\right\| }
	{A_{3,n}^{\theta} }
	\right)
	\\
	\leq&
	\frac{\sqrt{d}\tilde{C}_{1,Q}K_{1,Q} }{\varepsilon_{Q} N }
	\leq
	\frac{\tilde{C}_{3,Q} }{N},
	\end{aligned}
	\\
	&\label{l2.5.159}
	\begin{aligned}[b]
	\left\|
	E\left(
	\frac{C_{2,n}^{\theta} }{A_{3,n}^{\theta} }
	\right)
	\right\|
	\leq&
	E\left(
	\frac{\left\|E\left(\left.
	C_{2,n}^{\theta} \right|\boldsymbol Y \right)\right\| }
	{A_{3,n}^{\theta} }
	\right)
	\\
	\leq&
	\frac{\tilde{C}_{1,Q} }{\varepsilon_{Q}^{2} N }
	\leq
	\frac{\tilde{C}_{3,Q} }{N}.
	\end{aligned}
\end{align}
Since $A_{3,n}^{\theta}$, $B_{3,n}^{\theta}$, $C_{3,n}^{\theta}$
are measurable with respect to $\boldsymbol Y$,
(\ref{l2.5.51.b}), (\ref{l2.5.53.b}), (\ref{l2.5.55.b}), (\ref{l2.5.151}) yield
\begin{align}
	&\label{l2.5.161}
	\begin{aligned}[b]
	\left\|
	E\left(
	\frac{A_{2,n}^{\theta}B_{3,n}^{\theta} }
	{\left|A_{3,n}^{\theta} \right|^{2} }
	\right)
	\right\|
	\leq&
	E\left(
	\frac{\left|E\left(\left.
	A_{2,n}^{\theta} \right|\boldsymbol Y \right) \right|
	\left\|B_{3,n}^{\theta} \right\| }
	{\left|A_{3,n}^{\theta} \right|^{2} }
	\right)
	\\
	\leq&
	\frac{\tilde{C}_{1,Q}K_{1,Q} }{\varepsilon_{Q}^{3} N }
	\leq
	\frac{\tilde{C}_{3,Q} }{N},
	\end{aligned}
	\\
	&\label{l2.5.163}
	\begin{aligned}[b]
	\left\|
	E\left(
	\frac{A_{2,n}^{\theta}C_{3,n}^{\theta} }
	{\left|A_{3,n}^{\theta} \right|^{2} }
	\right)
	\right\|
	\leq&
	E\left(
	\frac{\left|E\left(\left.
	A_{2,n}^{\theta} \right|\boldsymbol Y \right) \right|
	\left\|C_{3,n}^{\theta} \right\| }
	{\left|A_{3,n}^{\theta} \right|^{2} }
	\right)
	\\
	\leq&
	\frac{\tilde{C}_{1,Q}C_{1,Q} }{\varepsilon_{Q}^{4} N }
	\leq
	\frac{\tilde{C}_{3,Q} }{N}.
	\end{aligned}
\end{align}

Let $\tilde{C}_{4,Q}=\sqrt{d}\varepsilon_{Q}^{-4}\tilde{C}_{2,Q}K_{1,Q}$,
$\tilde{C}_{5,Q}=\varepsilon_{Q}^{-2}\tilde{C}_{1,Q}\tilde{C}_{4,Q}K_{1,Q}$.
Relying on Assumptions \ref{a3}, \ref{a4} and (\ref{l2.5.755.a}), (\ref{l2.5.755.b}), 
we deduce
\begin{align}
	&\label{l2.5.251}
	\begin{aligned}[b]
	E\left(\left. |A_{2,n}^{\theta} |^{2} \right|\boldsymbol Y\right)
	=&
	E\left(\left.
	\left|
	\int q_{\theta}(Y_{n}|x)
	\tilde{\xi}_{n}^{\theta}(dx)
	\right|^{2}
	\right|\boldsymbol Y
	\right)
	\\
	\leq&
	\frac{\tilde{C}_{2,Q}}{\varepsilon_{Q}^{2} N },
	\end{aligned}
	\\
	&\label{l2.5.253}
	\begin{aligned}[b]
	E\left(\left. \|B_{2,n}^{\theta} \|^{2} \right|\boldsymbol Y\right)
	=&
	E\left(\left.
	\left\|
	\int \nabla_{\theta} q_{\theta}(Y_{n}|x)
	\tilde{\xi}_{n}^{\theta}(dx)
	\right\|^{2}
	\right|\boldsymbol Y
	\right)
	\\
	\leq&
	\frac{d \tilde{C}_{2,Q}K_{1,Q}^{2} }{N},
	\end{aligned}
	\\
	&\label{l2.5.255}
	\begin{aligned}[b]
	E\left(\left. \|C_{2,n}^{\theta} \|^{2} \right|\boldsymbol Y\right)
	=&
	E\left(\left.
	\left\|
	\int q_{\theta}(Y_{n}|x)
	\tilde{\zeta}_{n}^{\theta}(dx)
	\right\|^{2}
	\right|\boldsymbol Y
	\right)
	\\
	\leq&
	\frac{\tilde{C}_{2,Q}}{\varepsilon_{Q}^{2} N }
	\end{aligned}
\end{align}
almost surely.
Then, H\"{o}lder inequality and
(\ref{l2.5.51.b}), (\ref{l2.5.251}) -- (\ref{l2.5.255}) imply
\begin{align}
	&\label{l2.5.257}
	\begin{aligned}[b]
	\left\|
	E\left(
	\frac{A_{2,n}^{\theta}B_{2,n}^{\theta} }
	{\left|	A_{3,n}^{\theta} \right|^{2} }
	\right)
	\right\|
	\leq&
	\left(
	E\left(
	\frac{E\left(\left.
	|A_{2,n}^{\theta}|^{2}\right|\boldsymbol Y \right) }
	{\left|	A_{3,n}^{\theta} \right|^{2} }
	\right)
	\right)^{1/2}
	\\
	&\cdot 
	\left(
	E\left(
	\frac{E\left(\left.
	\|B_{2,n}^{\theta}\|^{2}\right|\boldsymbol Y \right) }
	{\left|	A_{3,n}^{\theta} \right|^{2} }
	\right)
	\right)^{1/2}
	\\
	\leq&
	\frac{\sqrt{d}\tilde{C}_{2,Q}K_{1,Q} }{\varepsilon_{Q}^{3} N }
	\leq
	\frac{\tilde{C}_{4,Q}}{N},
	\end{aligned}
	\\
	&\label{l2.5.259}
	\begin{aligned}[b]
	\left\|
	E\left(
	\frac{A_{2,n}^{\theta}C_{2,n}^{\theta} }
	{\left|	A_{3,n}^{\theta} \right|^{2} }
	\right)
	\right\|
	\leq&
	\left(
	E\left(
	\frac{E\left(\left.
	|A_{2,n}^{\theta}|^{2}\right|\boldsymbol Y \right) }
	{\left|	A_{3,n}^{\theta} \right|^{2} }
	\right)
	\right)^{1/2}
	\\
	&\cdot 
	\left(
	E\left(
	\frac{E\left(\left.
	\|C_{2,n}^{\theta}\|^{2}\right|\boldsymbol Y \right) }
	{\left|	A_{3,n}^{\theta} \right|^{2} }
	\right)
	\right)^{1/2}
	\\
	\leq&
	\frac{\tilde{C}_{2,Q}}{\varepsilon_{Q}^{4} N }
	\leq
	\frac{\tilde{C}_{4,Q}}{N}.
	\end{aligned}
\end{align}
Moreover, due to (\ref{l2.5.51.b}), (\ref{l2.5.251}), we have
\begin{align}\label{l2.5.261}
	E\left(
	\left|
	\frac{A_{2,n}^{\theta} }{A_{3,n}^{\theta}  }
	\right|^{2}
	\right)
	=
	E\left(
	\frac{E\left(\left.
	|A_{2,n}^{\theta}|^{2}\right|\boldsymbol Y \right) }
	{\left|	A_{3,n}^{\theta} \right|^{2} }
	\right)
	\leq
	\frac{\tilde{C}_{2,Q}}{\varepsilon_{Q}^{4} N}
	\leq
	\frac{\tilde{C}_{4,Q}}{N}.
\end{align}
Owing to (\ref{l2.5.51.a}), (\ref{l2.5.53.a}), (\ref{l2.5.55.a}), (\ref{l2.5.261}), we also have
\begin{align}
	&\label{l2.5.353}
	\begin{aligned}[b]
	&
	\left\|
	E\left(
	\frac{B_{1,n}^{\theta} }{A_{1,n}^{\theta} }
	\left|
	\frac{A_{2,n}^{\theta} }{A_{3,n}^{\theta} }
	\right|^{2}
	\right)
	\right\|
	\leq
	E\left(
	\frac{\left\|B_{1,n}^{\theta} \right\|}{A_{1,n}^{\theta} }
	\left|
	\frac{A_{2,n}^{\theta} }{A_{3,n}^{\theta} }
	\right|^{2}
	\right)
	\\
	&\leq
	\frac{K_{1,Q}}{\varepsilon_{Q} }
	E\left(
	\left|
	\frac{A_{2,n}^{\theta} }{A_{3,n}^{\theta} }
	\right|^{2}
	\right)
	\leq
	\frac{\tilde{C}_{4,Q}K_{1,Q}}{\varepsilon_{Q} N}
	\leq
	\frac{\tilde{C}_{5,Q} }{N},
	\end{aligned}
	\\
	&\label{l2.5.355}
	\begin{aligned}[b]
	&
	\left\|
	E\left(
	\frac{C_{1,n}^{\theta} }{A_{1,n}^{\theta} }
	\left|
	\frac{A_{2,n}^{\theta} }{A_{3,n}^{\theta} }
	\right|^{2}
	\right)
	\right\|
	\leq
	E\left(
	\frac{\left\|C_{1,n}^{\theta} \right\|}{A_{1,n}^{\theta} }
	\left|
	\frac{A_{2,n}^{\theta} }{A_{3,n}^{\theta} }
	\right|^{2}
	\right)
	\\
	&\leq
	\frac{\tilde{C}_{1,Q}}{\varepsilon_{Q}^{2} }
	E\left(
	\left|
	\frac{A_{2,n}^{\theta} }{A_{3,n}^{\theta} }
	\right|^{2}
	\right)
	\leq
	\frac{\tilde{C}_{1,Q}\tilde{C}_{4,Q}}{\varepsilon_{Q}^{2} N}
	\leq
	\frac{\tilde{C}_{5,Q} }{N}.
	\end{aligned}
\end{align}

Let $M_{Q}=4(\tilde{C}_{3,Q}+\tilde{C}_{4,Q}+\tilde{C}_{5,Q})$.
Then, (\ref{l2.5.7}), (\ref{l2.5.157}), (\ref{l2.5.161}), (\ref{l2.5.257}), (\ref{l2.5.353})
imply
\begin{align*}
	&
	\left\|
	E\left(
	H'_{\theta,Y_{n} }(\xi_{n}^{\theta},\zeta_{n}^{\theta} )
	-
	H'_{\theta,Y_{n} }(F_{\theta,\boldsymbol Y}^{0:n}(\xi_{0}^{\theta} ),
	G_{\theta,\boldsymbol Y}^{0:n}(\xi_{0}^{\theta}, \zeta_{0}^{\theta} ) )
	\right)
	\right\|
	\\
	&\leq
	\frac{2\tilde{C}_{3,Q}+\tilde{C}_{4,Q}+\tilde{C}_{5,Q} }{N}
	\leq
	\frac{M_{Q} }{2N}.
\end{align*}
Similarly, (\ref{l2.5.5}), (\ref{l2.5.159}), (\ref{l2.5.163}), (\ref{l2.5.259}), (\ref{l2.5.355})
yield
\begin{align*}
	&
	\left\|
	E\left(
	H''_{\theta,Y_{n} }(\xi_{n}^{\theta},\zeta_{n}^{\theta} )
	-
	H''_{\theta,Y_{n} }(F_{\theta,\boldsymbol Y}^{0:n}(\xi_{0}^{\theta} ),
	G_{\theta,\boldsymbol Y}^{0:n}(\xi_{0}^{\theta}, \zeta_{0}^{\theta} ) )
	\right)
	\right\|
	\\
	&\leq
	\frac{2\tilde{C}_{3,Q}+\tilde{C}_{4,Q}+\tilde{C}_{5,Q} }{N}
	\leq
	\frac{M_{Q} }{2N}.
\end{align*}
Combining this with (\ref{l2.5.1}), (\ref{l2.5.3}), we get
\begin{align*}
	&
	\left\|
	E\left(
	H(\theta,Z_{n}^{\theta} )
	-
	H_{\theta,Y_{n} }(F_{\theta,\boldsymbol Y}^{0:n}(\xi_{0}^{\theta} ),
	G_{\theta,\boldsymbol Y}^{0:n}(\xi_{0}^{\theta}, \zeta_{0}^{\theta} ) )
	\right)
	\right\|
	\\
	&
	\begin{aligned}
	\leq& 
	\left\|
	E\left(
	H'_{\theta,Y_{n} }(\xi_{n}^{\theta},\zeta_{n}^{\theta} )
	-
	H'_{\theta,Y_{n} }(F_{\theta,\boldsymbol Y}^{0:n}(\xi_{0}^{\theta} ),
	G_{\theta,\boldsymbol Y}^{0:n}(\xi_{0}^{\theta}, \zeta_{0}^{\theta} ) )
	\right)
	\right\|
	\\
	&+
	\left\|
	E\left(
	H''_{\theta,Y_{n} }(\xi_{n}^{\theta},\zeta_{n}^{\theta} )
	-
	H''_{\theta,Y_{n} }(F_{\theta,\boldsymbol Y}^{0:n}(\xi_{0}^{\theta} ),
	G_{\theta,\boldsymbol Y}^{0:n}(\xi_{0}^{\theta}, \zeta_{0}^{\theta} ) )
	\right)
	\right\|
	\end{aligned}
	\\
	&\leq 
	\frac{M_{Q} }{N}.
\end{align*}
Hence, we have
\begin{align*}
	&
	\|h(\theta) - \nabla l(\theta) \|
	\\
	&
	\begin{aligned}
	\leq &
	\left\|
	E\left(	H(\theta,Z_{n}^{\theta} ) \right)
	-
	h(\theta)
	\right\|
	\\
	&+
	\left\|
	E\left(
	H_{\theta,Y_{n} }(F_{\theta,\boldsymbol Y}^{0:n}(\xi_{0}^{\theta} ),
	G_{\theta,\boldsymbol Y}^{0:n}(\xi_{0}^{\theta}, \zeta_{0}^{\theta} ) )
	\right)
	-
	\nabla l(\theta )
	\right\|
	\\
	&+
	\left\|
	E\left(
	H(\theta,Z_{n}^{\theta} )
	-
	H_{\theta,Y_{n} }(F_{\theta,\boldsymbol Y}^{0:n}(\xi_{0}^{\theta} ),
	G_{\theta,\boldsymbol Y}^{0:n}(\xi_{0}^{\theta}, \zeta_{0}^{\theta} ) )
	\right)
	\right\|
	\end{aligned}
	\\
	&
	\begin{aligned}
	\leq &
	\left\|
	E\left(	H(\theta,Z_{n}^{\theta} ) \right)
	-
	h(\theta)
	\right\|
	\\
	&+
	\left\|
	E\left(
	H_{\theta,Y_{n} }(F_{\theta,\boldsymbol Y}^{0:n}(\xi_{0}^{\theta} ),
	G_{\theta,\boldsymbol Y}^{0:n}(\xi_{0}^{\theta}, \zeta_{0}^{\theta} ) )
	\right)
	-
	\nabla l(\theta )
	\right\|
	+
	\frac{M_{Q}}{N}.
	\end{aligned}
\end{align*}
Then, letting $n\rightarrow\infty$ and using Lemmas \ref{lemma1.3}, \ref{lemma2.4}, we conclude
that (\ref{l2.5.1*}) holds.
\end{proof}

\section{Proof of Main Results}\label{section3*}

In this section, we study the Monte Carlo estimation of the log-likelihood rate gradient
$\nabla l(\theta)$ (see Lemma \ref{lemma3.2}).
We also study the analytical properties of $l(\theta)$
(see Lemma \ref{lemma3.3}).
Using these results (together with the results of \cite{tadic&doucet21} -- \cite{tadic&doucet3}), we prove
Theorem \ref{theorem2}.

Throughout the section, the following notation is used.
$\{\zeta_{n} \}_{n\geq 0}$, $\{\eta_{n} \}_{n\geq 0}$ and $\{\xi_{n} \}_{n\geq 0}$
are the stochastic processes defined by
\begin{align*}
	&
	\zeta_{n}
	=
	H(\theta_{n}, Z_{n+1} ) - h(\theta_{n} ),
	\\
	&
	\eta_{n} = h(\theta_{n}) - \nabla l(\theta_{n}),
	\\
	&
	\xi_{n} = \zeta_{n}+\eta_{n}
\end{align*}
for $n\geq 0$
($H(\theta,z)$, $h(\theta)$, $\{Z_{n} \}_{n\geq 0}$
are specified in (\ref{2*.401}), (\ref{2*.403}), (\ref{2*.405})).
Then, using (\ref{2*.9}), it is easy to show that (\ref{3*.1}) holds for each $n\geq 0$.

\begin{remark}
Due to (\ref{3*.1}),
algorithm (\ref{1.1}) -- (\ref{1.5}) is stochastic gradient search which maximizes
log-likelihood rate $l(\theta)$,
while $\{\xi_{n} \}_{n\geq 0}$ can be interpreted as noise
in the (Monte Carlo) estimation of $\nabla l(\theta)$.
We also recall here that $Q$ is any compact set satisfying $Q\subset\Theta$.
\end{remark}

\begin{lemma}\label{lemma3.2}
Let Assumptions \ref{a1} -- \ref{a4} hold.
Then, relations
\begin{align}\label{l3.2.1*}
	\lim_{n\rightarrow\infty}
	\sup_{k\geq n}
	\left\| \sum_{i=n}^{k} \alpha_{i} \zeta_{i} \right\|
	=
	0,
	\;\;\;\;\;
	\limsup_{n\rightarrow\infty} \|\eta_{n} \|
	\leq
	\frac{M_{Q}}{N}
\end{align}
hold almost surely on $\Lambda_{Q}$ ($\Lambda_{Q}$ is defined in (\ref{1.301})).
\end{lemma}

\begin{proof}
Let $\tau_{Q}$ be the stopping time defined in Lemma \ref{lemma3.1}.
Moreover, let $\tilde{\Lambda}_{Q}$ be the event defined by
$\tilde{\Lambda}_{Q}=\bigcap_{n=0}^{\infty} \{\theta_{n}\in Q\}$.
Hence, on $\Lambda_{Q}$, $\theta_{n}\in Q$ for all, but finitely many $n\geq 0$.
Then, using Lemma \ref{lemma2.5}, we conclude that the second part of (\ref{l3.2.1*})
holds almost surely on $\Lambda_{Q}$.

Throughout the rest of the proof, the following notation is used.
${\cal F}_{k}$ is the $\sigma$-algebra defined by
${\cal F}_{k}=\sigma\{\theta_{0},Z_{0},\cdots,\theta_{k}, Z_{k} \}$ for $k\geq 0$.
$n$ is any positive integer,
while $\zeta_{1,n}$, $\zeta_{2,n}$ and $\zeta_{3,n}$ are the random variables defined by
\begin{align*}
	&
	\zeta_{1,n}
	=
	\tilde{H}(\theta_{n},Z_{n+1}) - (\Pi\tilde{H})(\theta_{n},Z_{n}),
	\\
	&
	\zeta_{2,n}
	=
	(\Pi\tilde{H})(\theta_{n},Z_{n})
	-
	(\Pi\tilde{H})(\theta_{n-1},Z_{n}),
	\\
	&
	\zeta_{3,n}
	=
	-(\Pi\tilde{H})(\theta_{n},Z_{n+1}).
\end{align*}
Then, for $l\geq k>1$, it is straightforward to verify
\begin{align}\label{l3.2.7}
	\sum_{i=k}^{l} \alpha_{i}\zeta_{i}
	=&
	\sum_{i=k}^{l} \alpha_{i}\zeta_{1,i}
	+
	\sum_{i=k}^{l} \alpha_{i}\zeta_{2,i}
	+
	\sum_{i=k}^{l} (\alpha_{i}-\alpha_{i+1} ) \zeta_{3,i}
	\nonumber\\
	&+
	\alpha_{l+1}\zeta_{3,l}
	-
	\alpha_{k}\zeta_{3,k-1}.
\end{align}

As a direct consequence of Lemmas \ref{lemma2.4}, \ref{lemma3.1}, we have
\begin{align*}
	\|\zeta_{1,n}\| I_{ \{\tau_{Q}>n \} }
	\leq&
	C_{5,Q}(2+\|W_{n}\Lambda\| + \|W_{n+1}\Lambda\| ) I_{ \{\tau_{Q}>n \} }
	\\
	\leq&
	4C_{5,Q}C_{6,Q}(1+\|W_{0}\Lambda\|).
\end{align*}
Since $W_{0}$ is measurable with respect to ${\cal F}_{0}$,
Assumption \ref{a1} yields
\begin{align}\label{l3.2.1}
	&
	E\left(\left.
	\sum_{n=1}^{\infty} \alpha_{n}^{2} \|\zeta_{1,n} \|^{2} I_{ \{\tau_{Q}>n \} }
	\right|{\cal F}_{0}
	\right)
	\nonumber\\
	&\leq
	16C_{5,Q}^{2}C_{6,Q}^{2} (1+\|W_{0}\Lambda\| )^{2}
	\left( \sum_{n=0}^{\infty} \alpha_{n}^{2} \right)
	<\infty
\end{align}
almost surely.
As $\{\tau_{Q}>n \}\in{\cal F}_{n}$, we also have
\begin{align*}
	&
	E\left(\left. \zeta_{1,n} I_{ \{\tau_{Q}>n \} } \right|{\cal F}_{n} \right)
	\\
	&=
	\left(
	E\left(\left. \tilde{H}(\theta_{n},Z_{n+1}) \right|{\cal F}_{n} \right)
	-
	(\Pi\tilde{H})(\theta_{n},Z_{n})
	\right)
	I_{ \{\tau_{Q}>n \} }
	=
	0
\end{align*}
almost surely.
Then, Doob theorem and (\ref{l3.2.1}) imply that
$\sum_{n=1}^{\infty} \alpha_{n}\zeta_{1,n} I_{ \{\tau_{Q}>n \} }$ is almost surely convergent.
Since $\tilde{\Lambda}_{Q}\subseteq\{\tau_{Q}>n\}$,
$\sum_{n=1}^{\infty} \alpha_{n}\zeta_{1,n}$ converges almost surely on $\tilde{\Lambda}_{Q}$.

Due to Lemmas \ref{lemma2.4}, \ref{lemma3.1} and (\ref{2*.9}), we have
\begin{align*}
	\|\zeta_{2,n}\| I_{ \{\tau_{Q}>n \} }
	\leq &
	C_{5,Q}\|\theta_{n}-\theta_{n-1}\|(1+\|W_{n}\Lambda\| ) I_{ \{\tau_{Q}>n \} }
	\\
	= &
	\begin{aligned}[t]
	&
	C_{5,Q}\alpha_{n-1} \|H(\theta_{n-1},Z_{n})\|
	\\
	&\cdot 
	(1+\|W_{n}\Lambda\| ) I_{ \{\tau_{Q}>n \} }
	\end{aligned}
	\\
	\leq &
	C_{5,Q}^{2}\alpha_{n-1} (1+\|W_{n}\Lambda\|)^{2} I_{ \{\tau_{Q}>n \} }
	\\
	\leq &
	4C_{5,Q}^{2}C_{6,Q}^{2}\alpha_{n-1}(1+\|W_{0}\Lambda\|)^{2}.
\end{align*}
Combining this with Assumption \ref{a1}, we get
\begin{align*}
	&
	\sum_{n=1}^{\infty} \alpha_{n}\|\zeta_{2,n}\| I_{ \{\tau_{Q}>n \} }
	\\
	&\leq 
	4C_{5,Q}^{2}C_{6,Q}^{2}(1+\|W_{0}\Lambda\| )^{2}
	\left(\sum_{n=0}^{\infty} \alpha_{n}\alpha_{n+1} \right)
	\\
	&\leq 
	2C_{5,Q}^{2}C_{6,Q}^{2}(1+\|W_{0}\Lambda\| )^{2}
	\left(\sum_{n=0}^{\infty} \alpha_{n}^{2} \right)
	<
	\infty.
\end{align*}
Hence, $\sum_{n=0}^{\infty} \alpha_{n}\zeta_{2,n} I_{ \{\tau_{Q}>n \} }$ converges almost surely.
Therefore, $\sum_{n=0}^{\infty} \alpha_{n}\zeta_{2,n}$ is convergent almost surely
on $\tilde{\Lambda}_{Q}$.

As a direct consequence of Lemmas \ref{lemma2.4}, \ref{lemma3.1}, we have
\begin{align*}
	\|\zeta_{3,n}\| I_{ \{\tau_{Q}>n \} }
	\leq&
	C_{5,Q}(1+\|W_{n+1}\Lambda\| ) I_{ \{\tau_{Q}>n \} }
	\\
	\leq&
	2C_{5,Q}C_{6,Q}(1+\|W_{0}\Lambda\| ).
\end{align*}
Consequently, Assumption \ref{a1} yields
\begin{align}
	&\label{l3.2.5}
	\sum_{n=1}^{\infty} \alpha_{n+1}^{2} \|\zeta_{3,n}\|^{2} I_{ \{\tau_{Q}>n \} }
	\nonumber\\
	&\leq
	4C_{5,Q}^{2}C_{6,Q}^{2}(1+\|W_{0}\Lambda\| )^{2}
	\left( \sum_{n=0}^{\infty}\alpha_{n}^{2} \right)
	<
	\infty,
	\\
	&\label{l3.2.3}
	\sum_{n=1}^{\infty} |\alpha_{n}-\alpha_{n+1} | \|\zeta_{3,n}\| I_{ \{\tau_{Q}>n \} }
	\nonumber\\
	&\leq
	2C_{5,Q}C_{6,Q} (1+\|W_{0}\Lambda\| )
	\left( \sum_{n=0}^{\infty} |\alpha_{n}-\alpha_{n+1} | \right)
	<
	\infty.
\end{align}
Therefore, we have
\begin{align*}
	\lim_{n\rightarrow\infty} \alpha_{n+1}\zeta_{3,n} I_{ \{\tau_{Q}>n \} }=0
\end{align*}
almost surely.
Hence, $\lim_{n\rightarrow\infty} \alpha_{n+1}\zeta_{3,n}=0$ almost surely on
$\tilde{\Lambda}_{Q}$.
Moreover, due to (\ref{l3.2.3}),
\begin{align*}
	\sum_{n=1}^{\infty} (\alpha_{n}-\alpha_{n+1} ) \zeta_{3,n} I_{ \{\tau_{Q}>n \} }
\end{align*}
is almost surely convergent.
Thus, $\sum_{n=1}^{\infty} (\alpha_{n}-\alpha_{n+1} ) \zeta_{3,n}$ converges almost surely on
$\tilde{\Lambda}_{Q}$.
Since
\linebreak
$\sum_{n=1}^{\infty}\alpha_{n}\zeta_{1,n}$,
$\sum_{n=1}^{\infty}\alpha_{n}\zeta_{2,n}$ are almost surely convergent on $\tilde{\Lambda}_{Q}$,
(\ref{l3.2.7}) implies that $\sum_{n=0}^{\infty} \alpha_{n}\zeta_{n}$ converges almost surely on
$\tilde{\Lambda}_{Q}$, too.
As $Q$ is any compact set in $\Theta$,
we conclude that $\sum_{n=0}^{\infty} \alpha_{n}\zeta_{n}$ is almost surely convergent on
$\{\sup_{n\geq 0}\|\theta_{n}\|<\infty, \inf_{n\geq 0} d(\theta_{n},\Theta^{c}) > 0 \}$.
Consequently, the first part of (\ref{l3.2.1}) holds almost surely on $\Lambda_{Q}$.
\end{proof}

\begin{lemma}\label{lemma3.3}
Let Assumptions \ref{a2} -- \ref{a4} hold.
Then, the following is true:

(i) $l(\theta)$ is well-defined for each $\theta\in\Theta$.
Moreover, $l(\theta)$ is Lipschitz continuously differentiable on $\Theta$.

(ii) If Assumption \ref{a5} also holds,
then $l(\theta)$ is $p$-times differentiable on $\Theta$.

(iii) If Assumption \ref{a6} also holds,
then $l(\theta)$ is real-analytic on $\Theta$.
\end{lemma}

\begin{proof}
(i) See Lemma \ref{lemma1.3}.
(ii) See \cite[Theorem 3.1]{tadic&doucet2}.
(iii) See \cite[Theorem 2.1]{tadic&doucet3}.
\end{proof}

\begin{proof}[\rm\bf Proof of Theorem \ref{theorem2}]
Let $\eta=\limsup_{n\rightarrow\infty} \|\eta_{n} \|$.
Then, Lemma \ref{lemma3.2} yields
$\eta\leq M_{Q}/N$ almost surely on $\Lambda_{Q}$.
Moreover, due to Assumption \ref{a1} and Lemmas \ref{lemma3.2}, \ref{lemma3.3},
Algorithm (\ref{3*.1}) satisfies all conditions which \cite[Theorem 2.1]{tadic&doucet21} is based on.
Consequently, \cite[Theorem 2.1]{tadic&doucet21} implies that there exist
a function $\psi_{Q}(t)$ and real numbers $r_{Q}$, $L_{1,Q}$, $L_{2,Q}$ with the properties
specified in the statement Theorem \ref{theorem2}.
\end{proof}

\refstepcounter{appendixcounter}\label{appendix1}
\section*{Appendix \arabic{appendixcounter}}

In this section, we present results on stochastic matrices which are needed for the proof of
Lemmas \ref{lemma2.3} and \ref{lemma3.1}.
Here, we rely on the following notation.
$\|\cdot\|$ denotes the Euclidean vector norm and Frobenius matrix norm,
while $\|\cdot\|_{1}$ stands for the $l_{1}$ vector norm.
$N\geq 1$ is an integer.
${\cal P}^{N}$ is the set of $N$-dimensional (column) probability vectors,
while ${\cal P}^{N\times N}$ is the set of $N\times N$ (column) stochastic matrices
(i.e., $A\in{\cal P}^{N\times N}$ if and only if
the columns of $A$ are elements of ${\cal P}^{N}$).
$e$ is the $N$-dimensional vector whose all elements are one.
For $1\leq i\leq N$, $e_{i}$ is the $i$-th standard unit vector in $\mathbb{R}^{N}$
(i.e., $e_{i}$ is the element of ${\cal P}^{N}$ whose $i$-th element is one).
$I$ is the $N\times N$ unit matrix.
$\Lambda$ is the matrix defined by
$
	\Lambda = I -ee^{T}/N
$.
For $A\in{\cal P}^{N\times N}$, $\tau(A)$ is the (Dobrushin) ergodicity coefficient, i.e.,
\begin{align*}
	\tau(A)
	=
	\frac{1}{2} \max_{1\leq j,j'\leq N}
	\sum_{i=1}^{N} |A_{i,j} - A_{i,j'} |,
\end{align*}
where $A_{i,j}$ is the $(i,j)$ entry of $A$.

\begin{lemmaappendix}\label{lemmaa2}
(i) If $A\in {\cal P}^{N\times N}$, then we have
\begin{align*}
	\tau(A)
	=
	1
	-
	\min_{1\leq j,j'\leq N} \sum_{i=1}^{N}
	\min\{A_{i,j},A_{i,j'} \},
\end{align*}
where $A_{i,j}$ is the $(i,j)$ entry of $A$.

(ii) If $A\in{\cal P}^{N\times N}$ and $z,z'\in{\cal P}^{N}$, then we have
\begin{align*}
	\|A(z-z')\|_{1}
	\leq
	\tau(A) \|z-z'\|_{1}.
\end{align*}
Moreover, if $A,A'\in{\cal P}^{N\times N}$,
then $\tau(AA')\leq\tau(A)\tau(A')$.
\end{lemmaappendix}

\begin{proof}
(i) See \cite[Definition 15.2.1, Equation (15.9)]{bremaud}.
(ii) See \cite[Theorems 15.2.4, 15.2.5]{bremaud}.
\end{proof}

\begin{lemmaappendix}\label{lemmaa1}
Let $\{A_{n} \}_{n\geq 1}$, $\{B_{n} \}_{n\geq 1}$ and $\{C_{n} \}_{n\geq 1}$
be sequences in ${\cal P}^{N\times N}$.
Moreover, let $a,b,c\in\mathbb{R}^{N}$.
Assume the following:

(i) There exists a real number $\alpha\in(0,1)$ such that
$\min\{A_{n,i,j}, B_{n,i,j}, C_{n,i,j} \} \geq \alpha/N$
for each $1\leq i,j\leq N$, $n\geq 1$,
where $A_{n,i,j}$, $B_{n,i,j}$, $C_{n,i,j}$ are the $(i,j)$ entries of
$A_{n}$, $B_{n}$, $C_{n}$ (respectively).

(ii) $e^{T}a=e^{T}b=e^{T}c=0$.
\\
Then, we have
\begin{align*}
	&
	\|A_{1}\cdots A_{n}\Lambda \|
	\leq
	K\beta^{n},
	\\
	&
	\|A_{1}\cdots A_{n}a \|
	\leq
	K\beta^{n}\|a\|,
	\\
	&
	\|B_{1}\cdots B_{n}b - C_{1}\cdots C_{n}c \|
	\\
	&\leq
	K\beta^{n}
	(\|b\|+\|c\|)
	\sum_{i=1}^{n} \|B_{i}-C_{i}\|
	+
	K\beta^{n}\|b-c\|
\end{align*}
for each $n\geq 1$,
where $\beta=1-\alpha$ and $K=4\beta^{-1}N$.
\end{lemmaappendix}

\begin{proof}
Throughout the proof, the following notation is used.
$n,k,l$ are any integers satisfying $n\geq 1$, $l\geq k\geq 0$.
$\tilde{A}_{k,k}$, $\tilde{B}_{k,k}$, $\tilde{C}_{k,k}$ and
$\tilde{A}_{k,m}$, $\tilde{B}_{k,m}$, $\tilde{C}_{k,m}$ are the matrices
defined by
$\tilde{A}_{k,k}=\tilde{B}_{k,k}=\tilde{C}_{k,k}=I$ and
\begin{align*}
	&\tilde{A}_{k,m} = A_{k+1}\cdots A_{m},
	\\
	&\tilde{B}_{k,m} = B_{k+1}\cdots B_{m},
	\\
	&\tilde{C}_{k,m} = C_{k+m}\cdots C_{m}
\end{align*}
for $m>k\geq 0$.
Then, using Lemma \ref{lemmaa2}, we conclude
\begin{align*}
	\tau\big(\tilde{A}_{k,m}\big)\leq \tau(A_{k+1})\cdots\tau(A_{m}) \leq \beta^{m-k}.
\end{align*}
Relying on the same lemma, we deduce
\begin{align*}
	\tau(A_{n} )
	\!=
	1
	-
	\min_{1\leq j,j'\leq N} \sum_{i=1}^{N}
	\min\{A_{n,i,j}, A_{n,i,j'} \}
	\!\leq
	1-\alpha
	\!=
	\beta.
\end{align*}
Noticing $e_{i},e/N\in{\cal P}^{N}$
and applying Lemma \ref{lemmaa2} again, we get
\begin{align}\label{la1.1.301}
	\left\| \tilde{A}_{k,l} \left(e_{i} - \frac{e}{N} \right) \right\|_{1}
	\leq
	\tau\big(\tilde{A}_{k,l}\big) \left\| e_{i} - \frac{e}{N} \right\|_{1}
	\leq
	2\beta^{l-k}
\end{align}
for $1\leq i\leq N$.
Since $\tilde{A}_{k,l}\left(e_{i}-\frac{e}{N} \right)$ is the $i$-th column of
$\tilde{A}_{k,l}\Lambda$,
(\ref{la1.1.301}) yields
\begin{align}\label{la1.1.1}
	\big\|
	\tilde{A}_{k,l} \Lambda
	\big\|
	\leq
	N^{1/2} \max_{1\leq i\leq N}
	\left\| \tilde{A}_{k,l} \left( e_{i} - \frac{e}{N} \right) \right\|_{1}
	\leq
	2N^{1/2}\beta^{l-k}.
\end{align}
Hence, we get
\begin{align*}
	\|A_{1}\cdots A_{n} \Lambda \|
	=
	\big\|\tilde{A}_{0,n}\Lambda \big\|
	\leq
	2N^{1/2}\beta^{n}
	\leq
	K\beta^{n}.
\end{align*}
Moreover, we have
\begin{align*}
	\tilde{A}_{k,l} \Lambda a
	=
	\tilde{A}_{k,l}a - \frac{\tilde{A}_{k,l}e}{N} e^{T}a
	=
	A_{k,l}a.
\end{align*}
Consequently, (\ref{la1.1.1}) implies
\begin{align}\label{la1.1.3}
	\big\|\tilde{A}_{k,l}a\big\|
	=
	\big\| \tilde{A}_{k,l} \Lambda a \big\|
	\leq
	\big\| \tilde{A}_{k,l} \Lambda \big\| \|a\|
	\leq
	2N^{1/2}\beta^{l-k}\|a\|.
\end{align}
Thus, we get
\begin{align*}
	\|A_{1}\cdots A_{n} a \|
	=
	\big\|\tilde{A}_{0,n}a\big\| \leq 2N^{1/2}\beta^{n}\|a\| \leq K\beta^{n}\|a\|.
\end{align*}

Since $e^{T}B_{n}=e^{T}C_{n}=e^{T}$, we have
\begin{align*}
	\Lambda (B_{n} - C_{n} )
	=
	B_{n} - C_{n}
	-
	\frac{e}{N} (e^{T}B_{n} - e^{T}C_{n} )
	=
	B_{n} - C_{n}.
\end{align*}
Therefore, we get
\begin{align*}
	\tilde{B}_{0,n}b - \tilde{C}_{0,n}c
	=&
	\sum_{i=1}^{n}
	\tilde{B}_{0,i-1} (B_{i} - C_{i} ) \tilde{C}_{i,n} b
	+
	\tilde{C}_{0,n} (b-c)
	\\
	=&
	\sum_{i=1}^{n}
	\tilde{B}_{0,i-1} \Lambda
	(B_{i} - C_{i} ) \tilde{C}_{i,n} b
	+
	\tilde{C}_{0,n} (b-c).
\end{align*}
Then, applying (\ref{la1.1.1}), (\ref{la1.1.3}) to
$\{B_{n} \}_{n\geq 1}$, $\{C_{n} \}_{n\geq 1}$,
$b$, $c$, we get
\begin{align*}
	\big\| \tilde{B}_{0,n}b - \tilde{C}_{0,n}c \big\|
	\leq &
	\sum_{i=1}^{n}
	\big\| \tilde{B}_{0,i-1} \Lambda \big\|
	\|B_{i} - C_{i} \|  \big\|\tilde{C}_{i,n} b \big\|
	\\
	&+
	\big\|\tilde{C}_{0,n} (b-c) \big\|
	\\
	\leq &
	4N\beta^{n-1} \|b\| \sum_{i=1}^{n} \|B_{i}-C_{i}\|
	\\
	&+
	2N^{1/2}\beta^{n} \|b-c\|
	\\
	\leq &
	K\beta^{n} (\|b\|+\|c\| ) \sum_{i=1}^{n} \|B_{i}-C_{i}\| 
	\\
	&+ 
	K\beta^{n}\|b-c\|.
\end{align*}
Hence, we have
\begin{align*}
	\|B_{1}\cdots B_{n} b - C_{1}\cdots C_{n} c \|
	=&
	\big\|\tilde{B}_{0,n}b - \tilde{C}_{0,n}c\big\|
	\\
	\leq&
	K\beta^{n} (\|b\|\!+\!\|c\| ) \sum_{i=1}^{n} \|B_{i}-C_{i}\| 
	\\
	&+ 
	K\beta^{n}\|b-c\|.
\end{align*}
\end{proof}

\refstepcounter{appendixcounter}\label{appendix2}
\section*{Appendix \arabic{appendixcounter}}

In this section, we prove Lemma \ref{lemma1.3}.
We rely on the following notation.
$\tilde{r}_{\theta}(y,x'|x)$
is the function defined for $\theta\in\Theta$, $x,x'\in{\cal X}$, $y\in{\cal Y}$ by
\begin{align}\label{a2.1*}
	\tilde{r}_{\theta}(y,x'|x)
	=
	q_{\theta}(y|x')p_{\theta}(x'|x).
\end{align}
$\tilde{h}_{\theta,y}(x|\xi,\zeta)$ and $\tilde{H}_{\theta,y}(\xi,\zeta)$
are the functions defined for $\xi\in{\cal P}({\cal X})$, $\zeta\in{\cal M}_{s}^{d}({\cal X})$ by
\begin{align}
	&\label{a2.5*.a}
	\tilde{h}_{\theta,y}(x|\xi,\zeta)
	=
	\frac{\int \tilde{r}_{\theta}(y,x|x')\zeta(dx')
	+ \int \nabla_{\theta} \tilde{r}_{\theta}(y,x|x')\xi(dx') }
	{\iint \tilde{r}_{\theta}(y,x''|x')\mu(dx'')\xi(dx') },
	\\
	&\label{a2.5*.b}
	\tilde{H}_{\theta,y}(\xi,\zeta)
	=
	\int \tilde{h}_{\theta,y}(x|\xi,\zeta) \mu(dx),
\end{align}
while
$\tilde{f}_{\theta,y}(x|\xi)$ and $\tilde{g}_{\theta,y}(x|\xi,\zeta)$
are defined as
\begin{align*}
	&
	\tilde{f}_{\theta,y}(x|\xi)
	=
	\frac{\int \tilde{r}_{\theta}(y,x|x')\xi(dx') }
	{\iint \tilde{r}_{\theta}(y,x''|x)\mu(dx'')\xi(dx') },
	\\
	&
	\tilde{g}_{\theta,y}(x|\xi,\zeta)
	=
	\tilde{h}_{\theta,y}(x|\xi,\zeta)
	-
	\tilde{f}_{\theta,y}(x|\xi)
	\tilde{H}_{\theta,y}(\xi,\zeta).
\end{align*}
$\tilde{F}_{\theta,y}(dx|\xi)$ and $\tilde{G}_{\theta,y}(dx|\xi,\zeta)$
are the measures defined for $B\in{\cal B}({\cal X})$ by
\begin{align*}
	&
	\tilde{F}_{\theta,y}(B|\xi)
	=
	\int_{B} \tilde{f}_{\theta,y}(x|\xi) \mu(dx),
	\\
	&
	\tilde{G}_{\theta,y}(B|\xi,\zeta)
	=
	\int_{B} \tilde{g}_{\theta,y}(x|\xi,\zeta) \mu(dx).
\end{align*}
Measures $\tilde{F}_{\theta,y}(dx|\xi)$
and $\tilde{G}_{\theta,y}(dx|\xi,\zeta)$
are also denoted by $\tilde{F}_{\theta,y}(\xi)$
and $\tilde{G}_{\theta,y}(\xi,\zeta)$ (short-hand notation).
$\tilde{r}_{\theta,\boldsymbol y}^{m:n}(x'|x)$
is the function recursively defined by
\begin{align*}
	&
	\tilde{r}_{\theta,\boldsymbol y}^{m:m+1}(x'|x)
	=
	\tilde{r}_{\theta}(y_{m+1},x'|x),
	\\
	&
	\tilde{r}_{\theta,\boldsymbol y}^{m:n+1}(x'|x)
	=
	\int \tilde{r}_{\theta,\boldsymbol y}^{n:n+1}(x'|x'')
	\tilde{r}_{\theta,\boldsymbol y}^{m:n}(x''|x) \mu(dx'')
\end{align*}
for $n>m\geq 0$ and a sequence $\boldsymbol y = \{y_{n} \}_{n\geq 0}$ in ${\cal Y}$.
$\tilde{h}_{\theta,\boldsymbol y}^{m:n}(x|\xi,\zeta)$ and
$\tilde{H}_{\theta,\boldsymbol y}^{m:n}(\xi,\zeta)$
are the functions defined by
\begin{align}
	&
	h_{\theta,\boldsymbol y}^{m:n}(x|\xi,\zeta)
	\!=\!
	\frac{\int r_{\theta,\boldsymbol y}^{m:n}(x|x') \zeta(dx')
	\!+\! \int \nabla_{\theta} r_{\theta,\boldsymbol y}^{m:n}(x|x') \xi(dx')}
	{\iint r_{\theta,\boldsymbol y}^{m:n}(x''|x') \xi(dx')\mu(dx'') },
	\nonumber\\
	&\label{a2.9*.b}
	\tilde{H}_{\theta,\boldsymbol y}^{m:n}(\xi,\zeta)
	\!=\!
	\int \tilde{h}_{\theta,\boldsymbol y}^{m:n}(x|\xi,\zeta) \mu(dx),
\end{align}
while
$\tilde{f}_{\theta,\boldsymbol y}^{m:n}(x|\xi)$ and
$\tilde{g}_{\theta,\boldsymbol y}^{m:n}(x|\xi,\zeta)$
are defined as
\begin{align*}
	&
	\tilde{f}_{\theta,\boldsymbol y}^{m:n}(x|\xi)
	=
	\frac{\int \tilde{r}_{\theta,\boldsymbol y}^{m:n}(x|x') \xi(dx') }
	{\iint \tilde{r}_{\theta,\boldsymbol y}^{m:n}(x''|x') \xi(dx')\mu(dx'') },
	\\
	&
	\tilde{g}_{\theta,\boldsymbol y}^{m:n}(x|\xi,\zeta)
	=
	\tilde{h}_{\theta,\boldsymbol y}^{m:n}(x|\xi,\zeta)
	-
	\tilde{f}_{\theta,\boldsymbol y}^{m:n}(x|\xi)
	\tilde{H}_{\theta,\boldsymbol y}^{m:n}(\xi,\zeta).
\end{align*}
$\tilde{F}_{\theta,\boldsymbol y}^{m:m}(dx|\xi)$,
$\tilde{F}_{\theta,\boldsymbol y}^{m:n}(dx|\xi)$
and
$\tilde{G}_{\theta,\boldsymbol y}^{m:m}(dx|\xi,\zeta)$,
$\tilde{G}_{\theta,\boldsymbol y}^{m:n}(dx|\xi,\zeta)$
are the measures defined by
$\tilde{F}_{\theta,\boldsymbol y}^{m:m}(B|\xi)=\xi(B)$, 
$\tilde{G}_{\theta,\boldsymbol y}^{m:m}(B|\xi,\zeta)=\zeta(B)$ 
and 
\begin{align*}
	&
	\tilde{F}_{\theta,\boldsymbol y}^{m:n}(B|\xi)
	=
	\int_{B} \tilde{f}_{\theta,\boldsymbol y}^{m:n}(x|\xi) \mu(dx),
	\\
	&
	\tilde{G}_{\theta,\boldsymbol y}^{m:n}(B|\xi,\zeta)
	=
	\int_{B} \tilde{g}_{\theta,\boldsymbol y}^{m:n}(x|\xi,\zeta) \mu(dx).
\end{align*}
Measures $\tilde{F}_{\theta,\boldsymbol y}^{m:n}(dx|\xi)$ and
$\tilde{G}_{\theta,\boldsymbol y}^{m:n}(dx|\xi,\zeta)$
are also denoted by $\tilde{F}_{\theta,\boldsymbol y}^{m:n}(\xi)$ and
$\tilde{G}_{\theta,\boldsymbol y}^{m:n}(\xi,\zeta)$
(short-hand notation).
Then, it is easy to show that
$\tilde{F}_{\theta,\boldsymbol y}^{m:n}(\xi)$
and $\tilde{G}_{\theta,\boldsymbol y}^{m:n}(\xi,\zeta)$
are the optimal filter and its gradient,
i.e.,
\begin{align*}
	&
	\tilde{F}_{\theta,\boldsymbol y}^{0:n}(B|\lambda)
	=
	P\left(\left. X_{n}^{\theta,\lambda}\in B \right|Y_{1:n}^{\theta,\lambda}=y_{1:n} \right),
	\\
	&
	\tilde{G}_{\theta,\boldsymbol y}^{0:n}(B|\lambda,\boldsymbol 0)
	=
	\nabla_{\theta} \tilde{F}_{\theta,\boldsymbol y}^{0:n}(B|\lambda)
\end{align*}
for each $\lambda\in{\cal P}({\cal X})$, $n\geq 1$.
Moreover, it is straightforward to verify
\begin{align*}
	&
	\tilde{F}_{\theta,\boldsymbol y}^{m:n+1}(\xi)
	=
	\tilde{F}_{\theta,y_{n+1}}\left( \tilde{F}_{\theta,\boldsymbol y}^{m:n}(\xi) \right),
	\\
	&
	\tilde{G}_{\theta,\boldsymbol y}^{m:n+1}(\xi,\zeta)
	=
	\tilde{G}_{\theta,y_{n+1}}\left( \tilde{F}_{\theta,\boldsymbol y}^{m:n}(\xi),
	\tilde{G}_{\theta,\boldsymbol y}^{m:n}(\xi,\zeta) \right)
\end{align*}
for each $\xi\in{\cal P}({\cal X})$,
$\zeta\in{\cal M}_{s}^{d}({\cal X})$, $n\geq m\geq 0$.

\begin{remark}
We recall here that $Q$ stands for any compact set satisfying $Q\subset\Theta$.
\end{remark}

\begin{lemmaappendix}\label{lemma1.1}
Let Assumptions \ref{a3} and \ref{a4} hold.
Then, there exists a real number $C_{7,Q}\in[1,\infty)$
(independent of $N$ and depending only on $p_{\theta}(x'|x)$, $q_{\theta}(y|x)$)
such that
\begin{align*}
	&
	\big\|\tilde{F}_{\theta,y}(\xi) - \tilde{F}_{\theta',y}(\xi) \big\|
	\leq
	C_{7,Q}\|\theta-\theta'\|,
	\\
	&
	\big\|\tilde{G}_{\theta,y}(\xi,\zeta) - \tilde{G}_{\theta',y}(\xi,\zeta) \big\|
	\leq
	C_{7,Q}\|\theta-\theta'\|(1 + \|\zeta\| ),
	\\
	&
	\big\|\tilde{H}_{\theta,y}(\xi,\zeta) - \tilde{H}_{\theta',y}(\xi',\zeta') \big\|
	\\
	&\leq
	C_{7,Q}(\|\theta-\theta'\| + \|\xi-\xi'\| )(1+ \|\zeta\| )
	+
	C_{7,Q}\|\zeta-\zeta'\|
\end{align*}
for all $\theta,\theta'\in Q$,
$\xi,\xi'\in{\cal P}({\cal X})$,
$\zeta,\zeta'\in{\cal M}_{s}^{d}({\cal X})$.
\end{lemmaappendix}

\begin{proof}
Throughout the proof, the following notation is used.
$\theta$, $\theta'$ are any elements of $Q$.
$x$, $x'$ are any elements of ${\cal X}$,
while $y$ is any element of ${\cal Y}$.
$\xi$, $\xi'$ are any elements of ${\cal P}({\cal X})$,
while $\zeta$, $\zeta'$ are any elements of ${\cal M}_{s}^{d}({\cal X})$.

Let $\tilde{C}_{1,Q}=2\varepsilon_{Q}^{-3}K_{1,Q}(1+\|\mu\|)$
($\varepsilon_{Q}$, $K_{1,Q}$ are specified in Assumptions \ref{a3}, \ref{a4},
while $\mu(dx)$ is defined in Subsection \ref{ssection1.1}).
Owing to Assumption \ref{a3}, we have
\begin{align}\label{l1.1.151}
	\varepsilon_{Q}^{2}
	\leq
	\varepsilon_{Q} p_{\theta}(x'|x)
	\leq
	\tilde{r}_{\theta}(y,x'|x)
	\leq
	\frac{1}{\varepsilon_{Q} } p_{\theta}(x'|x)
	\leq
	\frac{1}{\varepsilon_{Q}^{2} }.
\end{align}
Consequently, we get
\begin{align}
	&\label{l1.1.1.a}
	\varepsilon_{Q}
	\leq
	\int \tilde{r}_{\theta}(y,x'|x) \mu(dx')
	\leq
	\frac{1}{\varepsilon_{Q} },
	\\
	&\label{l1.1.1.b}
	\varepsilon_{Q}^{3} \leq
	\tilde{f}_{\theta,y}(x|\xi)
	\leq
	\frac{1}{\varepsilon_{Q}^{3} }
	\leq
	\tilde{C}_{1,Q}.
\end{align}
Moreover, due to Assumptions \ref{a3}, \ref{a4}, we have
\begin{align}\label{l1.1.153}
	\|\nabla_{\theta} \tilde{r}_{\theta}(x'|y,x)\|
	\leq&
	\|\nabla_{\theta} q_{\theta}(y|x') \| p_{\theta}(x'|x)
	\nonumber\\
	&+
	q_{\theta}(y|x') \|\nabla_{\theta} p_{\theta}(x'|x) \|
	\nonumber\\
	\leq&
	\frac{2K_{1,Q}}{\varepsilon_{Q} }.
\end{align}
Therefore, we get
\begin{align}\label{l1.1.3}
	\big\|\tilde{h}_{\theta,y}(x|\xi,\zeta)\big\|
	\leq&
	\frac{\int \|\nabla_{\theta} \tilde{r}_{\theta}(y,x|x') \| \xi(dx') }
	{\iint \tilde{r}_{\theta}(y,x''|x')\mu(dx'')\xi(dx') }
	\nonumber\\
	&+
	\frac{\int \tilde{r}_{\theta}(y,x|x')|\zeta|(dx') }
	{\iint \tilde{r}_{\theta}(y,x''|x')\mu(dx'')\xi(dx') }
	\nonumber\\
	\leq&
	\frac{2K_{1,Q}}{\varepsilon_{Q}^{3} } (1 + \|\zeta\| )
	\leq
	\tilde{C}_{1,Q} (1+\|\zeta \|).
\end{align}
Hence, we have
\begin{align}\label{l1.1.5}
	\big\|\tilde{H}_{\theta,y}(\xi,\zeta) \big\|
	\leq&
	\int \|\tilde{h}_{\theta,y}(x|\xi,\zeta) \| \mu(dx)
	\nonumber\\
	\leq&
	\frac{2K_{1,Q}\|\mu\| }{\varepsilon_{Q}^{3} } (1 + \|\zeta\| )
	\leq
	\tilde{C}_{1,Q}(1 + \|\zeta\| ).
\end{align}

Let $\tilde{C}_{2,Q}=6\varepsilon_{Q}^{-3}\tilde{C}_{1,Q}K_{1,Q}^{2}(1+\|\mu\| )$.
Due to Assumptions \ref{a3}, \ref{a4}, we have
\begin{align}\label{l1.1.7}
	|\tilde{r}_{\theta}(x'|y,x) \!-\! \tilde{r}_{\theta'}(x'|y,x) |
	\leq &
	|q_{\theta}(y|x') \!-\! q_{\theta'}(y|x') | p_{\theta}(x'|x)
	\nonumber\\
	&
	+
	q_{\theta'}(y|x') |p_{\theta}(x'|x) \!-\! p_{\theta'}(x'|x) |
	\nonumber\\
	\leq &
	\frac{2K_{1,Q}}{\varepsilon_{Q}} \|\theta-\theta' \|.
\end{align}
Owing to Assumptions \ref{a3}, \ref{a4}, we also have
\begin{align}\label{l1.1.9}
	&
	\|\nabla_{\theta} \tilde{r}_{\theta}(x'|y,x)
	- \nabla_{\theta} \tilde{r}_{\theta'}(x'|y,x) \|
	\nonumber\\
	&
	\begin{aligned}
	\leq &
	\|\nabla_{\theta} q_{\theta}(y|x') - \nabla_{\theta} q_{\theta'}(y|x') \|
	p_{\theta}(x'|x)
	\nonumber\\
	&+
	\|\nabla q_{\theta'}(y|x') \| |p_{\theta}(x'|x) - p_{\theta'}(x'|x) |
	\nonumber\\
	&+
	|q_{\theta}(y|x') - q_{\theta'}(y|x') | \|\nabla_{\theta} p_{\theta}(x'|x) \|
	\nonumber\\
	&+
	q_{\theta'}(y|x') \|\nabla_{\theta} p_{\theta}(x'|x) - \nabla_{\theta} p_{\theta'}(x'|x) \|
	\end{aligned}
	\nonumber\\
	&\leq 
	\frac{4K_{1,Q}^{2}}{\varepsilon_{Q}} \|\theta-\theta' \|.
\end{align}
Then, using (\ref{l1.1.1.a}), (\ref{l1.1.1.b}), (\ref{l1.1.7}), we conclude
\begin{align}\label{l1.1.21}
	&
	\big|\tilde{f}_{\theta,y}(x|\xi) - \tilde{f}_{\theta',y}(x|\xi) \big|
	\nonumber\\
	&
	\begin{aligned}
	\leq& 
	\frac{\int |\tilde{r}_{\theta}(y,x|x') - \tilde{r}_{\theta'}(y,x|x') | \xi(dx') }
	{\iint \tilde{r}_{\theta}(y,x''|x') \mu(dx'') \xi(dx') }
	\nonumber\\
	&+\!
	\begin{aligned}[t]
	&
	\frac{\iint |\tilde{r}_{\theta}(y,x''|x') \!-\! \tilde{r}_{\theta'}(y,x''|x') |
	\mu(dx'') \xi(dx') }
	{\iint \tilde{r}_{\theta}(y,x''|x') \mu(dx'') \xi(dx') }
	\nonumber\\
	&\cdot 
	\tilde{f}_{\theta',y}(x|\xi)
	\end{aligned}
	\nonumber\\
	\leq &
	\left(
	\frac{2K_{1,Q}}{\varepsilon_{Q}^{2} }
	+ \frac{2\tilde{C}_{1,Q}K_{1,Q}\|\mu\| }{\varepsilon_{Q}^{2} }
	\right)
	\|\theta-\theta' \|
	\end{aligned}
	\nonumber\\
	&\leq 
	\tilde{C}_{2,Q} \|\theta-\theta'\|.
\end{align}
Similarly, relying on (\ref{l1.1.1.a}), (\ref{l1.1.3}), (\ref{l1.1.7}), (\ref{l1.1.9}),
we deduce
\begin{align}\label{l1.1.23}
	&
	\big\|\tilde{h}_{\theta,y}(x|\xi,\zeta) - \tilde{h}_{\theta',y}(x|\xi,\zeta) \big\|
	\nonumber\\
	&
	\begin{aligned}
	\leq& 
	\frac{\int |\tilde{r}_{\theta}(y,x|x')
	- \tilde{r}_{\theta'}(y,x|x') | \: |\zeta|(dx') }
	{\iint \tilde{r}_{\theta}(y,x''|x') \mu(dx'') \xi(dx') }
	\nonumber\\
	&+
	\frac{\int \|\nabla_{\theta} \tilde{r}_{\theta}(y,x|x')
	- \nabla_{\theta} \tilde{r}_{\theta'}(y,x|x') \| \xi(dx') }
	{\iint \tilde{r}_{\theta}(y,x''|x') \mu(dx'') \xi(dx') }
	\nonumber\\
	&+
	\begin{aligned}[t]
	&
	\frac{\iint |\tilde{r}_{\theta}(y,x''|x') - \tilde{r}_{\theta'}(y,x''|x') |
	\mu(dx'')\xi(dx') }
	{\iint \tilde{r}_{\theta}(y,x''|x') \mu(dx'') \xi(dx') }
	\nonumber\\
	&\cdot 
	\big\|\tilde{h}_{\theta',y}(x|\xi,\zeta) \big\|
	\end{aligned}
	\end{aligned}
	\nonumber\\
	&
	\begin{aligned}
	\leq& 
	\frac{2K_{1,Q}}{\varepsilon_{Q}^{2} }
	\|\theta-\theta'\|
	\|\zeta\|
	+\!
	\frac{4K_{1,Q}^{2}}{\varepsilon_{Q}^{2} }
	\|\theta-\theta'\|
	\nonumber\\
	&+\!
	\frac{2\tilde{C}_{1,Q}K_{1,Q}\|\mu\| }{\varepsilon_{Q}^{2} }
	\|\theta-\theta' \|
	(1 \!+\! \|\zeta\| )
	\end{aligned}
	\nonumber\\
	&\leq 
	\tilde{C}_{2,Q} \|\theta-\theta'\| (1 + \|\zeta\| ).
\end{align}
Moreover, (\ref{l1.1.151}) -- (\ref{l1.1.3}) imply
\begin{align}\label{l1.1.103}
	&
	\big\|\tilde{h}_{\theta,y}(\xi,\zeta) - \tilde{h}_{\theta,y}(\xi',\zeta') \big\|
	\nonumber\\
	&
	\begin{aligned}
	\leq &
	\frac{\int \|\nabla_{\theta} \tilde{r}_{\theta}(y,x|x') \| |\xi-\xi'|(dx') }
	{\iint \tilde{r}_{\theta}(y,x''|x') \mu(dx'') \xi'(dx') }
	\nonumber\\
	&+
	\frac{\int \tilde{r}_{\theta}(y,x|x') |\zeta-\zeta'|(dx') }
	{\iint \tilde{r}_{\theta}(y,x''|x') \mu(dx'') \xi'(dx') }
	\nonumber\\
	&+
	\big\|\tilde{h}_{\theta,y}(\xi,\zeta) \big\|
	\frac{\iint \tilde{r}_{\theta}(y,x''|x') \mu(dx'') |\xi-\xi'|(dx') }
	{\iint \tilde{r}_{\theta}(y,x''|x') \mu(dx'') \xi'(dx') }
	\end{aligned}
	\nonumber\\
	&\leq 
	\frac{2K_{1,Q}}{\varepsilon_{Q}^{2} } \|\xi-\xi'\|
	+
	\frac{1}{\varepsilon_{Q}^{3} } \|\zeta-\zeta'\|
	+
	\frac{\tilde{C}_{1,Q}}{\varepsilon_{Q}^{3}}\|\xi-\xi'\|(1+\|\zeta\| )
	\nonumber\\
	&\leq 
	\tilde{C}_{2,Q} \|\xi-\xi'\|(1+\|\zeta\|)
	+
	\tilde{C}_{2,Q} \|\zeta-\zeta'\|.
\end{align}

Let $\tilde{C}_{3,Q}=2\tilde{C}_{1,Q}\tilde{C}_{2,Q}(1+\|\mu\|)$.
Then, (\ref{l1.1.1.a}), (\ref{l1.1.1.b}), (\ref{l1.1.3}), (\ref{l1.1.21}), (\ref{l1.1.23}) imply
\begin{align}\label{l1.1.27}
	&
	\big\|\tilde{g}_{\theta,y}(x|\xi,\zeta) - \tilde{g}_{\theta',y}(x|\xi,\zeta) \big\|
	\nonumber\\
	&
	\begin{aligned}
	\leq &
	\big\|\tilde{h}_{\theta,y}(x|\xi,\zeta) - \tilde{h}_{\theta',y}(x|\xi,\zeta) \big\|
	\nonumber\\
	&+
	\big|\tilde{f}_{\theta,y}(x|\xi) - \tilde{f}_{\theta',y}(x|\xi) \big|
	\int \big\|\tilde{h}_{\theta,y}(x'|\xi,\zeta) \big\| \mu(dx')
	\nonumber\\
	&+
	\tilde{f}_{\theta',y}(x|\xi)
	\int \big\|\tilde{h}_{\theta,y}(x'|\xi,\zeta) - \tilde{h}_{\theta',y}(x'|\xi,\zeta) \big\|
	\mu(dx')
	\end{aligned}
	\nonumber\\
	&\leq 
	(\tilde{C}_{2,Q} + 2\tilde{C}_{1,Q}\tilde{C}_{2,Q}\|\mu\| )
	\|\theta-\theta'\| (1 + \|\zeta\| )
	\nonumber\\
	&\leq 
	\tilde{C}_{3,Q} \|\theta-\theta'\| (1 + \|\zeta\| ).
\end{align}
Moreover, (\ref{l1.1.23}) yields
\begin{align}\label{l1.1.25}
	&
	\big\|\tilde{H}_{\theta,y}(\xi,\zeta) - \tilde{H}_{\theta',y}(\xi,\zeta) \big\|
	\nonumber\\
	&\leq 
	\int \big\|\tilde{h}_{\theta,y}(x|\xi,\zeta) - \tilde{h}_{\theta',y}(x|\xi,\zeta) \big\|
	\mu(dx)
	\nonumber\\
	&
	\leq 
	\tilde{C}_{2,Q}\|\mu\|\|\theta-\theta'\|(1+\|\zeta\| )
	\leq 
	\tilde{C}_{3,Q}\|\theta-\theta'\|(1+\|\zeta\| ).
\end{align}
Similarly, (\ref{l1.1.103}) implies
\begin{align}\label{l1.1.29}
	&
	\big\|\tilde{H}_{\theta,y}(\xi,\zeta) - \tilde{H}_{\theta,y}(\xi',\zeta') \big\|
	\nonumber\\
	&\leq 
	\int \big\|\tilde{h}_{\theta,y}(x|\xi,\zeta) - \tilde{h}_{\theta,y}(x|\xi',\zeta') \big\|
	\mu(dx)
	\nonumber\\
	&\leq 
	\tilde{C}_{2,Q}\|\mu\| \|\xi-\xi'\|(1+\|\zeta\|)
	+
	\tilde{C}_{2,Q}\|\mu\| \|\zeta-\zeta'\|
	\nonumber\\
	&\leq 
	\tilde{C}_{3,Q} \|\xi-\xi'\|(1+\|\zeta\|)
	+
	\tilde{C}_{3,Q} \|\zeta-\zeta'\|.
\end{align}

Let $C_{7,Q}=\tilde{C}_{3,Q}(1+\|\mu\|)$.
Then, using (\ref{l1.1.21}), we conclude
\begin{align*}
	&
	\big\|\tilde{F}_{\theta,y}(\xi) - \tilde{F}_{\theta',y}(\xi) \big\|
	\nonumber\\
	&\leq
	\int \big|\tilde{f}_{\theta,y}(x|\xi) - \tilde{f}_{\theta',y}(x|\xi) \big| \mu(dx)
	\nonumber\\
	&\leq
	\tilde{C}_{1,Q}\|\mu\| \|\theta-\theta'\|
	\leq
	C_{7,Q} \|\theta-\theta' \|
\end{align*}
(notice that $\tilde{C}_{1,Q}\leq\tilde{C}_{3,Q})$.
Similarly, relying on (\ref{l1.1.27}), we deduce
\begin{align*}
	&
	\big\|\tilde{G}_{\theta,y}(\xi,\zeta) - \tilde{G}_{\theta',y}(\xi,\zeta) \big\|
	\nonumber\\
	&
	\leq 
	\int \big|\tilde{g}_{\theta,y}(x|\xi,\zeta) - \tilde{g}_{\theta',y}(x|\xi,\zeta) \big| \mu(dx)
	\\
	&
	\leq 
	\tilde{C}_{3,Q}\|\mu\| \|\theta-\theta'\| (1 + \|\zeta\| )
	\leq 
	C_{7,Q} \|\theta-\theta' \| (1 + \|\zeta\| ).
\end{align*}
Moreover, combining (\ref{l1.1.25}), (\ref{l1.1.29}), we get
\begin{align*}
	&
	\big\|\tilde{H}_{\theta,y}(\xi,\zeta) - \tilde{H}_{\theta',y}(\xi',\zeta') \big\|
	\nonumber\\
	&
	\leq 
	\big\|\tilde{H}_{\theta,y}(\xi,\zeta) - \tilde{H}_{\theta',y}(\xi,\zeta) \big\|
	\!+\!
	\big\|\tilde{H}_{\theta',y}(\xi,\zeta) - \tilde{H}_{\theta',y}(\xi',\zeta') \big\|
	\\
	&
	\leq 
	\tilde{C}_{3,Q} (\|\theta-\theta'\| + \|\xi-\xi'\| ) (1+\|\zeta\| )
	+
	\tilde{C}_{3,Q} \|\zeta-\zeta' \|
	\\
	&\leq 
	C_{7,Q} (\|\theta-\theta'\| + \|\xi-\xi'\| ) (1+\|\zeta\| )
	+
	C_{7,Q} \|\zeta-\zeta' \|.
\end{align*}
\end{proof}

\begin{lemmaappendix}\label{lemma1.2}
Let Assumptions \ref{a3} and \ref{a4} hold.
Then, the following is true:

(i) There exist real numbers $\rho_{4,Q}\in(0,1)$, $C_{8,Q}\in[1,\infty)$
(independent of $N$ and depending only on $p_{\theta}(x'|x)$, $q_{\theta}(y|x)$)
such that
\begin{align}
	&\label{l1.2.1*}
	\big\|\tilde{G}_{\theta,\boldsymbol y}^{m:n}(\xi,\zeta) \big\|
	\leq
	C_{8,Q}(1 + \|\zeta\| ),
	\\
	&\label{l1.2.3*}
	\big\|\tilde{F}_{\theta,\boldsymbol y}^{m:n}(\xi)
	- \tilde{F}_{\theta,\boldsymbol y}^{m:n}(\xi') \big\|
	\leq
	C_{8,Q} \rho_{4,Q}^{n-m} \|\xi-\xi'\|,
	\\
	&\label{l1.2.5*}
	\big\|\tilde{G}_{\theta,\boldsymbol y}^{m:n}(\xi,\zeta)
	-
	\tilde{G}_{\theta,\boldsymbol y}^{m:n}(\xi',\zeta') \big\|
	\nonumber\\
	&\leq
	C_{8,Q} \rho_{4,Q}^{n-m} \|\xi-\xi'\| (1 + \|\zeta\| )
	+
	C_{8,Q} \rho_{4,Q}^{n-m} \|\zeta-\zeta' \|
\end{align}
for all $\theta\in Q$, $\xi,\xi'\in{\cal P}({\cal X})$,
$\zeta,\zeta'\in{\cal M}_{s}^{d}({\cal X})$,
$n\geq m\geq 0$ and any sequence $\boldsymbol y = \{y_{n} \}_{n\geq 0}$ in ${\cal Y}$.

(ii) There exists a real number $C_{9,Q}\in[1,\infty)$
(independent of $N$ and depending only on $p_{\theta}(x'|x)$, $q_{\theta}(y|x)$)
such that
\begin{align*}
	&
	\big\|\tilde{F}_{\theta,\boldsymbol y}^{m:n}(\xi)
	- \tilde{F}_{\theta',\boldsymbol y}^{m:n}(\xi) \big\|
	\leq
	C_{9,Q} \|\theta-\theta'\|,
	\\
	&
	\big\|\tilde{G}_{\theta,\boldsymbol y}^{m:n}(\xi,\zeta)
	-
	\tilde{G}_{\theta',\boldsymbol y}^{m:n}(\xi,\zeta) \big\|
	\leq
	C_{9,Q} \|\theta-\theta'\| (1 + \|\zeta\| )
\end{align*}
for all $\theta,\theta'\in Q$, $\xi\in{\cal P}({\cal X})$,
$\zeta\in{\cal M}_{s}^{d}({\cal X})$,
$n\geq m\geq 0$ and any sequence $\boldsymbol y = \{y_{n} \}_{n\geq 0}$ in ${\cal Y}$.
\end{lemmaappendix}

\begin{proof}
(i) See \cite[Theorems 3.1, 3.2]{tadic&doucet1} (or \cite[Theorem 2.2]{tadic&doucet2}).

(ii) Throughout this part of the proof, the following notation is used.
$\theta$, $\theta'$ are any elements of $Q$.
$\xi$, $\zeta$ are any elements of ${\cal P}({\cal X})$, ${\cal M}_{s}^{d}({\cal X})$ (respectively).
$\boldsymbol y = \{y_{n} \}_{n\geq 0}$ is any sequence in ${\cal Y}$.
$n,m$ are any integers satisfying $n>m\geq 0$.

It is straightforward to verify
\begin{align}\label{l1.2.1}
	&
	\tilde{F}_{\theta,\boldsymbol y}^{m:n}(\xi) - \tilde{F}_{\theta',\boldsymbol y}^{m:n}(\xi)
	\nonumber\\
	&=
	\sum_{i=m}^{n-1}
	\left(
	\tilde{F}_{\theta,\boldsymbol y}^{i:n}
	\big(\tilde{F}_{\theta',\boldsymbol y}^{m:i}(\xi) \big)
	-
	\tilde{F}_{\theta,\boldsymbol y}^{i+1:n}
	\big(\tilde{F}_{\theta',\boldsymbol y}^{m:i+1}(\xi) \big)
	\right)
	\nonumber\\
	&=
	\sum_{i=m}^{n-1}
	\left(
	\tilde{F}_{\theta,\boldsymbol y}^{i+1:n}
	\big(\tilde{F}_{\theta,y_{i}}
	\big(\tilde{F}_{\theta',\boldsymbol y}^{m:i}(\xi) \big)\big)
	-
	\tilde{F}_{\theta,\boldsymbol y}^{i+1:n}
	\big(\tilde{F}_{\theta',\boldsymbol y}^{m:i+1}(\xi) \big)
	\right).
\end{align}
It is also easy to show
\begin{align}\label{l1.2.3}
	&
	\tilde{G}_{\theta,\boldsymbol y}^{m:n}(\xi,\zeta)
	- \tilde{G}_{\theta',\boldsymbol y}^{m:n}(\xi,\zeta)
	\nonumber\\
	&
	\begin{aligned}
	=\!
	\sum_{i=m}^{n-1}\!
	\Big(
	&
	\tilde{G}_{\theta,\boldsymbol y}^{i:n}
	\big(\tilde{F}_{\theta',\boldsymbol y}^{m:i}(\xi),
	\tilde{G}_{\theta',\boldsymbol y}^{m:i}(\xi,\zeta) \big)
	\nonumber\\
	&
	-
	\tilde{G}_{\theta,\boldsymbol y}^{i+1:n}
	\big(\tilde{F}_{\theta',\boldsymbol y}^{m:i+1}(\xi),
	\tilde{G}_{\theta',\boldsymbol y}^{m:i+1}(\xi,\zeta) \big)
	\Big)
	\end{aligned}
	\nonumber\\
	&
	\begin{aligned}[b]
	=\!
	\sum_{i=m}^{n-1}\!
	\Big(
	&
	\tilde{G}_{\theta,\boldsymbol y}^{i+1:n}
	\big(\tilde{F}_{\theta,y_{i}}\big( \tilde{F}_{\theta',\boldsymbol y}^{m:i}(\xi) \big),
	\tilde{G}_{\theta,y_{i}}\big(\tilde{F}_{\theta',\boldsymbol y}^{m:i}(\xi),
	\tilde{G}_{\theta',\boldsymbol y}^{m:i}(\xi,\zeta) \big)\big)
	\\
	&
	-
	\tilde{G}_{\theta,\boldsymbol y}^{i+1:n}
	\big(\tilde{F}_{\theta',\boldsymbol y}^{m:i+1}(\xi),
	\tilde{G}_{\theta',\boldsymbol y}^{m:i+1}(\xi,\zeta) \big)
	\Big).
	\end{aligned}
\end{align}

Let $C_{9,Q}=4C_{7,Q}C_{8,Q}^{2}(1-\rho_{4,Q} )^{-1}$
($C_{7,Q}$ is specified in Lemma \ref{lemma1.1}).
Relying on Lemma \ref{lemma1.1} and (\ref{l1.2.3*}), (\ref{l1.2.1}), we deduce
\begin{align*}
	&
	\big\| \tilde{F}_{\theta,\boldsymbol y}^{m:n}(\xi)
	- \tilde{F}_{\theta',\boldsymbol y}^{m:n}(\xi) \big\|
	\nonumber\\
	&
	\leq 
	\sum_{i=m}^{n-1}
	\left\|
	\tilde{F}_{\theta,\boldsymbol y}^{i+1:n}
	\big(\tilde{F}_{\theta,y_{i}}\big(\tilde{F}_{\theta',\boldsymbol y}^{m:i}(\xi) \big)\big)
	-
	\tilde{F}_{\theta,\boldsymbol y}^{i+1:n}
	\big(\tilde{F}_{\theta',\boldsymbol y}^{m:i+1}(\xi) \big)
	\right\|
	\\
	&\leq 
	C_{8,Q} \sum_{i=m}^{n-1} \rho_{4,Q}^{n-i-1}
	\left\|
	\tilde{F}_{\theta,y_{i}}\big(\tilde{F}_{\theta',\boldsymbol y}^{m:i}(\xi) \big)
	-
	\tilde{F}_{\theta',y_{i}}\big(\tilde{F}_{\theta',\boldsymbol y}^{m:i}(\xi) \big)
	\right\|
	\\
	&\leq 
	C_{7,Q}C_{8,Q} \|\theta-\theta' \|
	\sum_{i=m}^{n-1} \rho_{4,Q}^{n-i-1}
	\\
	&\leq 
	C_{7,Q}C_{8,Q}(1-\rho_{4,Q} )^{-1} \|\theta-\theta'\|
	\leq 
	C_{9,Q} \|\theta-\theta'\|.
\end{align*}
Similarly, using (\ref{l1.2.5*}), (\ref{l1.2.3}),
we conclude
\begin{align*}
	&
	\big\| \tilde{G}_{\theta,\boldsymbol y}^{m:n}(\xi,\zeta)
	- \tilde{G}_{\theta',\boldsymbol y}^{m:n}(\xi,\zeta) \big\|
	\nonumber\\
	&
	\begin{aligned}
	\leq\! 
	\sum_{i=m}^{n-1}\!
	\Big\|
	&
	\tilde{G}_{\theta,\boldsymbol y}^{i+1:n}
	\big(\tilde{F}_{\theta,y_{i}}\big( \tilde{F}_{\theta',\boldsymbol y}^{m:i}(\xi) \big),
	\tilde{G}_{\theta,y_{i}}\big(\tilde{F}_{\theta',\boldsymbol y}^{m:i}(\xi),
	\tilde{G}_{\theta',\boldsymbol y}^{m:i}(\xi,\zeta) \big)\big)
	\\
	&
	-
	\tilde{G}_{\theta,\boldsymbol y}^{i+1:n}
	\big(\tilde{F}_{\theta',\boldsymbol y}^{m:i+1}(\xi),
	\tilde{G}_{\theta',\boldsymbol y}^{m:i+1}(\xi,\zeta) \big)
	\Big\|
	\end{aligned}
	\\
	&
	\begin{aligned}
	\leq& 
	\begin{aligned}[t]
	C_{8,Q} \sum_{i=m}^{n-1} \rho_{4,Q}^{n-i-1}
	&
	\Big\|
	\tilde{F}_{\theta,y_{i}}\big(\tilde{F}_{\theta',\boldsymbol y}^{m:i}(\xi) \big)
	-
	\tilde{F}_{\theta',y_{i}}\big(\tilde{F}_{\theta',\boldsymbol y}^{m:i}(\xi) \big)
	\Big\|
	\nonumber\\
	&\cdot 
	\Big(
	1	+ \big\| \tilde{G}_{\theta',\boldsymbol y}^{m:i+1}(\xi,\zeta) \big\|
	\Big)
	\end{aligned}
	\\
	&+
	\begin{aligned}[t]
	\!C_{8,Q} \!\sum_{i=m}^{n-i-1} \!\!\rho_{4,Q}^{n-i-1}
	\Big\|
	&
	\tilde{G}_{\theta,y_{i}}\big(\tilde{F}_{\theta',\boldsymbol y}^{m:i}(\xi),
	\tilde{G}_{\theta',\boldsymbol y}^{m:i}(\xi,\zeta) \big)
	\\
	&
	\!-\!
	\tilde{G}_{\theta',y_{i}}\big(\tilde{F}_{\theta',\boldsymbol y}^{m:i}(\xi),
	\tilde{G}_{\theta',\boldsymbol y}^{m:i}(\xi,\zeta) \big)
	\Big\|.
	\end{aligned}
	\end{aligned}
\end{align*}
Consequently, Lemma \ref{lemma1.1} and (\ref{l1.2.1*}) imply
\begin{align*}
	&
	\big\| \tilde{G}_{\theta,\boldsymbol y}^{m:n}(\xi,\zeta)
	- \tilde{G}_{\theta',\boldsymbol y}^{m:n}(\xi,\zeta) \big\|
	\nonumber\\
	&
	\begin{aligned}
	\leq& 
	C_{7,Q}C_{8,Q}\|\theta-\theta'\|
	\sum_{i=m}^{n-1} \rho_{4,Q}^{n-i-1}
	\left(
	1	+ \big\| \tilde{G}_{\theta',\boldsymbol y}^{m:i+1}(\xi,\zeta) \big\|
	\right)
	\\
	&+
	C_{7,Q}C_{8,Q}\|\theta-\theta'\|
	\sum_{i=m}^{n-1} \rho_{4,Q}^{n-i-1}
	\left(
	1	+ \big\| \tilde{G}_{\theta',\boldsymbol y}^{m:i}(\xi,\zeta) \big\|
	\right)
	\end{aligned}
	\\
	&\leq 
	4C_{7,Q}C_{8,Q}^{2}\|\theta-\theta'\|(1+\|\zeta\| ) \sum_{i=m}^{n-1} \rho_{4,Q}^{n-i-1}
	\\
	&\leq 
	4C_{7,Q}C_{8,Q}^{2}(1-\rho_{4,Q} )^{-1} \|\theta-\theta'\|(1+\|\zeta\| )
	\\
	&=
	C_{9,Q} \|\theta-\theta'\|(1+\|\zeta\| ).
\end{align*}
\end{proof}

\begin{proof}[\rm\bf Proof of Lemma \ref{lemma1.3}]
(i) See \cite[Theorem 3.1]{tadic&doucet2}.

(ii) and (iii)
Throughout these parts of the proof, 
$E_{x,y}(\cdot)$ and $E_{X_{0},Y_{0}}(\cdot)$ 
denote the conditional expectations 
$E(\cdot|X_{0}=x,Y_{0}=y)$ and $E(\cdot|X_{0},Y_{0} )$ (respectively). 
Due to \cite[Proposition 7.2]{tadic&doucet2},
there exist real numbers $\beta_{Q}\in(0,1)$, $\tilde{C}_{1,Q}\in[1,\infty)$
(independent of $N$ and depending only on
$p_{\theta}(x'|x)$, $q_{\theta}(y|x)$) such that
\begin{align}\label{l1.3.7}
	&
	\left\|
	E_{x,y}\left(
	\tilde{H}_{\theta,Y_{n} }\big(
	\tilde{F}_{\theta,\boldsymbol Y}^{0:n-1}(\xi),
	\tilde{G}_{\theta,\boldsymbol Y}^{0:n-1}(\xi,\zeta)
	\big)
	-
	\nabla l(\theta)
	\right)
	\right\|
	\nonumber\\
	&\leq
	\tilde{C}_{1,Q}\beta_{Q}^{n} (1 + \|\zeta\| )
\end{align}
for all $\theta\in Q$, $x\in{\cal X}$, $y\in{\cal Y}$, $\xi\in{\cal P}({\cal X})$,
$\zeta\in{\cal M}_{s}^{d}({\cal X})$, $n\geq 1$.

Throughout the proof of (ii), (iii), the following notation is used, too.
$\theta$, $\theta$, $\theta'$ are any elements of $Q$.
$x$, $y$ are any elements of ${\cal X}$, ${\cal Y}$ (respectively),
while $B$, $\xi$, $\zeta$ are any elements of ${\cal B}({\cal X})$, ${\cal P}({\cal X} )$,
${\cal M}_{s}^{d}({\cal X} )$ (respectively).
$\boldsymbol y = \{y_{n} \}_{n\geq 0}$ is any sequence in ${\cal Y}$.
$n$ is any positive integer.
$\tilde{\alpha}_{\theta,y}(dx|\xi)$, $\tilde{\beta}_{\theta,y}(dx|\xi,\zeta)$
are the measures defined by
\begin{align*}
	&
	\tilde{\alpha}_{\theta,y}(B|\xi)
	=
	\frac{\int_{B} q_{\theta}(y|x) \xi(dx) }{\int q_{\theta}(y|x) \xi(dx) },
	\\
	&
	\tilde{\beta}_{\theta,y}(B|\xi,\zeta)
	=
	\frac{\int_{B} q_{\theta}(y|x) \zeta(dx)
	+ \int_{B} \nabla_{\theta} q_{\theta}(y|x) \xi(dx) }
	{\int q_{\theta}(y|x) \xi(dx) }.
\end{align*}

It is straightforward to verify
\begin{align}
	&\label{l1.3.1'}
	\begin{aligned}[b]
	F_{\theta,\boldsymbol y}^{0:n}(B|\xi)
	=
	\begin{aligned}[t]
	\iint 
	&
	I_{B}(x') p_{\theta}(x'|x) \mu(dx')
	\\
	&\cdot 
	\tilde{F}_{\theta,\boldsymbol y}^{0:n-1}\big(dx|\tilde{\alpha}_{\theta,y_{0}}(\xi) \big),
	\end{aligned}
	\end{aligned}
	\\
	&\label{l1.3.1}
	\begin{aligned}[b]
	G_{\theta,\boldsymbol y}^{0:n}(B|\xi,\zeta)
	=&
	\begin{aligned}[t]
	\iint 
	&
	I_{B}(x') \nabla_{\theta} p_{\theta}(x'|x) \mu(dx')
	\\
	&\cdot 
	\tilde{F}_{\theta,\boldsymbol y}^{0:n-1}\big(dx|\tilde{\alpha}_{\theta,y_{0}}(\xi) \big)
	\end{aligned}
	\\
	&+
	\begin{aligned}[t]
	\iint 
	&
	I_{B}(x') p_{\theta}(x'|x) \mu(dx')
	\\
	&\cdot 
	\tilde{G}_{\theta,\boldsymbol y}^{0:n-1}
	\big(dx|\tilde{\alpha}_{\theta,y_{0}}(\xi),\tilde{\beta}_{\theta,y_{0}}(\xi,\zeta) \big).
	\end{aligned}
	\end{aligned}
\end{align}
(for a detailed derivation of (\ref{l1.3.1'}), (\ref{l1.3.1}), see \cite[Lemma SM2.1]{tadic&doucet4}).
Then, using (\ref{1.3*b}), (\ref{a2.9*.b}), we conclude
\begin{align*}
	&
	H_{\theta,y_{n} }\big(
	F_{\theta,\boldsymbol y}^{0:n}(\zeta),
	G_{\theta,\boldsymbol y}^{0:n}(\xi,\zeta)
	\big)
	\\
	&
	=
	\tilde{H}_{\theta,y_{n} }\big(
	\tilde{F}_{\theta,\boldsymbol y}^{0:n-1}
	\big(\tilde{\alpha}_{\theta,y_{0}}(\xi) \big),
	\tilde{G}_{\theta,\boldsymbol y}^{0:n-1}
	\big(\tilde{\alpha}_{\theta,y_{0}}(\xi),\tilde{\beta}_{\theta,y_{0}}(\xi,\zeta) \big)
	\big).
\end{align*}
Hence, we have
\begin{align}\label{l1.3.3}
	&
	E_{x,y}\left(
	H_{\theta,Y_{n} }\big(
	F_{\theta,\boldsymbol Y}^{0:n}(\zeta),
	G_{\theta,\boldsymbol Y}^{0:n}(\xi,\zeta)
	\big)
	\right)
	\nonumber\\
	&
	\begin{aligned}[b]
	=
	E_{x,y}\Big(
	&
	\tilde{H}_{\theta,Y_{n} }\big(
	\tilde{F}_{\theta,\boldsymbol Y}^{0:n-1}
	\big(\tilde{\alpha}_{\theta,y}(\xi) \big),
	\\
	&
	\tilde{G}_{\theta,\boldsymbol Y}^{0:n-1}
	\big(\tilde{\alpha}_{\theta,y}(\xi),\tilde{\beta}_{\theta,y}(\xi,\zeta) \big)
	\big)
	\Big).
	\end{aligned}
\end{align}

Let $\tilde{C}_{2,Q}=\varepsilon_{Q}^{-2}K_{1,Q}$,
$\tilde{C}_{3,Q}=5C_{7,Q}C_{8,Q}C_{9,Q}$,
while $C_{1,Q}=3C_{8,Q}\tilde{C}_{2,Q}(1+\|\mu\|)$
($\varepsilon_{Q}$, $K_{1,Q}$, $C_{7,Q}$, $C_{8,Q}$
are specified in Assumption \ref{a3}, \ref{a4} and Lemmas \ref{lemma1.1}, \ref{lemma1.2},
while $\mu(dx)$ is defined in Subsection \ref{ssection1.1}).
Owing to Assumptions \ref{a3}, \ref{a4}, we have
\begin{align}\label{l1.3.5}
	\big\| \tilde{\beta}_{\theta,y}(\xi,\zeta) \big\|
	\leq&
	\frac{\int q_{\theta}(y|x) |\zeta|(dx) + \int \|\nabla_{\theta} q_{\theta}(y|x) \| \xi(dx) }
	{\int q_{\theta}(y|x) \xi(dx) }
	\nonumber\\
	\leq&
	\frac{1}{\varepsilon_{Q}^{2} } \|\zeta\|
	+
	\frac{K_{1,Q} }{\varepsilon_{Q} }
	\leq
	\tilde{C}_{2,Q} (1 + \|\zeta\| ).
\end{align}
Consequently, Assumption \ref{a4}, Lemma \ref{lemma1.2} and (\ref{l1.3.1}) yield
\begin{align*}
	&
	\big\| G_{\theta,\boldsymbol y}^{0:n}(\xi,\zeta) \big\|
	\\
	&
	\begin{aligned}
	\leq& 
	\iint \| \nabla_{\theta} p_{\theta}(x'|x) \| \mu(dx')
	\tilde{F}_{\theta,\boldsymbol y}^{0:n-1}\big(dx|\tilde{\alpha}_{\theta,y_{0}}(\xi) \big)
	\\
	&+
	\iint p_{\theta}(x'|x) \mu(dx')
	\big| \tilde{G}_{\theta,\boldsymbol y}^{0:n-1} \big|
	\big(dx|\tilde{\alpha}_{\theta,y_{0}}(\xi),\tilde{\beta}_{\theta,y_{0}}(\xi,\zeta) \big)
	\end{aligned}
	\\
	&\leq 
	K_{1,Q} \|\mu \|
	+
	\big\| \tilde{G}_{\theta,\boldsymbol y}^{0:n-1}
	\big(\tilde{\alpha}_{\theta,y_{0}}(\xi),\tilde{\beta}_{\theta,y_{0}}(\xi,\zeta) \big) \big\|
	\\
	&\leq 
	K_{1,Q} \|\mu \|
	+
	C_{8,Q} \left( 1 + \big\| \tilde{\beta}_{\theta,y_{0}}(\xi,\zeta) \big\| \right)
	\\
	&\leq 
	K_{1,Q}\|\mu\| + 2C_{8,Q}\tilde{C}_{2,Q}(1 + \|\zeta\| )
	\leq 
	C_{1,Q} (1 + \|\zeta \| ).
\end{align*}
Hence, (iii) holds.

Combining (\ref{l1.3.7}), (\ref{l1.3.3}), (\ref{l1.3.5}), we get
\begin{align*}
	&
	\left\|
	E_{x,y}\left(
	H_{\theta,Y_{n} }\big(
	F_{\theta,\boldsymbol Y}^{0:n}(\xi),
	G_{\theta,\boldsymbol Y}^{0:n}(\xi,\zeta)
	\big)
	-
	\nabla l(\theta)
	\right)
	\right\|
	\nonumber\\
	&\leq 
	\tilde{C}_{1,Q}\beta_{Q}^{n}
	\left(1 + \big\|\tilde{\beta}_{\theta,y}(\xi,\zeta) \big\| \right)
	\leq 
	2\tilde{C}_{1,Q}\tilde{C}_{2,Q}\beta_{Q}^{n}
	(1 + \|\zeta\| ).
\end{align*}
Therefore, we have
\begin{align*}
	&
	\left\|
	E\left(
	H_{\theta,Y_{n} }\big(
	F_{\theta,\boldsymbol Y}^{0:n}(\xi),
	G_{\theta,\boldsymbol Y}^{0:n}(\xi,\zeta)
	\big)
	\right)
	-
	\nabla l(\theta)
	\right\|
	\nonumber\\
	&\leq 
	E\left(
	\left\|
	E_{X_{0},Y_{0}}\left(
	H_{\theta,Y_{n} }\big(
	F_{\theta,\boldsymbol Y}^{0:n}(\xi),
	G_{\theta,\boldsymbol Y}^{0:n}(\xi,\zeta)
	\big)
	-
	\nabla l(\theta)
	\right)
	\right\|
	\right)
	\\
	&\leq 
	2\tilde{C}_{1,Q}\tilde{C}_{2,Q}\beta_{Q}^{n}
	(1 + \|\zeta\| ).
\end{align*}
Thus, (\ref{l1.3.1*}) holds.

Owing to Lemmas \ref{lemma1.1}, \ref{lemma1.2}, we have
\begin{align*}
	&
	\begin{aligned}[t]
	\Big\|
	&
	\tilde{H}_{\theta,Y_{n} }\big(
	\tilde{F}_{\theta,\boldsymbol Y}^{0:n-1}(\xi),
	\tilde{G}_{\theta,\boldsymbol Y}^{0:n-1}(\xi,\zeta)
	\big)
	\\
	&
	-
	\tilde{H}_{\theta',Y_{n} }\big(
	\tilde{F}_{\theta',\boldsymbol Y}^{0:n-1}(\xi),
	\tilde{G}_{\theta',\boldsymbol Y}^{0:n-1}(\xi,\zeta)
	\big)
	\Big\|
	\end{aligned}
	\\
	&
	\begin{aligned}[t]
	\leq &
	\begin{aligned}[t]
	&
	C_{7,Q}\Big(
	\|\theta-\theta'\|
	+
	\big\| \tilde{F}_{\theta,\boldsymbol Y}^{0:n-1}(\xi)
	- \tilde{F}_{\theta',\boldsymbol Y}^{0:n-1}(\xi) \big\|
	\Big)
	\\
	&\cdot 
	\left( 1 + \big\| \tilde{G}_{\theta,\boldsymbol Y}^{0:n-1}(\xi,\zeta) \big\| \right)
	\end{aligned}
	\\
	&+
	C_{7,Q}
	\big\| \tilde{G}_{\theta,\boldsymbol Y}^{0:n-1}(\xi,\zeta)
	- \tilde{G}_{\theta',\boldsymbol Y}^{0:n-1}(\xi,\zeta) \big\|
	\end{aligned}
	\\
	&\leq
	5C_{7,Q}C_{8,Q}C_{9,Q} \|\theta-\theta'\| (1+\|\zeta\| )
	\\
	&
	\leq
	\tilde{C}_{3,Q}\|\theta-\theta'\| (1+\|\zeta\| ).
\end{align*}
Moreover, due to (\ref{l1.3.7}), we have
\begin{align*}
	&
	\left\|
	E\left(
	\tilde{H}_{\theta,Y_{n} }\big(
	\tilde{F}_{\theta,\boldsymbol Y}^{0:n-1}(\xi),
	\tilde{G}_{\theta,\boldsymbol Y}^{0:n-1}(\xi,\zeta)
	\big)
	\right)
	-
	\nabla l(\theta)
	\right\|
	\\
	&
	\leq\!
	E\!\left(
	\left\|
	E_{X_{0},Y_{0}}\!\left(
	\tilde{H}_{\theta,Y_{n} }\big(
	\tilde{F}_{\theta,\boldsymbol Y}^{0:n-1}(\xi),
	\tilde{G}_{\theta,\boldsymbol Y}^{0:n-1}(\xi,\zeta)
	\big)
	\!-\!
	\nabla l(\theta)
	\right)
	\right\|
	\right)
	\\
	&
	\leq
	\tilde{C}_{1,Q}\beta_{Q}^{n}	(1 + \|\zeta\| ).
\end{align*}
Therefore, we get
\begin{align*}
	&
	\|\nabla l(\theta) - \nabla l(\theta') \|
	\\
	&
	\begin{aligned}
	\leq 
	&
	\begin{aligned}[t]
	E\Big(
	\Big\|
	&
	\tilde{H}_{\theta,Y_{n} }\big(
	\tilde{F}_{\theta,\boldsymbol Y}^{0:n-1}(\xi),
	\tilde{G}_{\theta,\boldsymbol Y}^{0:n-1}(\xi,\zeta)
	\big)
	\\
	&
	-
	\tilde{H}_{\theta',Y_{n} }\big(
	\tilde{F}_{\theta',\boldsymbol Y}^{0:n-1}(\xi),
	\tilde{G}_{\theta',\boldsymbol Y}^{0:n-1}(\xi,\zeta)
	\big)
	\Big\|
	\Big)
	\end{aligned}
	\\
	&+
	\left\|
	E\left(
	\tilde{H}_{\theta,Y_{n} }\big(
	\tilde{F}_{\theta,\boldsymbol Y}^{0:n-1}(\xi),
	\tilde{G}_{\theta,\boldsymbol Y}^{0:n-1}(\xi,\zeta)
	\big)
	\right)
	-
	\nabla l(\theta)
	\right\|
	\\
	&+
	\left\|
	E\left(
	\tilde{H}_{\theta',Y_{n} }\big(
	\tilde{F}_{\theta',\boldsymbol Y}^{0:n-1}(\xi),
	\tilde{G}_{\theta',\boldsymbol Y}^{0:n-1}(\xi,\zeta)
	\big)
	\right)
	-
	\nabla l(\theta')
	\right\|
	\end{aligned}
	\\
	&\leq 
	\tilde{C}_{3,Q} \|\theta-\theta'\| (1+\|\zeta\| )
	+
	2\tilde{C}_{1,Q}\beta_{Q}^{n} (1 + \|\zeta\| ).
\end{align*}
Letting $n\rightarrow\infty$, we deduce
\begin{align*}
	\|\nabla l(\theta) - \nabla l(\theta') \|
	\leq
	\tilde{C}_{3,Q} \|\theta-\theta'\|.
\end{align*}
Since $Q$ is any compact set in $\Theta$,
we deduce that (ii) holds.
\end{proof}

\end{document}